\documentclass[11pt]{article}

\usepackage[a4paper,margin=1in]{geometry}
\usepackage{amsmath,amssymb,amsthm,mathtools,mathrsfs}
\usepackage{enumitem,needspace,booktabs,array}
\usepackage{microtype}
\ifdefined\pdfgentounicode
  \IfFileExists{glyphtounicode.tex}{\input{glyphtounicode}\pdfgentounicode=1}{}
\fi
\usepackage[hidelinks]{hyperref}
\hypersetup{
 pdftitle={On the ero sets of the Bargmann–Fock space},
 pdfauthor={Shengzhao Hou and Pan Ma},
 pdfkeywords={Bargmann--Fock space, zero set, critical density, logarithmic potential, insertion index, coherent states}
}

\newtheorem{theorem}{Theorem}[section]
\newtheorem{proposition}[theorem]{Proposition}
\newtheorem{lemma}[theorem]{Lemma}
\newtheorem{corollary}[theorem]{Corollary}
\theoremstyle{definition}
\newtheorem{definition}[theorem]{Definition}
\newtheorem{example}[theorem]{Example}
\newtheorem{remark}[theorem]{Remark}

\newcommand{\C}{\mathbb C}
\newcommand{\N}{\mathbb N}
\newcommand{\Z}{\mathbb Z}
\newcommand{\F}{\mathcal F^2}
\newcommand{\ord}{\operatorname{ord}}
\newcommand{\supp}{\operatorname{supp}}
\newcommand{\dist}{\operatorname{dist}}
\newcommand{\dd}{\,dA}

\newcommand{\one}{\mathbf 1}
\newcommand{\Ent}{\operatorname{Ent}}
\newcommand{\calP}{\mathcal P_2}
\newcommand{\calB}{\mathcal B_2}
\newcommand{\calH}{\mathcal H}
\newcommand{\calT}{\mathcal T}

\title{On the zero sets of the Bargmann–Fock space}
\author{Shengzhao Hou\thanks{School of Mathematical Sciences, Soochow University, Suzhou 215006, China. Email: \texttt{shou@suda.edu.cn}.}
\and
Pan Ma\thanks{School of Mathematics and Statistics, HNP-LAMA, Central South University, Changsha, Hunan 410083, China. Email: \texttt{pan.ma@csu.edu.cn}.}}
\date{}

\begin{document}
\maketitle

\begin{abstract}
At the critical density, the classical density theory of the Bargmann--Fock
space no longer determines exact zero sets. We show that this failure is
substantial: the complete radial counting function together with any
prescribed finite collection of angular moments does not determine exact
realizability, and the finite-dimensional relation between the dimension of
the vanishing space and the zero-insertion capacity breaks down in infinite
dimension. For arbitrary locally finite multisets, we establish a
variational characterization of exact zero sets in terms of a renormalized
logarithmic potential and relative entropy. This characterization yields
sharp criteria for several classes of critical configurations.
\end{abstract}

\medskip
\noindent\textbf{Keywords.}
Bargmann--Fock space; zero set; critical density; logarithmic potential;
relative entropy; finite modifications; insertion index; coherent states.
\smallskip

\noindent\textbf{2020 Mathematics Subject Classification.}
Primary 30H20; Secondary 30D10, 30D20.

\section{Introduction}\phantomsection
\label{sec:introduction}

The classical Bargmann--Fock space is
\[
 \F=\left\{f\in\mathcal O(\C):
 \|f\|_2^2:=\frac1\pi\int_\C |f(z)|^2e^{-|z|^2}\,dA(z)<\infty\right\},
\]
where $\mathcal O(\C)$ denotes the entire functions and $dA$ denotes planar Lebesgue area measure. A locally finite multiset
$\Lambda\subset\C$ is an \emph{exact zero multiset} (or simply a zero set)
if there is a nonzero $f\in\F$ whose zero divisor, with multiplicities, is
exactly $\Lambda$. We write
\[
 \mathcal I(\Lambda)=\{f\in\F:f\text{ vanishes on }\Lambda
 \text{ with at least the prescribed multiplicities}\}.
\]
Thus $\mathcal I(\Lambda)\neq\{0\}$ only asserts the existence of a Fock
function vanishing on the prescribed points; it does not in general imply
that $\Lambda$ is the complete zero divisor of such a function
\cite{Zhu1993,Zhu2011}.

The Bargmann--Fock space, the Bargmann transform, coherent states, and
Gaussian Gabor systems are closely related
\cite{Bargmann1961,Folland1989,Grochenig2001,Lyubarskii1992,Zhu2012}.
Sampling and interpolation are governed by Beurling-type densities in the
classical and weighted settings
\cite{Seip1992,SeipWallsten1992,OrtegaSeip1998,Lindholm2001,
BorichevDhuezKellay2007,GrochenigHaimiOrtegaRomero2019,
EscuderoHaimiRomero2021}; related reproducing-kernel questions in Fock-type
spaces appear, for example, in \cite{BorichevLyubarskii2010}. The exact
zero-divisor problem is different: one must prescribe the complete divisor,
without additional zeros. The resulting instability under deletion or
addition of finitely many zeros is already visible in Zhu's work
\cite{Zhu1993,Zhu2011}; zero-based extremal problems were studied in
\cite{BeneteauCarswellKouchekian2010}, while related questions on zero sets
and hereditary zero-set properties appear in
\cite{AadiBouyaOmari2018,AadiOmari2023}.

\medskip\noindent\textbf{The critical scale.}
Our normalization is fixed by the Gaussian weight above. The square lattice
\[
 \mathcal L=\sqrt\pi(\Z+i\Z)
\]
has $n_{\mathcal L}(r)=r^2+O(r)$. With
\[
 D^-(\Lambda)=\liminf_{r\to\infty}\inf_{w\in\C}
       \frac{\Lambda(D(w,r))}{\pi r^2},\qquad
 D^+(\Lambda)=\limsup_{r\to\infty}\sup_{w\in\C}
       \frac{\Lambda(D(w,r))}{\pi r^2},
\]
where $D(w,r)=\{z:|z-w|<r\}$ and $D_R:=\{z:|z|\le R\}$, the corresponding Beurling density is
$1/\pi$. This is the critical density for the present normalization.
There is also a direct growth interpretation. If
\[
 n_\Lambda(r)=\rho r^2+o(r^2),
\]
then Jensen integration gives
\[
 J_\Lambda(r)=\frac{\rho}{2}r^2+o(r^2),
\]
whereas the Fock pointwise estimate permits quadratic growth $|z|^2/2$ in
the logarithm of a nonzero Fock function. Thus $\rho=1$ is the scale at
which the leading quadratic contribution of the divisor matches the
Gaussian growth allowed by the space. At this scale the leading density
term no longer decides exact realizability, and lower-order radial and
angular terms become relevant. A centered asymptotic alone does not imply a
uniform Beurling density statement; the critical configurations constructed
below satisfy a disk-discrepancy estimate uniform in the center.

The purpose of the critical results in this paper is
to determine which aspects of this lower-order information are detected by
finite geometric statistics and which require the full renormalized
geometry. The standard critical lattice already exhibits the sensitivity of this
transition. For
\[
\mathcal L=\sqrt{\pi}(\mathbb Z+i\mathbb Z),
\]
deleting two lattice points yields an exact zero multiset, whereas deleting
one point leaves a uniqueness multiset; see \cite[Section~2]{Zhu2011}.
This raises the question of whether such critical behavior can be captured
by finitely many natural geometric statistics.

The first obstruction to a finite-dimensional description of critical
zero-set geometry is the failure of angular statistics to determine exact
realizability. For a discrete multiset $\Lambda$, the $j$-th angular moment
on the circle of radius $r>0$ is defined by
\[
M_j(r;\Lambda)
=
\sum_{\lambda\in\Lambda:\,|\lambda|=r}
\left(\frac{\lambda}{|\lambda|}\right)^j,
\qquad j\ge1,
\]
where the sum is taken with multiplicities.

\medskip
\noindent\textbf{Theorem A (Failure of finite geometric compression at
	critical density).}
For every integer $L\ge0$, every $\eta>0$, and every $R_*>0$, there exist
simple uniformly separated critical configurations $\Lambda_+$ and
$\Lambda_-$ and a bijection $T:\Lambda_+\to\Lambda_-$ such that
\begin{enumerate}[label=\textup{(\roman*)},leftmargin=2.2em]
	\item $\Lambda_+$ and $\Lambda_-$ have identical radial counting functions;
	\item their first $L$ angular moments agree on every circle;
	\item $T$ is the identity on $D_{R_*}$, the displacement satisfies
	\[
	\sup_{\lambda\in\Lambda_+}|T\lambda-\lambda|<\eta,
	\]
	and tends to zero at infinity;
	\item every finite modification of $\Lambda_+$ is an exact Fock zero
	multiset, whereas every finite modification of $\Lambda_-$ is a
	uniqueness multiset.
\end{enumerate}

This is
\hyperref[thm:finite-angular-data-insufficient]{Theorem~\ref*{thm:finite-angular-data-insufficient}}
below. The quantifier dependence on $L$ is essential: for each prescribed
finite collection of angular moments we construct a corresponding pair of
configurations, but we do not assert the existence of a single pair whose
angular moments agree to all orders.

The second phenomenon concerns the rigidity of finite modifications.
In the finite-dimensional setting, the dimension of the vanishing space
controls the number of zeros that can be inserted while preserving exact
realizability. More precisely, by
Lemma~\ref{lem:finite-modifications}, the effect of a finite addition
depends only on its total multiplicity. For an exact zero multiset
$\Lambda$, define the exact insertion index $\iota(\Lambda)$ as the
supremum of the integers $q\ge0$ such that every finite effective multiset
$\nu$ with $|\nu|=q$ satisfies that $\Lambda+\nu$ is again an exact zero
multiset. Zhu's finite-dimensional structure theorem gives
\[
\iota(\Lambda)=\dim\mathcal I(\Lambda)-1
\]
whenever $\mathcal I(\Lambda)$ is finite-dimensional
\cite[Corollaries~7--9 and Theorem~10]{Zhu2011}.

\medskip\noindent\textbf{Theorem B (Breakdown of finite-dimensional insertion
rigidity).}
For every integer $m\ge0$, there exists a simple exact zero multiset
$\Lambda_m$ such that
\[
 \dim\mathcal I(\Lambda_m)=\infty,
 \qquad \iota(\Lambda_m)=m.
\]

This is \hyperref[thm:infinite-dimensional-prescribed-index]{Theorem~\ref*{thm:infinite-dimensional-prescribed-index}}
below. Let
\[
 \mathcal V:=i\sqrt\pi\,\Z\subset\mathcal L.
\]
One may take
\[
 \Lambda_m=\mathcal L\setminus
 \Bigl(\mathcal V\cup\{a_1,\ldots,a_{m+1}\}\Bigr),
\]
where $a_1,\ldots,a_{m+1}$ are distinct points of
$\mathcal L\setminus\mathcal V$. These examples retain the critical
Beurling densities $D^-(\Lambda_m)=D^+(\Lambda_m)=1/\pi$. In particular,
for $m=0$ there is an exact zero multiset maximal under finite additions
whose vanishing space is
infinite-dimensional. The construction is based on
\[
 g_N(z)=\frac{\sigma(z)}{\sinh(\sqrt\pi z)
             \prod_{j=1}^{N}(z-a_j)},\qquad N=m+1,
\]
where $\sigma$ is the square-lattice sigma function. Deleting the lattice
line gives exponential decay in one Euclidean direction, while the
additional deleted points control the remaining polynomial insertion depth.

These two phenomena show that neither finite radial/angular data nor the
dimension of the vanishing space gives a complete description of exact
realizability at the critical scale. The full renormalized logarithmic
potential provides the information used in the following characterization.

For a multiset $\Lambda$, let $n_\Lambda(r)$ be its counting function. Under
the necessary growth condition $n_\Lambda(r)=O(r^2)$, let $P_\Lambda$ be
the normalized genus-two canonical product and let
$\mathcal U^{\mathrm{can}}_\Lambda$ be the circularly smoothed logarithmic
potential defined in Section~\ref{sec:analytic-characterization}. Write
\[
 d\gamma_0(z)=\pi^{-1}e^{-|z|^2}\,dA(z).
\]
For $\mu\ll\gamma_0$, define the relative entropy
\[
 \Ent(\mu\mid\gamma_0)
 :=\int_\C \log\!\left(\frac{d\mu}{d\gamma_0}\right)d\mu,
\]
with the value $+\infty$ when the integral is not finite. Let $\calP$
be the class of compactly supported probability measures $\mu\ll\gamma_0$
with finite relative entropy and satisfying
\[
 \int_\C z\,d\mu(z)=0,\qquad
 \int_\C z^2\,d\mu(z)=0.
\]
Set
\[
 \calH(\Lambda):=
 \sup_{\mu\in\calP}
 \left\{2\int_\C \mathcal U^{\mathrm{can}}_\Lambda\,d\mu
       -\Ent(\mu\mid\gamma_0)\right\}.
\]

\medskip\noindent\textbf{Theorem C (General variational characterization of exact
zero multisets).}
A locally finite multiset $\Lambda$ is an exact zero multiset for $\F$ if and
only if
\[
 n_\Lambda(r)=O(r^2)
 \qquad\text{and}\qquad
 \calH(\Lambda)<\infty.
\]
When these conditions hold, the minimizing quadratic harmonic correction is
unique.

This is \hyperref[thm:complete-zero-set-characterization]{Theorem~\ref*{thm:complete-zero-set-characterization}}
below. Proposition~\ref{thm:finite-blaschke-zero-test} gives, in addition, a
finite-product extremal scheme whose minimum values converge to the least
squared Fock norm of a normalized exact realizer, and whose minimizers
converge locally uniformly to that extremal function.

The square-lattice threshold above also gives a concrete reading of
$\calH(\Lambda)$. If $a\in\mathcal L\setminus\{0\}$, the sigma-function estimate implies
$\sigma(z)/[z(z-a)]\in\F$, while $\mathcal L\setminus\{0\}$ is a uniqueness
multiset \cite[Section~2]{Zhu2011}. Hence the characterization gives
\[
 \calH(\mathcal L\setminus\{0,a\})<\infty,
 \qquad \calH(\mathcal L\setminus\{0\})=\infty.
\]
Thus the entropy need not be evaluated numerically: its finiteness records
whether the renormalized potential is compatible with exact Fock
realizability.

Hadamard factorization describes a realizer, if one exists, as
$Ce^{az+bz^2}P_\Lambda$, while the Gibbs variational formula treats a fixed
pair $(a,b)$; see, for example, \cite{DemboZeitouni1998}. The proof of the
variational characterization passes from minimizing pairs $(a_R,b_R)$ on
disks to a global minimizer when $\calH(\Lambda)<\infty$. The
radius-independent coercive estimate \eqref{dual-eq:coercive} keeps these
pairs bounded in the four real harmonic directions, allowing a compactness
argument to establish attainment of the global minimum. The two complex
moment conditions above are the Euler--Lagrange equations that
remove the linear and quadratic harmonic freedom. Exact reconstruction is a
separate compactness step: normalization prevents a zero limit, and
Hurwitz's theorem excludes additional zeros. The precise statements are
Theorem~\ref{thm:complete-zero-set-characterization},
Proposition~\ref{thm:finite-blaschke-zero-test}, and
Remark~\ref{dual-rem:smoothed-extremum}.

\medskip\noindent\textbf{Further critical criteria and local stability.}
The remaining geometric results make the lower-order potential explicit in
several classes. For the circular arrays of
Theorem~\ref{thm:smooth-angular-profile}, logarithmic and
squared-logarithmic endpoint corrections are determined by the distribution
of directions near the maximum of an angular profile. For lattice
perturbations, Corollary~\ref{thm:lattice-drift-criterion} retains the
cumulative radial drift, including the $\log\log r$ contribution at the
endpoint. Theorem~\ref{thm:general-cluster-isomorphism} and
Corollary~\ref{cor:cluster-invariants} transfer exact realizability,
vanishing-space dimension, and insertion index through bounded local
replacements preserving two complex moments, subject to a local cubic-mass
bound.

\medskip\noindent\textbf{Relation to earlier work.}
The finite-dimensional identity used above follows from Zhu
\cite[Corollaries~7--9 and Theorem~10]{Zhu2011}. Omari's perturbation
conditions use upper and lower logarithmic averages of cumulative radial
displacement \cite[Theorem~1.6]{Omari2020}; our criterion in
Corollary~\ref{thm:lattice-drift-criterion} retains the full drift $B(r)$.
For example, if
\[
 B(r)=(q+1)\log r+\delta\log\log r+O(1),
\]
all values of $\delta$ give the same limit of $B(r)/\log r$, whereas the
zero-set condition after a finite change of net multiplicity $q$ is
$\delta>1/2$. Angular-density and critical-type results in
\cite{KononovaI,KononovaIII} concern the leading growth scale; the angular
criteria below resolve logarithmic and squared-logarithmic terms after the
quadratic scale is fixed. The noncritical alternatives for $d$-regular
sequences in \cite{MitkovskiWick2013} do not determine these endpoints.
Zhu's earlier work records the sensitivity of Fock zero sets under finite
changes \cite{Zhu1993,Zhu2011}. Dependence of zero sets on the Fock exponent
is studied in \cite{AadiBouyaOmari2018}, while \cite{AadiOmari2023}
characterizes hereditary zero sets and studies stability questions for
uniqueness sets of zero excess. Zero-based extremal problems are treated in
\cite{BeneteauCarswellKouchekian2010}. Moment cancellation for logarithmic
potentials is classical: local moment matching appears explicitly in
\cite[Section~3]{Ortega2002}, and balayage with respect to harmonic
polynomials and logarithmic kernels is developed in \cite{KM2020}. Here it yields a norm isomorphism between vanishing spaces
for prescribed configurations and preserves both dimension and exact
insertion index; see Corollary~\ref{cor:cluster-invariants}.

\medskip\noindent\textbf{Organization.}
Section~\ref{sec:analytic-characterization} proves the variational
characterization and extremal reconstruction, and
Section~\ref{sec:finite-stability} develops finite-modification principles.
Section~\ref{sec:infinite-dimensional-insertion} proves
Theorem~\ref{thm:infinite-dimensional-prescribed-index}.
Sections~\ref{sec:critical-lattice} and \ref{sec:circular-arrays} establish
radial and angular criteria; Section~\ref{sec:transfer} proves transfer
results. Section~\ref{sec:counterexamples} contains the construction of
Theorem~\ref{thm:finite-angular-data-insufficient} and related sharpness
examples. Section~\ref{sec:coherent-states} records consequences for
coherent states. Appendix~\ref{app:circular-estimates} collects the proofs
of the auxiliary sampling, growth, and kernel estimates used in
Section~\ref{sec:circular-arrays}.

\section{Variational characterization and extremal reconstruction}
\label{sec:analytic-characterization}
\label{sec:canonical-reconstruction}
\label{sec:finite-data}

\subsection{Preliminaries}
We identify a multiset with its counting measure and assume throughout
that its support is locally finite and its multiplicities are finite.
In sums and products with explicit multiplicity factors or exponents,
each support point is used once; otherwise multiset notation repeats
each point according to its multiplicity. We write $D(w,t)=\{|z-w|<t\}$ and
$D_R=\{|z|\le R\}$. For a finite multiset $\nu$, set
$q_\nu(z)=\prod_w(z-w)^{\nu(\{w\})}$.
By a finite modification we mean
$\Xi=\Lambda-\nu_-+\nu_+$, where $0\le\nu_-\le\Lambda$;
its net multiplicity is $q=|\nu_+|-|\nu_-|$.
We write
$\mathcal I(\Lambda)=\{f\in\F:f\text{ vanishes on }\Lambda\}$,
where vanishing is understood with at least the prescribed multiplicities.
We call $\Lambda$ a \emph{uniqueness multiset} if
$\mathcal I(\Lambda)=\{0\}$.
A nonzero function in $\mathcal I(\Lambda)$ may have additional
zeros. Thus $\mathcal I(\Lambda)\ne\{0\}$ does not in general imply
that $\Lambda$ is a zero set \cite{Zhu1993}.
Constants in estimates may depend on the configuration, but are
independent of the radius and the summation index.

For a multiset $\Lambda$, let $n_\Lambda(r)$ count its points in
$|z|\le r$, and let $m_\lambda$ denote the multiplicity at $\lambda$.
We use the functions
\begin{equation*}
 \begin{split}
 J_\Lambda(r)&=m_0\log r+
       \sum_{0<|\lambda|\le r}m_\lambda\log\frac r{|\lambda|},\\
 D_\Lambda(r)&=J_\Lambda(r)-\frac{r^2}{2}.
 \end{split}
 \end{equation*}

We shall use the pointwise estimate
\begin{equation}
 |f(z)|\le\|f\|_2 e^{|z|^2/2},\qquad f\in\F;
 \label{eq:fock-point-evaluation}
\end{equation}
see \cite{Zhu2012}. For $f(z)=\sum_{j\ge0}c_jz^j$, polar
integration gives
\begin{equation}
 \|f\|_2^2=\sum_{j\ge0}j!|c_j|^2.
 \label{eq:fock-coefficient-norm}
\end{equation}
Cauchy--Schwarz and $\sum |z|^{2j}/j!=e^{|z|^2}$ then give
\eqref{eq:fock-point-evaluation}.
Let $m_0=\ord_0f$ and $c=[z^{m_0}]f\ne0$.
Jensen's formula, applied after removal of the zero at the origin,
yields
\[
 J_{\mathcal Z(f)}(r)
 =\frac1{2\pi}\int_0^{2\pi}\log|f(re^{it})|\,dt-\log|c|
 \le\frac{r^2}{2}+\log\frac{\|f\|_2}{|c|}.
\]
For $R\ge1$, every zero in $D_R$ contributes at least $\log2$
to $J_{\mathcal Z(f)}(2R)$. Thus
$n_{\mathcal Z(f)}(R)\log2\le2R^2+O_f(1)$.
More generally, if a nonzero entire function $f$ vanishes on $\Lambda$, write
$\mathcal Z(f)=\Lambda+\Omega$ and let $c_f$ be its first nonzero
Taylor coefficient. Since $J_\Omega(r)\ge0$ for $r\ge1$, Jensen's
formula and the arithmetic--geometric mean inequality give
\[
 \frac1{2\pi}\int_0^{2\pi}|f(re^{i\theta})|^2\,d\theta
 \ge |c_f|^2e^{2J_\Lambda(r)+2J_\Omega(r)}
 \ge |c_f|^2e^{2J_\Lambda(r)}.
\]
Polar integration therefore yields
\begin{equation}
 \|f\|_2^2\ge2|c_f|^2\int_1^\infty
                         r e^{2D_\Lambda(r)}\,dr,
 \label{eq:jensen-radial-norm-lower}
\end{equation}
with infinite values allowed. Thus convergence is necessary for a zero
set, and divergence implies uniqueness. The radial criteria below give
classes in which this condition is sufficient; Example~\ref{ex:smooth-flat-maxima}
shows that it is not sufficient in general.

For a multiset satisfying
\begin{equation}
 n_\Lambda(r)=O(r^2),\qquad r\longrightarrow\infty,
 \label{eq:quadratic-count}
\end{equation}
the canonical product
\begin{equation*}
 P_\Lambda(z)=z^{m_0}\prod_{\lambda\ne0}
 E_2(z/\lambda)^{m_\lambda},\qquad
 E_2(w)=(1-w)e^{w+w^2/2},
 \end{equation*}
converges locally uniformly. Indeed, partial summation gives
\[
 \begin{gathered}
 \sum_{|\lambda|>1}\frac{m_\lambda}{|\lambda|^3}
 \le3\int_1^\infty\frac{n_\Lambda(t)}{t^4}\,dt<\infty,\\
 \log E_2(w)=-\sum_{j\ge3}\frac{w^j}{j}=O(|w|^3)
 \quad (|w|\le\tfrac12).
 \end{gathered}
\]
The remaining factors are finite in number, so $P_\Lambda$ has zero
multiset $\Lambda$ and $P_\Lambda(z)=z^{m_0}(1+O(z^3))$.
By \eqref{eq:fock-point-evaluation}, every nonzero Fock function has
order at most two. Hadamard factorization then gives
\begin{equation}
 f=Ce^{az+bz^2}P_\Lambda,\qquad C\ne0,\quad a,b\in\C,
 \label{eq:allforms}
\end{equation}
for every nonzero $f\in\F$ with $\mathcal Z(f)=\Lambda$; see
\cite[Chapters~II--III]{Boas1954}, \cite{Levin1980}, and \cite{Zhu1993}.
If a lower genus occurs in the factorization, passage to $E_2$ adds
convergent linear and quadratic terms to the exponent.
The main criterion removes the two harmonic terms by imposing moment
conditions on test measures.

\subsection{The complete potential criterion}
For background on logarithmic potentials and subharmonic functions in the
plane, see \cite{Ransford1995,SaffTotik1997}.

For $h>0$, write
\[
 S_hu(z)=\frac1{2\pi}\int_0^{2\pi}u(z+he^{it})\,dt,
 \qquad \ell(w)=\log\max\{1,|w|\}.
\]
Logarithmic singularities of a nonzero entire or meromorphic function
are integrable on each such circle, and Jensen's formula gives
$S_h\log|z-\lambda|=\log\max\{h,|z-\lambda|\}$.
For a multiset satisfying \eqref{eq:quadratic-count}, put
\begin{equation}
 \begin{split}
 \mathcal U^{\mathrm{can}}_\Lambda(z)={}&m_0\ell(z)\\
 &+\sum_{\lambda\ne0}m_\lambda
 \left\{\ell(z-\lambda)-\log|\lambda|
       +\Re\left(\frac z\lambda+\frac{z^2}{2\lambda^2}\right)\right\}.
 \end{split}
 \label{eq:canonical-smoothed-series}
\end{equation}
The terms inside braces are grouped. If $n_\Lambda(t)\le Ct^2$ for
$t\ge1$, then, for $R\ge1$,
\begin{equation}
 \sum_{|\lambda|>R}\frac{m_\lambda}{|\lambda|^3}
 =-\frac{n_\Lambda(R)}{R^3}
       +3\int_R^\infty\frac{n_\Lambda(t)}{t^4}\,dt
 \le\frac{3C}{R}.
 \label{eq:canonical-tail-count}
\end{equation}
On a fixed compact set the sufficiently distant summands in
\eqref{eq:canonical-smoothed-series} equal
$\log|E_2(z/\lambda)|=O(|\lambda|^{-3})$ uniformly. The tail thus
converges locally uniformly, and the finite remaining terms are
continuous. Termwise circular averaging, justified by the same estimate
and the integrability of the finitely many logarithmic singularities,
gives
\begin{equation}
 \mathcal U^{\mathrm{can}}_\Lambda=S_1\log|P_\Lambda|,
 \qquad \log|P_\Lambda|\le\mathcal U^{\mathrm{can}}_\Lambda.
 \label{eq:canonical-smoothed-potential}
\end{equation}
The inequality is subharmonicity. Set
\begin{equation*}
 E_\Lambda(a,b)=\int_\C
       e^{2\mathcal U^{\mathrm{can}}_\Lambda(z)+2\Re(az+bz^2)}\,d\gamma_0(z).
 \end{equation*}

Let $\calP$ be the class of compactly supported probability measures
$\mu\ll\gamma_0$ such that
\begin{equation}\label{eq:balanced-complex-moments}
 \int_\C z\,d\mu(z)=0,\qquad
 \int_\C z^2\,d\mu(z)=0,
\end{equation}
and whose relative entropy
\[
 \Ent(\mu\mid\gamma_0)
 =\int_\C\log\left(\frac{d\mu}{d\gamma_0}\right)\,d\mu
\]
is finite. For a multiset satisfying \eqref{eq:quadratic-count}, put
\begin{equation*}\calH(\Lambda)=\sup_{\mu\in\calP}
 \left\{2\int_\C \mathcal U^{\mathrm{can}}_\Lambda\,d\mu-\Ent(\mu\mid\gamma_0)\right\}.
\end{equation*}
The integrals of $\mathcal U^{\mathrm{can}}_\Lambda$ are finite because it is continuous and
$\mu$ has compact support. The class $\calP$ is nonempty: the Gaussian
measure conditioned on a centered disk belongs to it.

\begin{theorem}[Complete characterization of Fock zero sets]\label{thm:complete-zero-set-characterization}
For every locally finite multiset $\Lambda$,
\begin{equation*}
 \begin{gathered}
 \Lambda\text{ is a zero set for }\F\\
 \Longleftrightarrow\quad n_\Lambda(r)=O(r^2)
                       \quad\text{and}\quad\calH(\Lambda)<\infty.
 \end{gathered}
 \end{equation*}
 
Under \eqref{eq:quadratic-count}, the following identity holds in the
extended real numbers:
\begin{equation}\label{eq:entropy-zero-duality}
 \exp\calH(\Lambda)=\inf_{a,b\in\C}E_\Lambda(a,b).
\end{equation}
If this value is finite, the infimum is attained at a unique pair
$(a,b)$. The corresponding $P_\Lambda e^{az+bz^2}$ is a Fock function
with zero multiset $\Lambda$.
\end{theorem}

\begin{lemma}[Circular smoothing]
\label{lem:circular-fock-smoothing}
For $h>0$ and every nonzero entire function $g$,
\begin{equation}
 \|g\|_2^2\le\frac1\pi\int_\C
       e^{2S_h\log|g|(z)-|z|^2}\dd(z)
       \le e^{h^2}\|g\|_2^2.
 \label{eq:smoothed-fock-norm}
\end{equation}
The inequalities hold with infinite values allowed.
\end{lemma}
\begin{proof}
The first inequality is subharmonicity. For the second, Jensen's
inequality gives
\[
 e^{2S_h\log|g|(z)-|z|^2}
 =e^{h^2}\exp S_h\log(|g|^2e^{-|\cdot|^2})(z)
 \le e^{h^2}S_h(|g|^2e^{-|\cdot|^2})(z).
\]
If $G(w)=|g(w)|^2e^{-|w|^2}$, translation invariance and Tonelli's
theorem give, with infinite values allowed,
\[
 \int_\C S_hG(z)\dd(z)
 =\frac1{2\pi}\int_0^{2\pi}\int_\C G(z+he^{it})\dd(z)\,dt
 =\int_\C G(w)\dd(w).
\]
This proves the upper bound, also for infinite norms.
\end{proof}

\begin{proof}[Proof of Theorem~\ref{thm:complete-zero-set-characterization}]
For $R\ge1$, write
\[
 E_R(a,b)=\int_{|z|<R}
 e^{2\mathcal U^{\mathrm{can}}_\Lambda(z)+2\Re(az+bz^2)}\,d\gamma_0(z).
\]
This is a continuous strictly convex function of the four real
coordinates of $(a,b)$. Indeed, a nonzero polynomial
$\Re(az+bz^2)$ cannot vanish on an open disk. Moreover, the continuity
of $\mathcal U^{\mathrm{can}}_\Lambda$ gives a constant $c>0$, independent of $R\ge1$, such that
\begin{equation}\label{dual-eq:coercive}
 E_R(a,b)\ge c\left(1+\frac{|a|^2}{2}
                 +\frac{|b+a^2/2|^2}{3}\right).
\end{equation}
To see this, bound $e^{2\mathcal U^{\mathrm{can}}_\Lambda(z)-|z|^2}$ below on $|z|\le1$ and
integrate the first three Taylor coefficients of $e^{az+bz^2}$ over
that disk. Consequently $E_R$ has a unique minimizer $(a_R,b_R)$.
Put $m_R=E_R(a_R,b_R)$.

Differentiation with respect to all four coordinates at the minimum
shows that
\begin{equation}\label{dual-eq:gibbs}
 d\mu_R(z)=m_R^{-1}\one_{\{|z|<R\}}
 e^{2\mathcal U^{\mathrm{can}}_\Lambda(z)+2\Re(a_Rz+b_Rz^2)}\,d\gamma_0(z)
\end{equation}
satisfies \eqref{eq:balanced-complex-moments}. This measure has compact support and
finite entropy. For every $\mu\in\calP$ supported in $|z|\le R$,
\begin{equation*}2\int \mathcal U^{\mathrm{can}}_\Lambda\,d\mu-\Ent(\mu\mid\gamma_0)
 =\log m_R-\Ent(\mu\mid\mu_R)\le\log m_R.
\end{equation*}
The equality follows by inserting the density in
\eqref{dual-eq:gibbs}; the integral of the harmonic polynomial vanishes
by the two moment conditions. For the inequality, write
$\rho=d\mu/d\mu_R$ and apply Jensen to $x\log x$:
$\Ent(\mu\mid\mu_R)=\int\rho\log\rho\,d\mu_R\ge0$.
The density of $\mu_R$ relative to $\gamma_0$ is bounded above and below
on its disk, so the finite-entropy assumption justifies these identities.
Equality holds for $\mu=\mu_R$.
The boundary has $\gamma_0$-measure zero. Since every compactly supported
measure is contained in a sufficiently large disk, it follows that
\begin{equation*}\calH(\Lambda)=\lim_{R\to\infty}\log m_R.
\end{equation*}
The limit exists because $m_R$ is increasing.

Suppose this limit is finite. The bound \eqref{dual-eq:coercive} makes
$(a_R,b_R)$ bounded. Choose $R_j\to\infty$ such that these pairs
converge to $(a_*,b_*)$. For each fixed $T$, continuity on the disk gives
\[
 E_T(a_*,b_*)=\lim_{j\to\infty}E_T(a_{R_j},b_{R_j})
 \le\lim_{R\to\infty}m_R.
\]
Letting $T\to\infty$ yields
$E_\Lambda(a_*,b_*)\le\lim m_R$. Conversely, for every $(a,b)$,
$m_R\le E_R(a,b)\le E_\Lambda(a,b)$. This proves
\eqref{eq:entropy-zero-duality} and attainment. Strict convexity gives
uniqueness. If instead $m_R\to\infty$, the latter inequality shows
$E_\Lambda(a,b)=\infty$ for every pair, so the identity still holds.

If $\Lambda$ is a zero set, the preliminary estimates give
\eqref{eq:quadratic-count} and a factorization \eqref{eq:allforms}.
Lemma~\ref{lem:circular-fock-smoothing} then gives a finite value of $E_\Lambda$.
Conversely, a finite value in \eqref{eq:entropy-zero-duality} supplies a pair
$(a_*,b_*)$ for which Lemma~\ref{lem:circular-fock-smoothing} gives a Fock function.
The exponential factor has no zeros, so all prescribed orders are
unchanged and there are no additional zeros.
\end{proof}

\subsection{Finite products and extremal approximation}

For $R\ge2$, define the normalized finite Blaschke product
\begin{equation*}
 \mathcal B_{\Lambda,R}(z)
 =z^{m_0}\prod_{0<|\lambda|<R}
 \left(\frac{1-z/\lambda}{1-\overline\lambda z/R^2}
 \right)^{m_\lambda},\qquad |z|<R.
 \end{equation*}
It is holomorphic in a neighborhood of $D_R$, has leading coefficient
$[z^{m_0}]\mathcal B_{\Lambda,R}=1$, and satisfies
\begin{equation}
 |\mathcal B_{\Lambda,R}(Re^{it})|=e^{J_\Lambda(R)}.
 \label{eq:finite-blaschke-boundary-modulus}
\end{equation}
Indeed, for $|z|=R$ and $|\lambda|<R$,
\[
 |1-\overline\lambda z/R^2|=|z-\lambda|/R,
 \qquad
 \left|\frac{1-z/\lambda}{1-\overline\lambda z/R^2}\right|
 =\frac R{|\lambda|}.
\]
The zero at the origin contributes $R^{m_0}$. Boundary zeros are
omitted from the product and contribute zero to $J_\Lambda(R)$,
so the identity remains valid. Set
\begin{equation*}
 \delta(R)=\max\{0,R^2/2-J_\Lambda(R)\},\qquad
 d(R)=1+\delta(R),\qquad R\ge2.
 \end{equation*}

For the same multiset, define the normalized extremal value
\begin{equation}
 M(\Lambda)=\inf\bigl\{\|f\|_2^2:
 f\in\F,\ \mathcal Z(f)=\Lambda,\ [z^{m_0}]f=1\bigr\},
 \label{eq:normalized-zero-set-energy}
\end{equation}
with $\inf\varnothing=+\infty$. When finite, this value is at least
$m_0!$ by \eqref{eq:fock-coefficient-norm}.

\Needspace{18\baselineskip}
\begin{proposition}[Finite-data criterion and extremal convergence]
\label{thm:finite-blaschke-zero-test}
Let $\Lambda$ be any locally finite multiset with finite
multiplicities, and let $p\in\{0,1,2\}$.
Choose radii $R_k\ge2k$ such that
\begin{equation}
 d(R_k)\left(\frac{k}{R_k}\right)^{p+1}\le1,
 \qquad k\ge1.
 \label{eq:finite-blaschke-radius-selection}
\end{equation}
Define
\begin{equation}
 E_{k,p}(\Lambda)=
 \min_{a_1,\ldots,a_p\in\C}\frac1\pi\int_{|z|<k}
 \left|\mathcal B_{\Lambda,R_k}(z)
       e^{\sum_{j=1}^p a_jz^j}\right|^2
  e^{-|z|^2}\,dA(z),
 \label{eq:finite-blaschke-energy}
\end{equation}
where, for $p=0$, the exponential is $1$ and no minimization is taken.
The finite-dimensional minima are attained, and
\begin{equation}
 \lim_{k\to\infty}E_{k,p}(\Lambda)=M(\Lambda)
       \quad\text{in }[0,+\infty].
 \label{eq:finite-blaschke-energy-limit}
\end{equation}
In particular,
\begin{equation*}
 \Lambda\text{ is a zero set for }\F
 \quad\Longleftrightarrow\quad
 \lim_{k\to\infty}E_{k,p}(\Lambda)<\infty.
 \end{equation*}
When $M(\Lambda)<\infty$, there is a unique normalized function
$f_\Lambda$ attaining this value. Every choice of minimizing functions
in \eqref{eq:finite-blaschke-energy} converges locally uniformly to
$f_\Lambda$ on $\C$.
The radii in \eqref{eq:finite-blaschke-radius-selection} can be
chosen under the following hypotheses:
\begin{center}
\begin{tabular}{ccl}
$p$ & Parameters & Sufficient radial hypothesis \\
\hline
$2$ & $a_1,a_2\in\C$ & none; take $R_k=2k^3$ \\[2pt]
$1$ & $a_1\in\C$ & $\displaystyle\liminf_{R\to\infty}\delta(R)/R^2=0$ \\[5pt]
$0$ & none & $\displaystyle\liminf_{R\to\infty}\delta(R)/R=0$
\end{tabular}
\end{center}
\end{proposition}

Thus the limit does not depend on the admissible radii or on the
allowed value of $p$. One may equivalently require
$\liminf_k E_{k,p}(\Lambda)<\infty$ or
$\sup_k E_{k,p}(\Lambda)<\infty$. The proposition gives a convergent
extremal approximation; it does not assert monotonicity, a convergence
rate, or a finite stopping rule.

\begin{proof}
Since $J_\Lambda(R)\ge0$ for $R\ge2$, we have $d(R)\le R^2$.
For $p=2$ and $R_k=2k^3$, this gives
$d(R_k)(k/R_k)^3\le k^3/R_k=1/2$, as required.
For $p=1$, the radial hypothesis implies
$\liminf d(R)/R^2=0$, and we may choose $R_k\ge2k$ so that
$d(R_k)/R_k^2\le k^{-2}$.  For $p=0$, we choose
$d(R_k)/R_k\le k^{-1}$.  In all three cases the radii may be taken to increase.

For attainment of the minimum, fix $k$ and write
$\mathcal B_{\Lambda,R_k}(z)=z^{m_0}(1+\beta_1z+\beta_2z^2+\cdots)$.
When $p=2$, the coefficients of degrees $m_0$, $m_0+1$, and
$m_0+2$ in the holomorphic function are
\[
 1,\qquad a_1+\beta_1,\qquad
 a_2+a_1^2/2+\beta_1a_1+\beta_2.
\]
For a holomorphic function $G(z)=\sum_{j\ge0}c_jz^j$ on a
neighborhood of $D_k$, Parseval's identity and monotone convergence give
\[
 \frac1\pi\int_{|z|<k}|G(z)|^2e^{-|z|^2}\,dA(z)
 =\sum_{j\ge0}|c_j|^2 M_j(k),
 \qquad M_j(k)=2\int_0^k r^{2j+1}e^{-r^2}\,dr>0.
\]
An energy bound therefore bounds $a_1$ and then $a_2$. The integral
is continuous in the parameters, since the integrands converge uniformly
on $D_k$ when the parameters converge. Its sublevel sets are compact,
so the minimum is attained. For $p=1$, use the first two coefficients;
for $p=0$, there is nothing to minimize.

Suppose $f\in\F$ has zero multiset $\Lambda$, and normalize it by
$[z^{m_0}]f=1$. After cancellation, $f/\mathcal B_{\Lambda,R}$ is
holomorphic near $D_R$, nonvanishing in its interior, and equal to
one at zero. Let $h_R$ be its logarithm in $|z|<R$, with
$h_R(0)=0$. Equations \eqref{eq:fock-point-evaluation} and
\eqref{eq:finite-blaschke-boundary-modulus} give
\[
 \left|\frac{f(z)}{\mathcal B_{\Lambda,R}(z)}\right|
 \le\exp\!\left(C_f+R^2/2-J_\Lambda(R)
                       \right)
 \le e^{C'_f d(R)},\qquad |z|=R.
\]
The maximum principle gives $\Re h_R\le C'_f d(R)$ in the disk.
Boundary zeros cause no difficulty: the quotient is holomorphic there,
and its logarithm is used only in the interior.
For $t<R$, the nonnegative harmonic function $C'_f d(R)-\Re h_R$
has mean $C'_f d(R)$. Each Fourier coefficient is bounded by this
mean, while its $j$th coefficient is $-t^j[z^j]h_R/2$.
Letting $t\uparrow R$ gives
\[
 |[z^j]h_R|\le\frac{2C'_f d(R)}{R^j},\qquad j\ge1.
\]
Let $T_p h_R$ denote the Taylor polynomial of degree $p$, whose
constant term is zero. For $|z|\le k$ and $R=R_k$, we have
\begin{align}
 |h_R(z)-T_p h_R(z)|
 &\le2C'_f d(R)
       \frac{(|z|/R)^{p+1}}{1-|z|/R}\notag\\
 &\le4C'_f\left(\frac{|z|}{k}\right)^{p+1}.
 \label{eq:finite-blaschke-log-tail}
\end{align}
The coefficients of $T_p h_{R_k}$ give the admissible function
\[
 F_k=\mathcal B_{\Lambda,R_k}e^{T_p h_{R_k}}
     =f\,e^{-(h_{R_k}-T_p h_{R_k})}\quad\text{on }D_k.
\]
On each fixed compact set, \eqref{eq:finite-blaschke-log-tail}
tends to zero, so $F_k\to f$ locally uniformly. On the whole disk
$D_k$, the same estimate gives $|F_k|^2\le e^{8C'_f}|f|^2$.
Extend $|F_k|^2\boldsymbol1_{D_k}$ by zero outside $D_k$.
These functions converge pointwise to $|f|^2$ and are dominated
by the integrable function $e^{8C'_f}|f|^2$ with respect to
$\pi^{-1}e^{-|z|^2}dA$. Dominated convergence yields
\[
 \lim_{k\to\infty}\frac1\pi\int_{D_k}
              |F_k(z)|^2e^{-|z|^2}\,dA(z)=\|f\|_2^2.
\]
Since $F_k$ is admissible for the finite minimum,
\begin{equation}
 \limsup_{k\to\infty}E_{k,p}(\Lambda)\le M(\Lambda)
       \quad\text{if }M(\Lambda)<\infty.
 \label{eq:finite-energy-upper-limit}
\end{equation}
The domination also gives $\sup_k E_{k,p}(\Lambda)<\infty$.

Conversely, suppose $L=\liminf_k E_{k,p}(\Lambda)<\infty$.
Choose minimizing functions $G_j$ at indices $k_j\to\infty$
such that $E_{k_j,p}(\Lambda)\to L$.
For $k_j>M+1$ and $|z|\le M$, the mean-value inequality gives
\[
 |G_j(z)|^2\le\frac1\pi\int_{D(z,1)}|G_j(w)|^2\,dA(w)
 \le e^{(M+1)^2} E_{k_j,p}(\Lambda).
\]
Montel's theorem and a diagonal argument now give a subsequence
converging locally uniformly on $\C$ to an entire function $f$.
Cauchy's formula preserves the coefficients of degrees $0,\ldots,m_0$,
which are $0,\ldots,0,1$. Hence $f\ne0$.

Hurwitz's theorem excludes zeros outside $\Lambda$.
For $\lambda\in\Lambda$, choose $\rho>0$ so that
$|z-\lambda|\le\rho$ contains no other point of its support.
For large $j$, the functions $H_j(z)=G_j(z)/(z-\lambda)^{m_\lambda}$
are holomorphic and nonvanishing there, and
\[
 H_j(z)=\frac1{2\pi i}\int_{|\zeta-\lambda|=\rho}
 \frac{G_j(\zeta)}{(\zeta-\lambda)^{m_\lambda}(\zeta-z)}
 \,d\zeta.
\]
Consequently $H_j$ converges locally uniformly to a holomorphic
function $H$ with $f(z)=(z-\lambda)^{m_\lambda}H(z)$.
Since $f\not\equiv0$, also $H\not\equiv0$, and Hurwitz's
theorem gives $H(\lambda)\ne0$. Thus
$\mathcal Z(f)=\Lambda$, with the prescribed multiplicities.
Finally, for every fixed $M$,
\[
 \frac1\pi\int_{|z|<M}|f(z)|^2e^{-|z|^2}\,dA(z)
 =\lim_{j\to\infty}\frac1\pi
       \int_{|z|<M}|G_j(z)|^2e^{-|z|^2}\,dA(z)
 \le L.
\]
Letting the disk radius $M\to\infty$ gives
$\|f\|_2^2\le L$, and therefore
\[
 M(\Lambda)\le\liminf_{k\to\infty}E_{k,p}(\Lambda).
\]
Together with \eqref{eq:finite-energy-upper-limit}, this proves
\eqref{eq:finite-blaschke-energy-limit} when $M(\Lambda)<\infty$.
If $M(\Lambda)=+\infty$, the same compactness argument excludes
a finite lower limit, so $E_{k,p}(\Lambda)\to+\infty$.

When $M(\Lambda)<\infty$, the function obtained from a minimizing
subsequence satisfies $\|f\|_2^2=M(\Lambda)$, proving attainment.
Suppose $f_1$ and $f_2$ are two normalized minimizers. After canceling
their common zeros, $f_1/f_2$ is an entire nonvanishing function with
value one at zero. It has an entire logarithm $Q$, chosen so that
$Q(0)=0$. The function $g=f_2e^{Q/2}$ has zero multiset $\Lambda$
and leading coefficient one. Moreover,
\[
 M(\Lambda)\le\|g\|_2^2
 =\frac1\pi\int_\C |f_1(z)||f_2(z)|e^{-|z|^2}\,dA(z)
 \le\|f_1\|_2\|f_2\|_2=M(\Lambda).
\]
Equality in Cauchy--Schwarz implies $|f_1|=|f_2|$ almost everywhere,
and hence everywhere by continuity. Their entire quotient has constant
modulus one and equals one at zero. Thus $f_1=f_2$.

Finally, let $G_k$ be any choice of finite minimizing functions.
The bounded energies make this sequence normal on every fixed disk.
Every subsequential local limit has the prescribed zeros and leading
coefficient, by the argument above, and its squared norm is at most
$\lim_k E_{k,p}(\Lambda)=M(\Lambda)$. Uniqueness identifies every
such limit with $f_\Lambda$. If local uniform convergence failed on a
compact set, a subsequence staying a fixed positive distance from
$f_\Lambda$ there would have a further subsequence converging to
$f_\Lambda$, a contradiction. This proves convergence of the entire
sequence.
\end{proof}

\begin{remark}\label{dual-rem:smoothed-extremum}
The minimum in \eqref{eq:entropy-zero-duality} is a smoothed norm extremum.
It need not equal the least squared Fock norm of a function with these
zeros and leading coefficient one. With $M(\Lambda)$ as in \eqref{eq:normalized-zero-set-energy}, Lemma~\ref{lem:circular-fock-smoothing} gives
\[
 M(\Lambda)\le e^{\calH(\Lambda)}\le eM(\Lambda).
\]
The compact-disk minimizers in the proof converge to the unique finite
minimizer, when it exists: every convergent subsequence has that limit,
and \eqref{dual-eq:coercive} gives boundedness.
\end{remark}

Two radial conditions simplify Proposition~\ref{thm:finite-blaschke-zero-test}.
If $n_\Lambda(r)/r^2\to1$, then
\[
 J_\Lambda(r)-\frac{r^2}{2}
 =J_\Lambda(r_0)-\frac{r_0^2}{2}
   +\int_{r_0}^r\frac{n_\Lambda(t)-t^2}{t}\,dt=o(r^2).
\]
Indeed, for each $\varepsilon>0$ the part of the integral beyond
a sufficiently large fixed radius has modulus at most
$\varepsilon r^2/2$, and the remaining part is constant.
The proposition then applies with $p=1$, with no angular hypothesis.
Under the stronger assumption
\[
 \liminf_{R\to\infty}
       \frac{(R^2/2-J_\Lambda(R))_+}{R}=0,
\]
it applies with $p=0$. In particular, $D_\Lambda(R)=O(\log R)$
suffices; only a suitable sequence of radii is needed, however,
and the deficit may be larger elsewhere. Finite modifications
preserve both radial hypotheses, since they change
$J_\Lambda(R)$ by $q\log R+O(1)$. In applying the proposition after
a modification, one forms the energy from the modified multiset.

Even when $p=0$, the products retain the full angular distribution.
Their normalization does not imply convergence of unregularized
reciprocal sums. Nor can radial counts and finitely many angular
moments replace the full products; see
Theorem~\ref{thm:finite-angular-data-insufficient}.
The exact-divisor normalization in the reconstruction problem is
essential; replacing it by vanishing constraints alone would allow
additional zeros.

\section{Finite modifications and exact stability}
\label{sec:finite-stability}

The complete criterion concerns prescribed multiplicities. In this
section we determine how finite changes enter it and characterize the
multisets for which every such change is allowed.

\subsection{Finite additions and deletions}
\label{sec:finite-modifications}

Finite additions and deletions enter the criteria through their net
multiplicity.

\begin{lemma}[Finite modifications]\label{lem:finite-modifications}
\label{lem:finite-change-weight}
Let $\Xi=\Lambda-\nu_-+\nu_+$ be a finite modification of net
multiplicity $q$. Then
\begin{equation*}
 \begin{gathered}
 \Xi\text{ is a zero set for }\F
 \quad\Longleftrightarrow\\
 \exists f\in \mathcal O(\C)\setminus\{0\}:\quad\mathcal Z(f)=\Lambda,
 \quad\int_\C |f(z)|^2(1+|z|^2)^q e^{-|z|^2}\,dA(z)<\infty.
 \end{gathered}
 \end{equation*}
In particular, two finite modifications of the same multiset
with the same net multiplicity are Fock zero sets simultaneously.
\end{lemma}
\begin{proof}
If $\mathcal Z(f)=\Lambda$, then
$g=fq_{\nu_+}/q_{\nu_-}$ is entire because $\nu_-\le\Lambda$,
and $\mathcal Z(g)=\Xi$. Conversely, $\Xi\ge\nu_+$,
so $gq_{\nu_-}/q_{\nu_+}$ is entire and recovers $f$.
This is a bijection, including when the added and deleted multisets
overlap. The two polynomials are monic, and hence
\[
 \frac{q_{\nu_+}(z)}{q_{\nu_-}(z)}
 =z^q(1+O(|z|^{-1}))\qquad (|z|\to\infty).
\]
Hence $|g(z)|^2\asymp |f(z)|^2(1+|z|^2)^q$ outside a fixed
disk. Both integrands are integrable on the disk, so their integrals
are finite simultaneously. The condition depends on the finite changes
only through $q$.
\end{proof}

For all sufficiently large radii, a finite modification gives
\begin{equation}
 D_\Xi(r)=D_\Lambda(r)+q\log r+O(1).
 \label{eq:finite-change-jensen}
\end{equation}
Indeed, for a finite multiset $\nu$ and $r$ beyond its support,
$J_\nu(r)=|\nu|\log r-\sum_{w\ne0}\nu(\{w\})\log|w|$.
Subtracting this identity for $\nu_-$ from the one for $\nu_+$
proves \eqref{eq:finite-change-jensen}, with an error that is constant
for all sufficiently large $r$.
A criterion expressed in terms of $D_\Lambda$ therefore acquires
a factor $r^{2q}$. A criterion expressed in terms of $D_\Xi$
already includes the finite change.

The following condition ensures that a multiset is either a zero set
or a uniqueness multiset.

\begin{corollary}[A Jensen condition]
\label{cor:jensen-subsequence-dichotomy}
Let $\Lambda$ be a locally finite multiset such that
\begin{equation*}
 \liminf_{r\to\infty}
 \frac{(r^2/2-J_\Lambda(r))_+}{\log r}<\infty.
 \end{equation*}
Then
\[
 \Lambda\text{ is a zero set for }\F
 \quad\Longleftrightarrow\quad\mathcal I(\Lambda)\ne\{0\}.
\]
The hypothesis and the equivalence hold after any finite modification.
\end{corollary}
\begin{proof}
Let $0\ne f\in\mathcal I(\Lambda)$ and write
$\mathcal Z(f)=\Lambda+\Omega$. Jensen's formula and
\eqref{eq:fock-point-evaluation} give, for $r\ge2$,
\[
 J_\Omega(r)\le (r^2/2-J_\Lambda(r))_+
                  +C_f.
\]
The right side is $O(\log r)$ along an unbounded sequence.
If $\Omega$ were infinite, any $N$ of its zeros would give
$J_\Omega(r)\ge N\log r-O_N(1)$ for large $r$. Letting $N$
increase would force $J_\Omega(r)/\log r\to\infty$, a contradiction.
Thus $\Omega$ is finite. Division by $q_\Omega$ gives an entire
function $g$ with zero multiset $\Lambda$. Since
$|q_\Omega(z)|\ge c|z|^{|\Omega|}$ outside a fixed disk and $g$
is bounded on that disk, $g\in\F$.
For a finite modification of net multiplicity $q$,
\eqref{eq:finite-change-jensen} gives
\[
 (r^2/2-J_\Xi(r))_+
 \le (r^2/2-J_\Lambda(r))_+ +|q|\log r+C
\]
for all sufficiently large $r$. Hence the hypothesis is preserved,
and the equivalence just proved applies to $\Xi$.
\end{proof}

\subsection{Stability under every finite modification}

\begin{definition}\label{dual-def:stable}
The multiset $\Lambda$ is \emph{stable under finite modifications} if
$\Lambda-\nu_-+\nu_+$ is a zero set for $\F$ for every pair of finite
multisets $\nu_-,\nu_+$ with $\nu_-\le\Lambda$. The added and removed
multiplicities are prescribed.
\end{definition}

\begin{proposition}\label{prop:finite-modification-moments}
A locally finite multiset $\Lambda$ is stable under finite modifications
if and only if, for every integer $q\ge0$, there is an entire function
$f_q$ such that
\begin{equation}\label{dual-eq:moments}
 \mathcal Z(f_q)=\Lambda,\qquad f_q\in\F,\qquad z^qf_q\in\F.
\end{equation}
The function $f_q$ is allowed to depend on $q$.
\end{proposition}

\begin{proof}
Suppose $\Lambda$ is stable. Choose $g_q\in\F$ with
$\mathcal Z(g_q)=\Lambda+q\delta_0$ and set $f_q=g_q/z^q$. The singularity at zero
is removable. The quotient is in $\F$, since division by $z^q$ is
bounded outside the unit disk and the quotient is entire on that disk.
This proves \eqref{dual-eq:moments}.

Conversely, let $A$ and $B$ be the monic polynomials corresponding to
$\nu_+$ and $\nu_-$. Choose $q=\deg A$. The function $Af_q/B$ is entire
and has precisely the modified multiset of zeros. For large $|z|$,
\[
 |A(z)f_q(z)/B(z)|\le C(1+|z|)^q|f_q(z)|.
\]
Its Fock norm is finite by \eqref{dual-eq:moments}; the integral over a compact
disk is finite by holomorphy. This proves stability.
\end{proof}

\subsection{Balanced pointwise inequalities}

Let $\calB$ consist of the finitely supported probability measures
\begin{equation*}\mu=\sum_{\ell=1}^N t_\ell\delta_{z_\ell},\qquad
 t_\ell\ge0,\quad \sum_\ell t_\ell=1,\quad
 \sum_\ell t_\ell z_\ell=\sum_\ell t_\ell z_\ell^2=0.
\end{equation*}
For a continuous real function $W$ on $\C$, write
\[
 \calT(W)=\sup_{\mu\in\calB}\int W\,d\mu.
\]

\begin{proposition}\label{prop:balanced-five-points}
For a continuous real function $W$ and a real number $A$, the following
are equivalent:
\begin{enumerate}[label=\textup{(\roman*)},leftmargin=2.2em]
 \item $\calT(W)\le A$;
 \item there are $a,b\in\C$ such that
 \begin{equation}\label{dual-eq:pointwise-dual}
 W(z)+\Re(az+bz^2)\le A\qquad(z\in\C).
 \end{equation}
\end{enumerate}
In the definition of $\calT(W)$, it is enough to use measures with at
most five support points.
\end{proposition}

\begin{proof}
Integrating \eqref{dual-eq:pointwise-dual} proves (i). For the converse, use
the five points $\xi_j=e^{2\pi ij/5}$, $0\le j<5$. For
$h(z)=\Re(az+bz^2)$, discrete Fourier inversion gives
\begin{equation}\label{dual-eq:guards}
 \sum_{j=0}^4h(\xi_j)=0,\qquad
 a=\frac25\sum_{j=0}^4h(\xi_j)\overline{\xi_j},\qquad
 b=\frac25\sum_{j=0}^4h(\xi_j)\overline{\xi_j}^{\,2}.
\end{equation}
Consequently, upper bounds on the five values $h(\xi_j)$ give lower
bounds as well and bound both coefficients.

Take any finite collection of points and adjoin these five points.
Denote the resulting list by $z_1,\ldots,z_N$. We claim that the finite
system
\[
 h(z_i)\le A-W(z_i),\qquad 1\le i\le N,
\]
is feasible. For completeness, consider the convex cone in $\mathbb R^N$
consisting of vectors $(h(z_i)+s_i)_i$, with $s_i\ge0$. This cone is
closed. Indeed, for a convergent sequence of such vectors, the five
guard coordinates bound the coefficients of $h$ by \eqref{dual-eq:guards};
a subsequence of the coefficients converges, and so do the nonnegative
slacks. If the vector $(A-W(z_i))_i$ were outside the cone, finite
dimensional separation would give a nonzero vector $(p_i)_i$ with
\[
 p_i\ge0,\qquad
 \sum_i p_i z_i=\sum_i p_i z_i^2=0,\qquad
 \sum_i p_i(A-W(z_i))<0.
\]
The first conditions follow respectively from the nonnegative slacks
and the free four real coefficients of $h$. After normalization by
$\sum_i p_i$, this contradicts (i). The finite system is therefore
feasible.

The coefficients satisfying the five guard inequalities form a compact
set, by \eqref{dual-eq:guards}. Every finite family of the remaining closed
half-space constraints has a solution in that set. The finite
intersection property now proves \eqref{dual-eq:pointwise-dual} on all of
$\C$.

It remains to reduce the size of a test. Suppose a balanced measure has
more than five positive weights. The corresponding vectors
\[
 (1,\Re z_i,\Im z_i,\Re z_i^2,\Im z_i^2)\in\mathbb R^5
\]
are linearly dependent. Choose a nonzero real dependence $(c_i)_i$.
Replacing the weights by $t_i+sc_i$ preserves all constraints. Choose
the sign so that the objective $\sum_i(t_i+sc_i)W(z_i)$ does not
decrease, and continue until one weight becomes zero. Such an endpoint
exists because $\sum_i c_i=0$ and the dependence has both signs.
Repeating reduces the support to at most five points without decreasing
the objective. This proves the assertion about the supremum.
\end{proof}

\subsection{A geometric criterion for stable zero sets}

For a multiset satisfying \eqref{eq:quadratic-count}, define
\begin{equation*}\begin{split}
 W_{\Lambda,q}(z)&=\mathcal U^{\mathrm{can}}_\Lambda(z)-\frac{|z|^2}{2}
                              +q\log(1+|z|),\\
 \calT_q(\Lambda)&=\sup_{\mu\in\calB}\int W_{\Lambda,q}\,d\mu,
 \qquad q=0,1,2,\ldots.
 \end{split}
\end{equation*}

\begin{theorem}\label{thm:finite-modification-stability}
A locally finite multiset $\Lambda$ is stable under finite modifications
if and only if \eqref{eq:quadratic-count} holds and
\begin{equation}\label{dual-eq:stable-tests}
 \calT_q(\Lambda)<\infty\qquad\text{for every integer }q\ge0.
\end{equation}
Each supremum in this criterion can be taken over balanced measures
with at most five support points.
\end{theorem}

\begin{proof}
Suppose first that $\Lambda$ is stable, and choose $f_q$ from
Proposition~\ref{prop:finite-modification-moments}. Applying \eqref{eq:fock-point-evaluation} to
$f_q$ and $z^qf_q$ gives
\[
 |f_q(w)|\le C_q(1+|w|)^{-q}e^{|w|^2/2},\qquad w\in\C.
\]
Factor $f_q=c_qP_\Lambda e^{a_qz+b_qz^2}$ using
\eqref{eq:allforms}. Since
$1+|z+e^{it}|\ge(1+|z|)/2$, taking the circle mean of the logarithm
gives
\[
 \mathcal U^{\mathrm{can}}_\Lambda(z)+\Re(a_qz+b_qz^2)
 \le\frac{|z|^2}{2}-q\log(1+|z|)+C'_q.
\]
The constants absorb $\log|c_q|$, $q\log2$, and $1/2$.
Integrating against a balanced measure proves \eqref{dual-eq:stable-tests}.

Conversely, fix $q\ge0$ and apply Proposition~\ref{prop:balanced-five-points} to
$W_{\Lambda,q+2}$. It gives $a,b$ and $C$ such that
\[
 \mathcal U^{\mathrm{can}}_\Lambda(z)+\Re(az+bz^2)
 \le\frac{|z|^2}{2}-(q+2)\log(1+|z|)+C.
\]
By \eqref{eq:canonical-smoothed-potential}, $f=P_\Lambda e^{az+bz^2}$ satisfies
\[
 |f(z)|\le e^C e^{|z|^2/2}(1+|z|)^{-q-2}.
\]
Both $f$ and $z^qf$ belong to $\F$, because
$\int_0^\infty r(1+r)^{-4}\,dr<\infty$. The zero multiset of $f$ is
$\Lambda$, while that of $z^qf$ is $\Lambda+q\delta_0$. Proposition~\ref{prop:finite-modification-moments} proves stability. The support
bound follows from Proposition~\ref{prop:balanced-five-points}.
\end{proof}

\begin{remark}\label{dual-rem:two-tests}
The test measures in Theorems~\ref{thm:complete-zero-set-characterization} and \ref{thm:finite-modification-stability}
serve different purposes. In the entropy criterion they are absolutely
continuous with respect to area measure. An atomic measure has infinite
relative entropy and cannot be substituted there. In the stability
criterion they test a pointwise majorant, and five atoms suffice.
The constants and the realizing functions in the stability theorem may
depend on $q$. A single function satisfying all polynomial moment
conditions is not asserted by that theorem.
\end{remark}

\subsection{Distance kernels and truncation error}

The atomic tests can be written directly in terms of the positions and
multiplicities, without the normalizing harmonic terms in
\eqref{eq:canonical-smoothed-series}. For $\mu=\sum_{\ell=1}^N t_\ell\delta_{z_\ell}\in\calB$,
put
\begin{equation*}K_\mu(\lambda)=\sum_{\ell=1}^N t_\ell \ell(z_\ell-\lambda)-\log|\lambda|,
 \qquad \lambda\ne0.
\end{equation*}

\begin{proposition}\label{prop:balanced-distance-tail}
Suppose $n_\Lambda(r)\le Cr^2$ for $r\ge1$. Then
\begin{equation}\label{dual-eq:kernel-sum}
 \int \mathcal U^{\mathrm{can}}_\Lambda\,d\mu
 =m_0\sum_\ell t_\ell \ell(z_\ell)
   +\sum_{\lambda\ne0}m_\lambda K_\mu(\lambda),
\end{equation}
where the series is absolutely convergent. Write
$B=\max_\ell|z_\ell|$ and $M_3=\sum_\ell t_\ell|z_\ell|^3$.
For $R>2B+2$,
\begin{equation}\label{dual-eq:tail-error}
 \left|\sum_{|\lambda|>R}m_\lambda K_\mu(\lambda)\right|
 \le\frac{2CM_3}{R}.
\end{equation}
\end{proposition}

\begin{proof}
Integrating each summand in \eqref{eq:canonical-smoothed-series}, the moment conditions
cancel $\Re(z/\lambda+z^2/(2\lambda^2))$. This gives
\eqref{dual-eq:kernel-sum} once convergence is checked. For
$|\lambda|>2B+2$, all logarithmic caps are inactive and
\[
 K_\mu(\lambda)
 =-\Re\sum_{j=3}^\infty\frac{1}{j\lambda^j}
                         \sum_\ell t_\ell z_\ell^j.
\]
Since $|z_\ell/\lambda|<1/2$,
\[
 |K_\mu(\lambda)|
 \le\sum_\ell t_\ell\sum_{j=3}^\infty
                       \frac{|z_\ell/\lambda|^j}{j}
 \le\frac{2M_3}{3|\lambda|^3}.
\]
Now \eqref{eq:canonical-tail-count} proves absolute convergence and
\eqref{dual-eq:tail-error}.
\end{proof}

Thus each test in Theorem~\ref{thm:finite-modification-stability} has the explicit form
\begin{equation*}\begin{split}
 &m_0\sum_\ell t_\ell \ell(z_\ell)
 +\sum_{\lambda\ne0}m_\lambda K_\mu(\lambda)\\
 &\hspace{1cm}-\frac12\sum_\ell t_\ell|z_\ell|^2
            +q\sum_\ell t_\ell\log(1+|z_\ell|).
 \end{split}
\end{equation*}
For each $q$, these quantities must have a common finite upper bound
as the balanced configurations of at most five points vary. If a
truncated test exceeds a proposed bound by more than
$2CM_3/R$, it disproves that bound using finitely many zeros.
The number five refers to the support of the \emph{test measure}.
Each test still reads the whole zero multiset, and the theorem does
not reduce the characterization to finitely many inequalities or to
finitely many moments of $\Lambda$.

\section{Infinite-dimensional spaces and finite insertion indices}
\label{sec:infinite-dimensional-insertion}

The finite-modification lemma shows that, once $\Lambda$ is an exact zero
multiset, the possibility of adding finitely many prescribed zeros depends
only on their total multiplicity.  This allows a zero-data invariant which is
independent of the locations of the added points.

\begin{definition}[Exact insertion index]
Let $\Lambda$ be an exact zero multiset.  Its exact insertion index is
\[
 \iota(\Lambda):=\sup\bigl\{q\in\Z_{\ge0}:\Lambda+\nu
 \text{ is an exact zero multiset for every finite effective }\nu
 \text{ with }|\nu|=q\bigr\}.
\]
By Lemma~\ref{lem:finite-modifications}, ``every'' may equivalently be
replaced by ``some''. The admissible integers form an initial segment:
dividing out any subset of finitely many added zeros preserves membership
in $\F$. The value $+\infty$ is allowed. We say that $\Lambda$ is
\emph{maximal under finite additions} if $\iota(\Lambda)=0$. This does not
assert maximality under inclusion among all exact zero multisets.
\end{definition}

We now construct infinite-dimensional zero-based subspaces with arbitrary
prescribed finite insertion index.  Put $s:=\sqrt\pi$. Recall that
\[
 \mathcal L=s(\Z+i\Z),\qquad \mathcal V:=is\Z,
\]
and let $\sigma$ be the Weierstrass sigma function associated with $\mathcal L$.
We use the standard critical-lattice estimate
\begin{equation}\label{eq:sigma-critical-line}
 c_0\,\dist(z,\mathcal L)
 \le |\sigma(z)|e^{-|z|^2/2}
 \le C_0\,\dist(z,\mathcal L),\qquad z\in\C,
\end{equation}
which follows from square-lattice quasi-periodicity; see
\cite[Chapter~5]{Zhu2012}.

\begin{lemma}[Line deletion estimate]\label{lem:line-deletion-estimate}
There exists $C>0$ such that
\[
 \frac{|\sigma(z)|e^{-|z|^2/2}}{|\sinh(sz)|}
 \le C e^{-s|\Re z|},\qquad z\in\C\setminus\mathcal V,
\]
and the quotient extends continuously across $\mathcal V$.
\end{lemma}

\begin{proof}
Write $z=x+iy$.  Since
\[
 |\sinh(sz)|^2=\sinh^2(sx)+\sin^2(sy),
\]
there is a constant $c>0$ such that
\[
 |\sinh(sz)|\ge c e^{s|x|}\min\{1,\dist(z,\mathcal V)\}.
\]
Indeed, this is immediate for $|x|\ge1$, while for $|x|\le1$ it follows,
after reducing $y$ to one period, from the simple zeros of $\sinh(sz)$ on
$\mathcal V$ and compactness away from those zeros.  Because $\mathcal V\subset\mathcal L$,
$\dist(z,\mathcal L)\le\dist(z,\mathcal V)$; combining this
with \eqref{eq:sigma-critical-line} gives the asserted estimate.  Every point
of $\mathcal V$ is a simple zero of both $\sigma$ and $\sinh(sz)$,
so $\sigma/\sinh(sz)$ extends holomorphically across $\mathcal V$.
Its weighted modulus gives the required continuous extension.
\end{proof}

Fix $N\ge1$, choose distinct
$a_1,\ldots,a_N\in\mathcal L\setminus\mathcal V$, and put
\[
 P_N(z):=\prod_{j=1}^N(z-a_j),\qquad
 g_N(z):=\frac{\sigma(z)}{\sinh(sz)P_N(z)}.
\]
All apparent poles are removable and
\[
 \mathcal Z(g_N)=\Sigma_N:=\mathcal L\setminus
 \bigl(\mathcal V\cup\{a_1,\ldots,a_N\}\bigr).
\]

\begin{proposition}\label{prop:insertion-lower}
The function $g_N$ satisfies
\[
 \int_\C (1+|z|^2)^{N-1}|g_N(z)|^2e^{-|z|^2}\,dA(z)<\infty.
\]
Consequently $\iota(\Sigma_N)\ge N-1$.
\end{proposition}

\begin{proof}
For $|z|$ outside a fixed disk containing all $a_j$,
$|P_N(z)|\ge c_N|z|^N$.  Lemma~\ref{lem:line-deletion-estimate} therefore
gives
\[
 (1+|z|^2)^{N-1}|g_N(z)|^2e^{-|z|^2}
 \le C_N\frac{e^{-2s|x|}}{1+x^2+y^2},\qquad z=x+iy.
\]
The right-hand side is integrable because
\[
 \int_\mathbb R\frac{dy}{1+x^2+y^2}=\frac{\pi}{\sqrt{1+x^2}}.
\]
The integral on the complementary compact disk is finite by holomorphy.
Hence $z^{N-1}g_N\in\F$, and Lemma~\ref{lem:finite-modifications} gives the
last assertion.
\end{proof}

We next rule out one further insertion.  For $p,q\in\Z$ define the
shifted cell centres
\[
 z_{p,q}:=s\left(p+\frac12\right)
       +is\left(q+\frac12\right).
\]
Choose $0<\rho<s/5$ and put $Q_{p,q}=D(z_{p,q},\rho)$.  The disks are
uniformly separated from $\mathcal L$.  Hence \eqref{eq:sigma-critical-line}
provides a constant $c>0$ such that
\[
 |\sigma(z)|^2e^{-|z|^2}\ge c,
 \qquad z\in Q_{p,q},
\]
while $|\sinh(sz)|\le C e^{s|\Re z|}$ and
$|P_N(z)|\le C_N(1+|z|)^N$.  Therefore, on every sufficiently remote
$Q_{p,q}$,
\begin{equation}\label{eq:critical-cell-lower-new}
 (1+|z|^2)^N|g_N(z)|^2e^{-|z|^2}
 \ge c_N e^{-2s|\Re z|}.
\end{equation}

\begin{lemma}[Hadamard rigidity]\label{lem:hadamard-rigidity-line}
If $F\in\F$ and $\mathcal Z(F)=\Sigma_N$, then
\[
 F(z)=C e^{Az^2+Bz}g_N(z)
\]
for some $A,B,C\in\C$ with $C\ne0$.
\end{lemma}

\begin{proof}
Both $F$ and $g_N$ are Fock functions and therefore have order at most two.
They have the same zero divisor.  Applying Hadamard factorization to both
functions with the same genus-two canonical product for $\Sigma_N$ shows
that their quotient is the exponential of a polynomial of degree at most two.
\end{proof}

\begin{proposition}\label{prop:insertion-upper}
There is no exact realizer $F$ of $\Sigma_N$ for which
\[
 \int_\C(1+|z|^2)^N|F(z)|^2e^{-|z|^2}\,dA(z)<\infty.
\]
Hence $\iota(\Sigma_N)=N-1$.
\end{proposition}

\begin{proof}
Assume such an $F$ exists.  By Lemma~\ref{lem:hadamard-rigidity-line},
$F=Ce^{Az^2+Bz}g_N$.

Suppose first that $A\ne0$.  There is an open sector $\mathcal S$ and a
constant $c_A>0$ for which $\Re(Az^2)\ge c_A|z|^2$ on $\mathcal S$.
A smaller sector contains infinitely many shifted lattice centres
$z_{p,q}$, and for all sufficiently remote such centres the entire disk
$Q_{p,q}$ lies in $\mathcal S$.  On these pairwise disjoint disks,
\eqref{eq:critical-cell-lower-new} yields
\[
 (1+|z|^2)^N|F(z)|^2e^{-|z|^2}
 \ge c\exp\bigl(2c_A|z|^2-C_B|z|\bigr),
\]
so the integral diverges.  Thus $A=0$.

Write $B=u+iv$.  Consider the vertical family
\[
 w_q=\frac s2+is\left(q+\frac12\right),\qquad Q_q=D(w_q,\rho).
\]
On these disks \eqref{eq:critical-cell-lower-new} has a positive lower bound
independent of $q$.  If $v\ne0$, then $\Re(Bw_q)$ tends to $+\infty$ in one
of the directions $q\to\pm\infty$, and the corresponding disk contributions
diverge exponentially.  If $v=0$, then $e^{2\Re(Bz)}$ is uniformly bounded
below on all $Q_q$, so infinitely many disjoint disks each contribute a fixed
positive amount.  This again gives divergence.  No $B$ is possible.

Therefore the order-$N$ weighted exact realization does not exist.  By
Lemma~\ref{lem:finite-modifications}, $N$ additional zeros cannot be inserted,
while Proposition~\ref{prop:insertion-lower} allows $N-1$; hence
$\iota(\Sigma_N)=N-1$.
\end{proof}

\begin{proposition}\label{prop:infinite-dim-line}
For every $N\ge1$, $\dim\mathcal I(\Sigma_N)=\infty$.
\end{proposition}

\begin{proof}
For every real $t$ with $|t|<s$, the function $e^{tz}g_N$ has the same exact
zero divisor $\Sigma_N$.  Outside a fixed disk,
\[
 |e^{tz}g_N(z)|^2e^{-|z|^2}
 \le C_N\frac{e^{-2(s-|t|)|x|}}{(1+x^2+y^2)^N},
\]
which is integrable.  Thus $e^{tz}g_N\in\F$.  Distinct real parameters give
linearly independent exponentials; hence the corresponding Fock functions
are linearly independent.
\end{proof}

\begin{theorem}[Infinite-dimensional zero spaces with prescribed finite insertion index]
\label{thm:infinite-dimensional-prescribed-index}
For every $m\in\Z_{\ge0}$ there exists a simple exact zero multiset
$\Lambda_m$ such that
\[
 \dim\mathcal I(\Lambda_m)=\infty,
 \qquad
 \iota(\Lambda_m)=m.
\]
Explicitly, with $N=m+1$, one may take
\[
 \Lambda_m=\mathcal L\setminus
 \Bigl(\mathcal V\cup\{a_1,\ldots,a_N\}\Bigr),
\]
where the $a_j$ are arbitrary distinct points of $\mathcal L\setminus\mathcal V$. In particular, $m=0$ yields a zero multiset maximal under
finite additions, with an infinite-dimensional vanishing space.
\end{theorem}

\begin{proof}
Take $N=m+1$ and $\Lambda_m=\Sigma_{m+1}$. Propositions~\ref{prop:insertion-lower} and
\ref{prop:insertion-upper} give $\iota(\Sigma_N)=N-1=m$, while
Proposition~\ref{prop:infinite-dim-line} gives
$\dim\mathcal I(\Sigma_N)=\infty$.
\end{proof}

These examples also have critical Beurling density. The square lattice
satisfies $\mathcal L(D(w,r))=r^2+O(r+1)$ uniformly in $w\in\C$.
Every disk $D(w,r)$ contains at most $2r/s+1$ points of $\mathcal V$.
Deleting this line and $m+1$ further points therefore gives
\[
 \Lambda_m(D(w,r))=r^2+O_m(r+1)\qquad(w\in\C,\ r>0),
\]
with an error uniform in the center. Consequently,
$D^-(\Lambda_m)=D^+(\Lambda_m)=1/\pi$.

\begin{remark}
When $\dim\mathcal I(\Lambda)=d<\infty$, Zhu's finite-dimensional structure
theorem \cite[Corollary~9 and Theorem~10]{Zhu2011} implies that the exact
insertion index is $d-1$. Theorem~\ref{thm:infinite-dimensional-prescribed-index}
shows that the dimension of the vanishing space alone does not determine
the exact insertion index in the infinite-dimensional case: $\iota$ can
be any prescribed finite integer while $\dim\mathcal I(\Lambda)=\infty$.
\end{remark}

\begin{corollary}[Finite additions beyond the insertion index]
\label{cor:intermediate-zero-status}
Let $\Lambda_m$ be the multiset in
Theorem~\ref{thm:infinite-dimensional-prescribed-index}. For every finite
effective multiset $\nu$ with $|\nu|>m$, the multiset $\Lambda_m+\nu$ is
not an exact Fock zero multiset, but
\[
 \dim\mathcal I(\Lambda_m+\nu)=\infty.
\]
In particular, it is not a uniqueness multiset.
\end{corollary}

\begin{proof}
Lemma~\ref{lem:finite-modifications} and $\iota(\Lambda_m)=m$ exclude an
exact realization. Within $\mathcal I(\Lambda_m)$, the extra vanishing
conditions defining $\mathcal I(\Lambda_m+\nu)$ are the kernel of a linear
map to $\C^{|\nu|}$. At a point already in $\Lambda_m$, these conditions
start at its existing prescribed order. This kernel has finite codimension
in an infinite-dimensional space and is therefore infinite dimensional.
\end{proof}

\begin{remark}[Finite maximality and infinite enlargement]
\label{rem:finite-versus-inclusion-maximality}
For $m=0$, the function $g_1$ above has zero multiset $\Lambda_0$.
Proposition~\ref{prop:infinite-dim-line} gives
$g_1,e^{sz/2}g_1\in\F$. Consequently
\[
 (e^{sz/2}-1)g_1\in\F,\qquad
 \mathcal Z\bigl((e^{sz/2}-1)g_1\bigr)=\Lambda_0+4is\Z.
\]
The added simple zeros lie on the deleted line and are disjoint from
$\Lambda_0$. Thus this multiset, although maximal under finite additions,
has a strict infinite enlargement that is an exact zero multiset.
\end{remark}

\section{Critical lattice perturbations}
\label{sec:critical-lattice}
\label{sec:geometric-characterizations}

A bounded motion of each critical lattice point can accumulate into an
unbounded change of logarithmic potential. We compare with the
square-lattice sigma function to remove the Gaussian background, and
retain the first and second displacement terms; the remaining tail is
summable. Example~\ref{ex:bounded-oscillatory-endpoint} shows why the second
term is needed. Lemma~\ref{lem:circular-fock-smoothing} converts the resulting
potential estimate into a norm criterion. The relative potential
$\mathcal U^{\mathrm{rel}}_{\Gamma/\mathcal L}$ below has already had the Gaussian background removed,
unlike the canonical potential $\mathcal U^{\mathrm{can}}_\Lambda$ of
Section~\ref{sec:analytic-characterization}.

\subsection{Bounded displacements}
\label{sec:bounded-displacement-geometry}

A bounded matching to the lattice has the following equivalent area
formulation.

Recall $\mathcal L=\sqrt\pi(\Z+i\Z)$. A labelled multiset
$\Gamma$ admits a bounded allocation if the plane can be partitioned,
up to null sets, into measurable cells $A_\gamma$ of area $\pi$ such that
$\gamma\in A_\gamma$ and $\sup_\gamma\operatorname{diam}A_\gamma<\infty$.
Attaching a representative to its cell does not change its area.
Null intersections are allowed, so distinct labels may represent the
same point.

\begin{lemma}
\label{lem:bounded-area-allocation}
A labelled multiset admits a bounded allocation if and only if
there is a bijective matching to $\mathcal L$ with uniformly bounded
displacement.
\end{lemma}

\begin{proof}
Let $Q_\lambda$ be the square lattice cell of area $\pi$ centered at
$\lambda$. Suppose $\operatorname{diam}A_\gamma\le D$. Join
$\lambda$ to $\gamma$ if $|Q_\lambda\cap A_\gamma|>0$, and give the
edge weight $\pi^{-1}|Q_\lambda\cap A_\gamma|$. Row and column
sums are one. Every edge satisfies
$|\lambda-\gamma|\le D+\sqrt{2\pi}$. Cells with representatives
in a disk of radius $R$ lie in the concentric disk of radius $R+D$;
their equal areas bound their number. Thus the graph is locally finite.

For a finite collection $E$ of lattice vertices, let $N(E)$ be its
neighbors. Writing $w_{\lambda\gamma}$ for the edge weights gives
\[
 |E|=\sum_{\lambda\in E}\sum_{\gamma\in N(E)}w_{\lambda\gamma}
 \le\sum_{\gamma\in N(E)}\sum_{\lambda\in\mathcal L}
                   w_{\lambda\gamma}=|N(E)|.
\]
The same inequality holds for finite sets of representatives. Both
vertex sets are countable. By Hall's finite marriage theorem, there
is a matching covering the first $n$ lattice vertices for every $n$.
Local finiteness permits a diagonal subsequence in which each partner
stabilizes. Its limit covers the lattice. Reversing the two sides gives
a matching covering the representatives.

The union of these two matchings has degree at most two. An edge that
belongs to both matchings is an isolated two-vertex component. In every
other component the edges alternate between the two matchings. A finite
path cannot have an odd number of vertices: its endpoints would lie on
the same side and would both have to be covered by the matching that
covers that side, whereas alternating edges at those endpoints belong
to different matchings. Thus each component is an even cycle, a finite
path with an even number of vertices, or a one-sided or two-sided
infinite path. Alternating edges, starting at the endpoint when there
is one, give a perfect matching in every component. All chosen edges
satisfy the same displacement bound.

Conversely, assign $Q_\lambda$ to the point matched to $\lambda$
and attach that point to the cell. The area is unchanged, the diameters
remain uniformly bounded, and the cells partition the plane up to
null sets.
\end{proof}

For example, consider regular $M_k$-gons of radii $R_k$, where
$M_k\ge ck$, $\sum_{j\le k}(M_j-2j)=O(k)$, and $R_k=k+O(1)$.
Put $s_k=(\sum_{j\le k}M_j)^{1/2}$. Then
$s_k^2=k(k+1)+O(k)$ and $s_k=k+O(1)$. Subtracting consecutive
partial sums gives $M_k=O(k)$, whence
\[
 s_k-s_{k-1}=\frac{M_k}{s_k+s_{k-1}}=O(1),
 \qquad \frac{2\pi s_k}{M_k}=O(1).
\]
Divide each annulus $s_{k-1}\le|z|<s_k$ into $M_k$ equal-angle
sectors. A sector has area
$\pi(s_k^2-s_{k-1}^2)/M_k=\pi$; its radial width and outer arc
length are bounded. Attaching the polygon points increases the
diameters by a bounded amount, since $R_k=k+O(1)$. The resulting
cells give a bounded allocation, which remains bounded after any
uniformly bounded tangential perturbation.

Let $\Gamma=(\mu_\lambda)_{\lambda\in\mathcal L}$, where
$|\mu_\lambda-\lambda|\le M$, and put
$\delta_\lambda=\mu_\lambda-\lambda$. Choose a finite set $A_0$
containing every lattice point with $|\lambda|\le2M+2$, so that
$\lambda\mu_\lambda\ne0$ outside $A_0$. For $T>4M+4$, define
\[
 K_j(w)=\begin{cases}w^{-j},&|w|>T,\\0,&|w|\le T.\end{cases}
\]
Define the displacement field by
\begin{equation}
 \begin{split}
 V_\Gamma(z)=-\Re\sum_{\lambda\notin A_0}\bigg\{
 &\delta_\lambda
       \left[K_1(z-\lambda)+\frac1\lambda+\frac z{\lambda^2}\right]\\
 &+\frac{\delta_\lambda^2}{2}
       \left[K_2(z-\lambda)-\frac1{\lambda^2}\right]\bigg\}.
 \end{split}
 \label{eq:bounded-displacement-field}
\end{equation}
The terms in each brace are grouped. For $|z|\le R$ and
$|\lambda|>2(R+T+1)$, the cutoffs are inactive and
\[
 \begin{split}
 \frac1{z-\lambda}+\frac1\lambda+\frac z{\lambda^2}
   &=-\frac{z^2}{\lambda^2(\lambda-z)},\\
 \frac1{(z-\lambda)^2}-\frac1{\lambda^2}
   &=\frac{2\lambda z-z^2}{\lambda^2(\lambda-z)^2}.
 \end{split}
\]
Each brace is $O_R(|\lambda|^{-3})$. The lattice sum of these
bounds is finite, so the grouped series converges absolutely, with
locally uniform tails. The remaining cutoff terms are locally bounded
and measurable. Hence so is $V_\Gamma$.

\begin{theorem}[Bounded displacement criterion]
\label{thm:bounded-displacement-field}
Suppose $\Gamma$ and $V_\Gamma$ are defined as above. Let
$\Xi=\Gamma-\nu_-+\nu_+$ be a finite modification, where
$0\le\nu_-\le\Gamma$ and $q=|\nu_+|-|\nu_-|$. Then
\begin{equation}
 \begin{gathered}
 \Xi\text{ is a zero set for }\F
 \quad\Longleftrightarrow\\
 \exists a\in\C:\quad
 \int_\C(1+|z|^2)^q
       e^{2V_\Gamma(z)+2\Re(az)}\dd(z)<\infty.
 \end{gathered}
 \label{eq:bounded-displacement-zero-test}
\end{equation}
The points need not be separated, and no convergence of angular
moments is required. The criterion is independent of the cutoff,
the exceptional set $A_0$, and the bounded matching.
\end{theorem}

We next estimate the displacement potential.

For the displacement field, put $h=2M+2$ and $c_h(w)=\log\max(h,|w|)$. Consider
\begin{equation}
 \begin{split}
 \mathcal U^{\mathrm{rel}}_{\Gamma/\mathcal L}(z)=\sum_{\lambda\notin A_0}\bigg\{
 &c_h(z-\mu_\lambda)-c_h(z-\lambda)
                 -\log|\mu_\lambda/\lambda|\\
 &+\Re\left[z\left(\frac1{\mu_\lambda}-\frac1\lambda\right)\right]
 \bigg\}.
 \end{split}
 \label{eq:bounded-capped-field}
\end{equation}

\begin{lemma}[Displacement potential]
\label{lem:bounded-displacement-potential}
There is an entire function $H$ of order at most two with
$\mathcal Z(H)=\Gamma$ such that
\begin{equation}
 S_h\log|H|(z)-|z|^2/2=\mathcal U^{\mathrm{rel}}_{\Gamma/\mathcal L}(z)+O(1).
 \label{eq:bounded-generator-mean}
\end{equation}
The capped potential satisfies
\begin{equation}
 |\mathcal U^{\mathrm{rel}}_{\Gamma/\mathcal L}(z)|\le C(1+|z|)\log(2+|z|).
 \label{eq:bounded-field-growth}
\end{equation}
and, for a constant $c\in\C$,
\begin{equation}
 \mathcal U^{\mathrm{rel}}_{\Gamma/\mathcal L}(z)=V_\Gamma(z)+\Re(cz)+O(1).
 \label{eq:bounded-second-order-reduction}
\end{equation}
All error bounds are uniform in $z\in\C$.
\end{lemma}
\begin{proof}
The terms in each brace of \eqref{eq:bounded-capped-field} are grouped.
Put $E_1(w)=(1-w)e^w$.
For $|z|\le R$ and $|\lambda|>2(R+M+h)$, the logarithm normalized
to vanish at zero satisfies
\[
 \log\frac{E_1(z/\mu_\lambda)}{E_1(z/\lambda)}
 =-\sum_{m\ge2}\frac{z^m}{m}
             (\mu_\lambda^{-m}-\lambda^{-m}).
\]
Integration along the segment from $\lambda$ to $\mu_\lambda$ gives
\[
 |\mu_\lambda^{-m}-\lambda^{-m}|
 \le\frac{mM}{(|\lambda|-M)^{m+1}},
 \qquad
 \left|\log\frac{E_1(z/\mu_\lambda)}{E_1(z/\lambda)}\right|
 \le\frac{MR^2}{(|\lambda|-M)^3}
       \frac1{1-R/(|\lambda|-M)}.
\]
The last bound is $O_R(|\lambda|^{-3})$. Thus the logarithmic tails
of the product
\[
 Q(z)=\prod_{\lambda\in A_0}\frac{z-\mu_\lambda}{z-\lambda}
       \prod_{\lambda\notin A_0}
             \frac{E_1(z/\mu_\lambda)}{E_1(z/\lambda)}
\]
converge normally. Since a compact set meets only finitely many labels,
$Q$ is meromorphic with zero and pole orders given by $\Gamma-\mathcal L$, including at
coincident points. For the distant factors, the caps are inactive and
the brace in \eqref{eq:bounded-capped-field} is the real part of the
same logarithm. The grouped series for $\mathcal U^{\mathrm{rel}}_{\Gamma/\mathcal L}$ therefore also
converges absolutely and locally uniformly.

Let $\sigma$ be the square-lattice sigma function with
$\sigma'(0)=1$. We recall that
\begin{equation}
 |\sigma(z)|e^{-|z|^2/2}\asymp\dist(z,\mathcal L);
 \label{eq:lattice-classical-sigma}
\end{equation}
see \cite[Chapter~5]{Zhu2012}. The normalization follows from
square-lattice quasi-periodicity:
\[
 \sigma(z+\omega)=-e^{\overline\omega z+|\omega|^2/2}\sigma(z),
 \qquad \omega\in\{\sqrt\pi,i\sqrt\pi\}.
\]
Indeed, $|z+\omega|^2=|z|^2+2\Re(\overline\omega z)+|\omega|^2$
makes $|\sigma(z)|e^{-|z|^2/2}$ periodic. Dividing by
$\dist(z,\mathcal L)$ gives a positive continuous function across
each simple zero. Its maximum and minimum on a fundamental square
prove \eqref{eq:lattice-classical-sigma}. Thus $H=\sigma Q$ is
entire with zero multiset $\Gamma$.

The identity $S_h\log|z-\alpha|=c_h(z-\alpha)$, applied term by
term to the normally convergent tail, gives
\[
 S_h\log|Q|(z)=\mathcal U^{\mathrm{rel}}_{\Gamma/\mathcal L}(z)+O(1).
\]
The finite logarithmic singularities are integrable. For each exceptional
factor, the mean is $c_h(z-\mu_\lambda)-c_h(z-\lambda)$.
Since $c_h$ is $h^{-1}$-Lipschitz, their sum has modulus at most
$h^{-1}\sum_{\lambda\in A_0}|\delta_\lambda|$.

By \eqref{eq:lattice-classical-sigma},
$S_h\log|\sigma|(z)-|z|^2/2=O(1)$. To see the uniformity,
the quadratic term has mean $|z|^2/2+h^2/2$, and the logarithmic
distance is bounded above by a constant. For
$0<u<\min(1,h/2)$, only the uniformly finitely many lattice points
in $D(z,h+1)$ can be within $u$ of the averaging circle. Each
subtends an arc of angular length at most $C_hu$, so
\[
 \frac1{2\pi}\bigl|\{t:\dist(z+he^{it},\mathcal L)<u\}\bigr|
 \le C_hu.
\]
The layer-cake identity
\[
 \frac1{2\pi}\int_0^{2\pi}
    \bigl[-\log\min(1,\dist(z+he^{it},\mathcal L))\bigr]\,dt
 =\int_0^\infty\frac1{2\pi}
    \bigl|\{t:\dist(z+he^{it},\mathcal L)<e^{-v}\}\bigr|\,dv
\]
now gives a uniform lower bound for the circular mean of the logarithmic
distance. This proves \eqref{eq:bounded-generator-mean}.

We next prove \eqref{eq:bounded-field-growth}.
We use the elementary lattice estimates
\[
 \begin{split}
 \sum_{|\lambda-w|\le R}\frac1{1+|\lambda-w|}&\le C(1+R),\\
 \sum_{0<|\lambda|\le R}\frac1{|\lambda|}&\le C(1+R),
 \qquad
 \sum_{0<|\lambda|\le R}\frac1{|\lambda|^2}
       \le C\log(2+R),\\
 \sum_{|\lambda-w|>R}\frac1{|\lambda-w|^3}&\le \frac C R
       \quad(R\ge1).
 \end{split}
\]
To obtain these bounds uniformly in $w$, place the disjoint lattice
squares in each unit-width annulus inside an annulus enlarged by their
diameter. There are $O(j+1)$ such squares at radius $j$; summation
over $j$ gives the estimates.

Write $r=|z|\ge1$. When $|\lambda|>2r+2h+2M$, the product estimate
above, with $R=r$, bounds each grouped term in
\eqref{eq:bounded-capped-field} by $Cr^2|\lambda|^{-3}$. Their sum
is at most $Cr^2/(2r+2h+2M)=O(r)$. For the other terms we use
\[
 \begin{split}
 |c_h(z-\mu_\lambda)-c_h(z-\lambda)|
      &\le C(1+|z-\lambda|)^{-1},\\
 |\log|\mu_\lambda/\lambda||&\le C|\lambda|^{-1},\qquad
 |\mu_\lambda^{-1}-\lambda^{-1}|\le C|\lambda|^{-2}.
 \end{split}
\]
The first bound follows from the Lipschitz bound for $c_h$ when
$|z-\lambda|\le2(M+h)$ and from the mean-value estimate for the
logarithm otherwise. For the last two, use
$|\delta_\lambda/\lambda|<1/2$ and
$|\mu_\lambda|\ge|\lambda|/2$ outside $A_0$. In the present
range $|\lambda|\le2r+2h+2M$, also $|z-\lambda|\le3r+2h+2M$.
The lattice estimates therefore show that the first two bounds sum
to $O(r)$, while the third, multiplied by $r$, sums to
$O(r\log(2+r))$. This proves \eqref{eq:bounded-field-growth}.
Subharmonicity and \eqref{eq:bounded-generator-mean} give
\[
 \log|H(z)|\le\frac{|z|^2}{2}
                 +C(1+|z|)\log(2+|z|),
\]
so $H$ has order at most two.

We finally compare $\mathcal U^{\mathrm{rel}}_{\Gamma/\mathcal L}$ with $V_\Gamma$. Since $T-M>h$ and $M/T<1/4$,
both caps are inactive when $|z-\lambda|>T$. Expanding
$\log(1-\delta_\lambda/(z-\lambda))$ gives
\[
 c_h(z-\mu_\lambda)-c_h(z-\lambda)
 =-\Re\left(\frac{\delta_\lambda}{z-\lambda}
       +\frac{\delta_\lambda^2}{2(z-\lambda)^2}\right)
       +O(|z-\lambda|^{-3}).
\]
The remainders sum to at most $CM^3/T$, uniformly in $z$.
In $|z-\lambda|\le T$, the kernels vanish, there are at most
$C(1+T)^2$ labels, and each capped difference is at most $M/h$.
These terms also give a bounded error. At the origin,
\[
 -\log|\mu_\lambda/\lambda|
 =-\Re\frac{\delta_\lambda}\lambda
       +\frac12\Re\frac{\delta_\lambda^2}{\lambda^2}
       +O(|\lambda|^{-3}).
\]
Here the error is bounded by $CM^3|\lambda|^{-3}$ and is summable,
since $|\delta_\lambda/\lambda|<1/2$ outside $A_0$. The linear
compensation terms differ by $\Re(cz)$, where
\[
 c=\sum_{\lambda\notin A_0}
   \left(\frac1{\mu_\lambda}-\frac1\lambda
                         +\frac{\delta_\lambda}{\lambda^2}\right)
  =\sum_{\lambda\notin A_0}
                \frac{\delta_\lambda^2}{\lambda^2\mu_\lambda}
\]
is absolutely convergent: $|\mu_\lambda|\ge|\lambda|/2$ makes
its summands at most $2M^2|\lambda|^{-3}$. The identity for $c$
follows from
\[
 \frac1{\lambda+\delta_\lambda}-\frac1\lambda
          +\frac{\delta_\lambda}{\lambda^2}
 =\frac{\delta_\lambda^2}{\lambda^2(\lambda+\delta_\lambda)}.
\]
Combining these estimates proves \eqref{eq:bounded-second-order-reduction}.
\end{proof}

\begin{proof}[Proof of Theorem~\ref{thm:bounded-displacement-field}]
Take $H$ and $\mathcal U^{\mathrm{rel}}_{\Gamma/\mathcal L}$ from Lemma~\ref{lem:bounded-displacement-potential}.
Choose polynomials $P_\pm$ with zero multisets $\nu_\pm$ and put
$H_\Xi=HP_+/P_-$. The function is entire and its zero multiset is
$\Xi$. Since $c_h(z-\alpha)=\tfrac12\log(1+|z|^2)+O(1)$ for
fixed $\alpha$, we obtain
\begin{equation}
 S_h\log|H_\Xi|(z)-|z|^2/2
  =\mathcal U^{\mathrm{rel}}_{\Gamma/\mathcal L}(z)+\tfrac q2\log(1+|z|^2)+O(1).
 \label{eq:bounded-modified-generator-mean}
\end{equation}
It follows that $H_\Xi$ has order at most two as well.

Suppose $f\in\F\setminus\{0\}$ has zero multiset $\Xi$.
Hadamard factorization for $f$ and $H_\Xi$, both of order at most
two, gives $f=e^P H_\Xi$ with $\deg P\le2$.
Averaging \eqref{eq:fock-point-evaluation} gives
\[
 S_h\log|f|(z)\le\log\|f\|_2+\tfrac12(|z|^2+h^2).
\]
Since $S_h\Re P=\Re P$, subtract
\eqref{eq:bounded-modified-generator-mean} and apply
\eqref{eq:bounded-field-growth}. This gives
$\Re P(z)\le C(1+|z|)\log(2+|z|)$; the finite-modification term
is $O(\log(2+|z|))$. If $P(z)=p_2z^2+p_1z+p_0$ and $p_2\ne0$,
choose $\theta$ with $p_2e^{2i\theta}=|p_2|$. Then
\[
 \Re P(re^{i\theta})\ge |p_2|r^2-|p_1|r-|p_0|,
\]
contrary to the preceding bound as $r\to\infty$. Thus $p_2=0$,
and $f$ is a constant multiple of $e^{az}H_\Xi$ for some $a\in\C$.

Harmonicity of $\Re(az)$ and
\eqref{eq:bounded-modified-generator-mean} give
\[
 e^{2S_h\log|e^{az}H_\Xi|(z)-|z|^2}
 \asymp(1+|z|^2)^q e^{2\mathcal U^{\mathrm{rel}}_{\Gamma/\mathcal L}(z)+2\Re(az)}.
\]
By Lemma~\ref{lem:circular-fock-smoothing}, this integral is finite exactly
when $e^{az}H_\Xi\in\F$. Since every Fock function with zero
multiset $\Xi$ has this form up to a constant, we have proved the
criterion with $\mathcal U^{\mathrm{rel}}_{\Gamma/\mathcal L}$ in place of $V_\Gamma$.

Substitute \eqref{eq:bounded-second-order-reduction} and absorb $c$
into $a$. This proves \eqref{eq:bounded-displacement-zero-test}.
Equivalence with the zero-set property also proves independence of
the cutoff, exceptional set, and matching.
\end{proof}

\subsection{Radial displacements}
\label{sec:smooth-radial-displacements}

For a smooth radial displacement, the field is radial up to a bounded
error. This reduces the zero-set criterion to a one-dimensional integral.

\begin{corollary}
\label{cor:smooth-radial-displacements}
Suppose $b:[R_0,\infty)\to\mathbb R$ is $C^2$ and satisfies
\begin{equation*}
 |b(r)|\le C,\qquad |b'(r)|\le C/r,
       \qquad |b''(r)|\le C/r^2.
 \end{equation*}
Set $\mu_\lambda=(1+b(|\lambda|)/|\lambda|)\lambda$ outside a
finite set of lattice labels, and replace the remaining points
arbitrarily. Increase $R_0$, if necessary, so that
$R_0>2\sup|b|+2$ and all exceptional labels lie in
$|\lambda|\le R_0$. Define
\[
 B(r)=\sum_{R_0<|\lambda|\le r}
                   \log(1+b(|\lambda|)/|\lambda|).
\]
For a finite modification $\Xi$ of this array with net multiplicity
$q$, we have
\begin{equation}
 \Xi\text{ is a zero set for }\F
 \quad\Longleftrightarrow\quad
 \int_1^\infty r^{2q+1}e^{-2B(r)}\,dr<\infty.
 \label{eq:smooth-radial-zero-test}
\end{equation}
This is equivalent to
\begin{equation}
 \int_{R_0}^\infty r^{2q+1}
  \exp\left\{-4\int_{R_0}^r b(t)\,dt
                   +2\int_{R_0}^r\frac{b(t)^2}{t}\,dt\right\}\,dr
       <\infty.
 \label{eq:smooth-radial-drift}
\end{equation}
If the integral diverges, the multiset is a uniqueness multiset.
\end{corollary}

\begin{proof}
We prove $\mathcal U^{\mathrm{rel}}_{\Gamma/\mathcal L}(z)=-B(|z|)+O(1)$ uniformly in $z$.
Choose $A_0$ to contain every lattice point in a sufficiently large
disk and all exceptional labels. This changes the tail of $B$ by
a constant. Outside $A_0$, the identity
$\mu_{-\lambda}=-\mu_\lambda$ cancels the linear compensation
in opposite pairs. No symmetry is required of the exceptional points.

Extend $v(w)=b(|w|)w/|w|$ to a bounded radial $C^2$ field that
vanishes near zero and satisfies $T(w):=w+v(w)\ne0$ for $w\ne0$.
Such an extension follows by interpolating the positive radial factor
$1+b(r)/r$ to one. For $r=|w|>R_0$ and $e=w/r$,
\[
 Dv=b'(r)e\otimes e+\frac{b(r)}r(I-e\otimes e),\qquad
 |D^2v|\le C\left(|b''(r)|+\frac{|b'(r)|}r
                                      +\frac{|b(r)|}{r^2}\right).
\]
Thus the extension satisfies
\begin{equation}
 |Dv(w)|\le C(1+|w|)^{-1},\qquad
 |D^2v(w)|\le C(1+|w|)^{-2}.
 \label{eq:radial-field-derivatives}
\end{equation}
It changes only finitely many lattice points.

Let $\eta$ be a smooth radial probability density supported in
$|w|\le L$, and put $\phi=\eta*\log|\cdot|$. If
$\eta(w)=\eta_0(|w|)$, the circular mean identity gives
\[
 \phi(w)=2\pi\int_0^L\eta_0(t)t\log\max(|w|,t)\,dt.
\]
Thus $\phi$ is smooth, agrees with $\log|w|$ for $|w|>L$, and
satisfies $|D^j\phi(w)|\le C_j(1+|w|)^{-j}$ for $j=1,2,3$.
The bounded, compactly supported difference $\phi-c_h$ changes
\eqref{eq:bounded-capped-field} by $O(1)$, since the array has
uniformly bounded local counts.

Set
\[
 \begin{split}
 F_z(w)&=\phi(z-T(w))-\phi(z-w)-\phi(T(w))+\phi(w),\\
 F_z^{\mathrm{ev}}(w)&=\tfrac12(F_z(w)+F_z(-w)).
 \end{split}
\]
For $|w|$ sufficiently large compared with $|z|$, all four kernels
are logarithms. Since $T(-w)=-T(w)$,
\[
 F_z^{\mathrm{ev}}(w)
 =-\Re\sum_{m\ge1}\frac{z^{2m}}{2m}
                    \big(T(w)^{-2m}-w^{-2m}\big)
 =O_z((1+|w|)^{-3}).
\]
If $M_v=\sup|v|$, real integration along the radial segment gives
\[
 |T(w)^{-2m}-w^{-2m}|
       \le 2mM_v(|w|-M_v)^{-2m-1}.
\]
For $|w|>2|z|+2M_v+L$, the resulting series is bounded by
$C|z|^2(|w|-M_v)^{-3}$. Thus the lattice sum and area integral of
$F_z^{\mathrm{ev}}$ converge absolutely. The linear compensation
in \eqref{eq:bounded-capped-field} cancels on $\lambda,-\lambda$.
Also, $\phi(T(\lambda))-\phi(\lambda)
=\log|\mu_\lambda/\lambda|$ outside a fixed disk. The remaining
finitely many terms are bounded uniformly in $z$, since $\phi$ has
bounded gradient. Hence
\begin{equation}
 \mathcal U^{\mathrm{rel}}_{\Gamma/\mathcal L}(z)=\sum_{\lambda\in\mathcal L}
                            F_z^{\mathrm{ev}}(\lambda)+O(1).
 \label{eq:radial-symmetrized-field}
\end{equation}

To compare the sum with area measure, write
$G_z(w)=\phi(z-T(w))-\phi(z-w)$, so that $F_z=G_z-G_0$.
With $A=I+Dv(w)$, $x=z-w$, and $y=z-T(w)$, the chain rule gives
\[
 D_w^2G_z
 =A^{\mathsf T}D^2\phi(y)A-D^2\phi(x)
                 -\sum_{j=1}^2\partial_j\phi(y)D^2v_j(w).
\]
The difference $D^2\phi(y)-D^2\phi(x)$ is bounded by the
third-derivative estimate along the segment joining $x$ to $y$.
All its points are a bounded distance from $x$. Expanding the two
factors $A$ and using \eqref{eq:radial-field-derivatives} gives
\begin{equation*}
 \begin{split}
 |D_w^2[\phi(z-T(w))-\phi(z-w)]|
 \le C\big[&(1+|z-w|)^{-3}\\
 &+(1+|w|)^{-1}(1+|z-w|)^{-2}\\
 &+(1+|w|)^{-2}(1+|z-w|)^{-1}\big].
 \end{split}
 \end{equation*}
Applying the same estimate with $z=0$ shows that the constant
terms in $F_z$ contribute at most $C(1+|w|)^{-3}$.

Put $p_j(w)=(1+|w|)^{-j}$. Since
$p_3\in L^1(\C)$, $p_1\in L^3(\C)$, and
$p_2\in L^{3/2}(\C)$, H\"older's inequality gives
\[
 \int_\C p_1(w)p_2(z-w)\dd(w)
 \le\|p_1\|_3\|p_2\|_{3/2},\qquad
 \int_\C p_2(w)p_1(z-w)\dd(w)
 \le\|p_2\|_{3/2}\|p_1\|_3.
\]
These bounds and translation invariance of the $p_3$ integral are
uniform in $z$. Each weight is comparable at any two points of a
lattice square, so they also give
\[
 \sup_z\sum_{\lambda\in\mathcal L}
                \sup_{Q_\lambda}|D^2F_z^{\mathrm{ev}}|<\infty.
\]
Taylor's formula with integral remainder and the centroid identity
$\int_{Q_\lambda}(w-\lambda)\dd(w)=0$ now imply
\[
 \left|F_z^{\mathrm{ev}}(\lambda)
      -\pi^{-1}\int_{Q_\lambda}F_z^{\mathrm{ev}}(w)\dd(w)\right|
   \le C\sup_{Q_\lambda}|D^2F_z^{\mathrm{ev}}|.
\]
Summation and \eqref{eq:radial-symmetrized-field} give
\begin{equation}
 \mathcal U^{\mathrm{rel}}_{\Gamma/\mathcal L}(z)=\pi^{-1}\int_\C F_z^{\mathrm{ev}}(w)\dd(w)+O(1),
 \label{eq:radial-area-comparison}
\end{equation}
uniformly in $z$. Write
$A(z)=\pi^{-1}\int_\C F_z^{\mathrm{ev}}(w)\dd(w)$. For every rotation
$R$, $T(Rw)=RT(w)$ and $\phi(Rw)=\phi(w)$, so changing variables
shows that $A(Rz)=A(z)$.

To identify $A$, first consider any bounded lattice displacement and
set
\[
 B(r)=\sum_{\substack{\lambda\notin A_0\\|\lambda|\le r}}
                         \log|\mu_\lambda/\lambda|.
\]
For the nonexceptional labels, write
$j_r(a)=\log^+(r/|a|)$. If both $|\lambda|$ and $|\mu_\lambda|$
are at most $r$, then
$j_r(\mu_\lambda)-j_r(\lambda)=-\log|\mu_\lambda/\lambda|$;
if both exceed $r$, the difference is zero. A discrepancy with
$-B(r)$ can therefore occur only when
$\big||\lambda|-r\big|\le M$. This annulus contains $O(r)$
lattice points, and each discrepancy is $O(1/r)$, because
$\big||\mu_\lambda|-|\lambda|\big|\le M$. The exceptional
labels give a constant once $r$ exceeds all their old and new radii:
each replacement preserves the number of labels. Hence
\[
 J_\Gamma(r)-J_{\mathcal L}(r)=-B(r)+O(1).
\]
The lattice term $D_{\mathcal L}(r)$ is bounded. Indeed, Jensen's
formula for $\sigma$, with $\sigma'(0)=1$, and
\eqref{eq:lattice-classical-sigma} identify it, up to $O(1)$, with
the angular mean of $\log\dist(re^{i\theta},\mathcal L)$.
The upper bound follows from the bounded covering radius of the
lattice. For $r\ge2$ and $0<t<1$, only $O(r)$ lattice points can
lie within distance $t$ of the circle, and each subtends an angular
interval of length at most $Ct/r$. Thus
\[
 \frac1{2\pi}\left|\{\theta:
       \dist(re^{i\theta},\mathcal L)<t\}\right|\le Ct.
\]
The layer-cake formula bounds the mean of
$-\log\min(1,\dist(re^{i\theta},\mathcal L))$ by
$C\int_0^1dt<\infty$. On $1\le r\le2$, the same bound follows
from the finitely many relevant lattice points. We have proved
\begin{equation}
 D_\Gamma(r)=-B(r)+O(1).
 \label{eq:bounded-radial-jensen-drift}
\end{equation}
The annular count also proves, for every fixed $L$,
\begin{equation}
 \sup_{|r-r'|\le L}|B(r)-B(r')|<\infty,
 \label{eq:bounded-radial-local-drift}
\end{equation}
since $O(r)$ summands change and each is $O(1/r)$.

Let
$S_\eta u(z)=\int_\C\eta(w)u(z-w)\dd(w)$ and let
$\mathcal M u(r)$ denote the angular mean at radius $r$.
The argument for \eqref{eq:bounded-generator-mean}, with $\phi$
in place of $c_h$, gives
\[
 S_\eta\log|H|(z)-|z|^2/2=\mathcal U^{\mathrm{rel}}_{\Gamma/\mathcal L}(z)+O(1).
\]
The lattice term is uniformly bounded because the locally integrable
periodic function $\log\dist(\cdot,\mathcal L)$ has bounded
convolution with the smooth compactly supported kernel $\eta$.
Rotation invariance of $\eta$ gives
$\mathcal M S_\eta u=S_\eta\mathcal M u$, where the angular
mean is regarded as a radial function on the plane. Jensen's
formula says $\mathcal M\log|H|(r)=J_\Gamma(r)+c_H$.
Moreover,
\[
 S_\eta(|\cdot|^2/2)(z)
       =|z|^2/2+\tfrac12\int_\C |w|^2\eta(w)\dd(w),
\]
since the first moment of $\eta$ is zero. Angularly averaging
\eqref{eq:radial-area-comparison} therefore yields
\[
 A(r)=S_\eta D_\Gamma(r)+O(1)=-B(r)+O(1).
\]
Indeed, $||z-w|-|z||\le L$ on the support of $\eta$, so
\eqref{eq:bounded-radial-jensen-drift} and
\eqref{eq:bounded-radial-local-drift} give the last equality for
large $r$. The fields are bounded on the remaining compact disk, so
\begin{equation}
 \mathcal U^{\mathrm{rel}}_{\Gamma/\mathcal L}(z)=-B(|z|)+O(1).
 \label{eq:smooth-radial-potential}
\end{equation}

Apply the proof of Theorem~\ref{thm:bounded-displacement-field}
with the capped field. In polar coordinates, its integral is,
up to fixed positive factors,
\[
 \int_1^\infty r^{2q+1}e^{-2B(r)}
       \left(\int_0^{2\pi}e^{2\Re(are^{i\theta})}\,d\theta\right)dr.
\]
Jensen's inequality bounds the inner integral below by $2\pi$,
and it equals $2\pi$ when $a=0$. Finiteness for some $a$ is thus
equivalent to \eqref{eq:smooth-radial-zero-test}.

For the uniqueness assertion, combine
\eqref{eq:bounded-radial-jensen-drift} with
\eqref{eq:finite-change-jensen} to obtain
$D_\Xi(r)=-B(r)+q\log r+O(1)$ for large $r$.
The lower bound \eqref{eq:jensen-radial-norm-lower} now shows that
divergence excludes every nonzero Fock function vanishing on $\Xi$,
including functions with additional zeros.

Finally, put $u(r)=b(r)/r$ and $\psi(w)=\log(1+u(|w|))$
outside a fixed disk, extending $\psi$ smoothly inside. The
hypotheses imply
\[
 u'=\frac{b'}r-\frac b{r^2}=O(r^{-2}),\qquad
 u''=\frac{b''}r-\frac{2b'}{r^2}+\frac{2b}{r^3}=O(r^{-3}).
\]
Since $|u|<1/2$ there,
$(\log(1+u))'=O(r^{-2})$ and
$(\log(1+u))''=O(r^{-3})$. The radial Hessian has eigenvalues
$(\log(1+u))''$ and $(\log(1+u))'/r$, so
$|D^2\psi(w)|\le C(1+|w|)^{-3}$. Midpoint quadrature therefore
has total error bounded by
$C\sum_{\lambda\in\mathcal L}(1+|\lambda|)^{-3}<\infty$.
The union of cells with centers in $|w|\le r$ differs from that
disk only in a fixed-width boundary annulus. There $\psi=O(1/r)$,
and the annulus has area $O(r)$, so the boundary error is $O(1)$.
Consequently,
\[
 B(r)=2\int_{R_0}^r t\log(1+b(t)/t)\,dt+O(1)
     =2\int_{R_0}^r b(t)\,dt
             -\int_{R_0}^r\frac{b(t)^2}{t}\,dt+O(1).
\]
Indeed, $t[\log(1+u)-u+u^2/2]=O(t|u|^3)=O(t^{-2})$,
whose integral is bounded independently of the upper limit.
Exponentiating the bounded error proves
\eqref{eq:smooth-radial-drift}.
\end{proof}

\section{Circular arrays and angular criticality}
\label{sec:circular-arrays}
\label{sec:circular-geometry}

For circular arrays, the radial placement and the angular rearrangement
produce competing lower-order terms after the quadratic growth is fixed.
At the endpoint, integrability depends on the measure of directions close
to the angular maximum. The distribution function of these near-maximizing
directions treats both finite-order and infinitely flat maxima. We first
establish radial and sampling estimates, then a quotient rigidity result
that will turn failure of the norm test into uniqueness. The proofs of
the auxiliary sampling, growth, and kernel estimates are collected in
Appendix~\ref{app:circular-estimates}; the geometric construction and the
proof of the angular criterion remain here.

\subsection{Circular arrays}
\label{sec:circular-radial-data}

The next two lemmas give the radial growth and circle sampling formulas
used in Theorem~\ref{thm:smooth-angular-profile} and
Proposition~\ref{prop:pair-moment-sharpness}.

\begin{lemma}[Radial data for circular arrays]
\label{lem:joint-radial-data}
Let $R_k>0$ satisfy $R_k=k+O(1)$. Let $\Lambda$ consist of
$2k$ points, counted with multiplicity, on each circle $|z|=R_k$,
with arbitrary angles. Set
\[
 T_n=\sum_{k\le n}2k\log(R_k/k).
\]
Then $n_\Lambda(r)=r^2+O(r)$ and
\begin{equation}
 D_\Lambda(r)=-\frac16\log r-T_{\lfloor r\rfloor}+O(1),
 \qquad T_{n+1}-T_n=O(1).
 \label{eq:variable-population-deficit}
\end{equation}
The number of labelled radii in an interval of length one is uniformly
bounded. Moreover, for every $c>0$,
\begin{equation}
 \sup_{r>0}\sum_{k\ge1}e^{-ck|\log(r/R_k)|}<\infty.
 \label{eq:variable-population-shell-overlap}
\end{equation}
If the points form regular $2k$-gons, they admit equal-area cells of
uniformly bounded diameter, as does any bounded displacement of this
configuration. In either case,
\begin{equation}
 |\Lambda(D(w,t))-t^2|\le C(t+1)\qquad(w\in\C,\ t>0).
 \label{eq:variable-population-disk-discrepancy}
\end{equation}
\end{lemma}

\begin{proof}
Choose $A$ such that $|R_k-k|\le A$ for every $k$. Then
\[
 \sum_{k\le(r-A)_+}2k\le n_\Lambda(r)
       \le\sum_{k\le r+A}2k.
\]
Since $\sum_{k\le n}2k=n(n+1)$, we have
$n_\Lambda(r)=r^2+O(r)$ without ordering the radii.
First take $R_k=k$ and write
$D_0(r)=\sum_{k\le r}2k\log(r/k)-r^2/2$. At integer radii,
\[
 D_0(n+1)-D_0(n)
 =n(n+1)\log(1+1/n)-n-\tfrac12
 =-\frac1{6n}+O(n^{-2}).
\]
Summation gives $D_0(n)=-\tfrac16\log n+O(1)$, since
$\sum n^{-2}<\infty$ and $\sum_{k<n}k^{-1}=\log n+O(1)$.
For $n<r<n+1$,
\[
 D'_0(r)=\frac{n(n+1)-r^2}{r}=O(1),
\]
so the same formula holds for real $r$. To compare the two
configurations, use $\log^+x=\max(\log x,0)$ and write
\[
 J_\Lambda(r)-J_0(r)+T_{\lfloor r\rfloor}
 =\sum_{k\ge1}2k\left(
    \log^+\frac r{R_k}-\log^+\frac r{k}
       +\boldsymbol1_{\{k\le r\}}\log\frac{R_k}{k}\right).
\]
Only indices with $r$ between $k$ and $R_k$ contribute. They
satisfy $|k-r|\le A$, so there are boundedly many. For large $r$,
each contribution is at most
$2k|\log(r/R_k)|+2k|\log(r/k)|=O(1)$; bounded radii contribute
a fixed error. This proves \eqref{eq:variable-population-deficit},
including the increment bound $2k\log(R_k/k)=O(1)$.

The bound on the number of radii in unit intervals also follows from
bounded radial displacement. When $R_k\in[r/2,2r]$,
$k|\log(r/R_k)|\ge c_1|k-r|-C_1$. For the remaining indices,
the summands are bounded by $e^{-c_2k}$. Summation gives
\eqref{eq:variable-population-shell-overlap}.

To construct the cells, set $\rho_k=\sqrt{k(k+1)}$, $\rho_0=0$,
and divide $\rho_{k-1}\le|z|<\rho_k$ into $2k$ equal sectors
whose central rays pass through the polygon points. Each sector has
area $\pi$. Its radial width, angular arc length, and distance from
its assigned point are uniformly bounded. The point can be attached
to the sector without changing its area. These properties persist
under bounded displacement. If $A$ bounds the distance of every
point of a cell from its assigned point, comparison of areas gives
$(t-A)_+^2\le\Lambda(D(w,t))\le(t+A)^2$. This proves the disk
estimate.
\end{proof}

For the regular-polygon arrays in Lemma~\ref{lem:joint-radial-data} and
their bounded displacements, Lemma~\ref{lem:bounded-area-allocation}
gives bounded matchings to $\mathcal L$. The angular motions below have
size at most $O((\log k)^2/k)$ and preserve this property.

\begin{lemma}[Sampling circular means]
\label{lem:gaussian-circle-sampling}
Let $r_n\in[n+1/4,n+3/4]$. For every entire function $f$,
allowing either side to be infinite, we have
\begin{equation}
 \int_{\C}|f(z)|^2 e^{-|z|^2}\,dA(z)
 \asymp\sum_{n\ge1}r_n e^{-r_n^2}
         \frac1{2\pi}\int_0^{2\pi}|f(r_ne^{i\theta})|^2d\theta.
 \label{eq:gaussian-circle-sampling}
\end{equation}
The comparison constants are independent of the choice of $r_n$.
\end{lemma}

The proof is given in Appendix~\ref{app:circle-sampling}.

\subsection{A critical direction forces polynomial quotients}

We begin with a discrete growth lemma. Its role is to treat functions
which may initially have infinitely many additional zeros.

\begin{lemma}[Polynomial bounds on a sampled ray]\label{lem:sampled-ray-polynomial}
Suppose that $Q$ is entire and
\begin{equation}
 \log\max_{|z|\le r}|Q(z)|\le C\log^2 r\qquad(r\ge2).
 \label{eq:quotient-log-squared-growth}
\end{equation}
If, for some $B\ge0$ and all sufficiently large integers $n$,
\begin{equation}
 |Q(n+1/2)|\le C_1(1+n)^B,
 \label{eq:quotient-ray-data}
\end{equation}
then $Q$ is a polynomial. The same assertion holds for samples
$(n+1/2)e^{i\theta_0}$ on any fixed ray.
\end{lemma}
The proof is given in Appendix~\ref{app:sampled-ray}.

The growth restriction in this lemma is essential to the argument.
For example, $\sin\sqrt z/\sqrt z$ is an entire function of order
$1/2$ bounded on the positive axis and is not a polynomial.
Discrete boundedness for broader classes of entire functions belongs
to the classical theory of P\'olya sequences; see \cite{MitkovskiPoltoratski2009}.
Lemma~\ref{lem:sampled-ray-polynomial} uses the stronger bound \eqref{eq:quotient-log-squared-growth}
and its proof in Appendix~\ref{app:sampled-ray} does not invoke that theory.

\begin{proposition}[One critical direction]\label{prop:critical-direction}
Let $G\not\equiv0$ be entire and $r_n=n+1/2$. Suppose that, for
all sufficiently large $n$,
\begin{align}
 \log|G(r_ne^{i\theta})|
   &\ge r_n^2/2-C\log^2 r_n\quad\text{for every }\theta,
       \label{eq:critical-direction-global}\\
 \log|G(r_ne^{i\theta_0})|
   &\ge r_n^2/2-M\log r_n\quad\text{for one fixed }\theta_0.
       \label{eq:critical-direction-ray}
\end{align}
Every $f\in\F$ vanishing on $\mathcal Z(G)$ has the form $f=PG$
with $P$ a polynomial. Consequently,
\[
 \mathcal Z(G)\text{ is a Fock zero set}\quad\Longleftrightarrow\quad
 G\in\F.
\]
If $G\notin\F$, its zero multiset is a uniqueness multiset.
\end{proposition}
\begin{proof}
For $0\ne f\in\F$ vanishing on $\mathcal Z(G)$, the quotient
$Q=f/G$ is entire, with all additional zeros retained.
The pointwise estimate \eqref{eq:fock-point-evaluation} and
\eqref{eq:critical-direction-global} give
$\log|Q(r_ne^{i\theta})|\le C'\log^2 r_n$.
The maximum principle, using a circle $r_n\in[r,r+1]$, implies
\eqref{eq:quotient-log-squared-growth} after increasing its constant.
Estimate \eqref{eq:critical-direction-ray} gives
$|Q(r_ne^{i\theta_0})|\le C''r_n^M$. Replace $M$ by a larger
nonnegative number if necessary. Lemma~\ref{lem:sampled-ray-polynomial} makes $Q$
a polynomial.

For any nonzero polynomial $P$, $|P(z)|\asymp |z|^{\deg P}$
outside a fixed disk, uniformly in the angle. Thus $PG\in\F$
implies $G\in\F$. If the zeros of $f$ are precisely those of $G$,
then $P$ has no zeros and is constant. These observations prove
both assertions, including uniqueness when $G\notin\F$.
\end{proof}

The proposition is independent of the particular circular construction.
The second estimate requires only one direction attaining the Gaussian
growth bound up to a polynomial factor. It replaces an estimate on
the radial mean when that mean has a deficit of order $\log^2 r$.

\subsection{Angular profiles and their endpoints}
\label{subsec:smooth-angular-profiles}

We now give a criterion in terms of the angular distribution of the
points. Two scales occur: a logarithmic angular potential and a squared
logarithmic one. The same realization estimate treats both, but their
critical endpoints are different.

Let $\tau\in\{1,2\}$, and let $h\in C^2(\mathbb R/2\pi\Z)$
be real valued with mean zero. Fix $a,b,c\in\mathbb R$, with $a=0$
when $\tau=1$. For all sufficiently large $k$, put
\begin{align}
 R_k&=k+\frac{a\log k}{k}+\frac b{k}+\frac c{k\log k},
 \qquad \theta_{k,\ell}=\phi_k+\frac{\pi\ell}{k},\notag\\
 \lambda_{k,\ell}
 &=R_k\exp\left(i\theta_{k,\ell}
       -i\frac{(\log k)^\tau}{2k^2}h'(\theta_{k,\ell})\right),
       \qquad 0\le\ell<2k.
 \label{eq:smooth-profile-configuration}
\end{align}
The phases $\phi_k$ are arbitrary. Choose the starting index
$K>e^2+1$ so that $|R_k-k|<1/4$ and the angular maps have derivative at least
$1/2$ for $k\ge K$. On the initial shells $k<K$, use regular
$2k$-gons of radius $k$. Denote the resulting set by $\Lambda$.
Changes in finitely many initial shells will be covered by finite
modifications.

Set
\[
 \gamma=\frac16+2b,\qquad \delta=2c,\qquad
 H=\max h,\qquad u_h(\theta)=H-h(\theta),
\]
and define
\begin{equation*}
 \nu_h(u)=\frac1{2\pi}
   \bigl|\{\theta\in[0,2\pi):u_h(\theta)\le u\}\bigr|,
 \qquad
 B_h(t)=\frac1{2\pi}\int_0^{2\pi}e^{-2t u_h(\theta)}\,d\theta.
 \end{equation*}
Thus $\nu_h(u)$ is the proportion of directions within $u$ of the
maximum. We say that $\mathcal E_\tau(h,\delta)$ holds if
\begin{equation}
 \begin{cases}
  \text{no additional condition},&\delta>\frac12,\\[2pt]
  \displaystyle\int_0^1\frac{\nu_h(u)}u\,du<\infty,
       &\delta=\frac12,\\[8pt]
  \displaystyle\int_0^1\nu_h(u)u^{(2\delta-1)/\tau-1}\,du<\infty,
       &\delta<\frac12.
 \end{cases}
 \label{eq:smooth-profile-geometric-endpoint}
\end{equation}
These nonnegative integrals include the case in which the maximum
set has positive measure.

\begin{theorem}[Geometric angular criterion]
\label{thm:smooth-angular-profile}
Let $\Lambda$ be given by \eqref{eq:smooth-profile-configuration},
and let $\Xi$ be a finite modification of net multiplicity $q$.
Then $\Xi$ is a zero set for $\F$ if and only if
\begin{equation}
 \int_3^\infty r^{2q+1-2\gamma}(\log r)^{-2\delta}
   e^{-2a(\log r)^2+2H(\log r)^\tau}
   B_h((\log r)^\tau)\,dr<\infty.
 \label{eq:smooth-profile-integral}
\end{equation}
Equivalently, the parameters satisfy the following conditions.
\begin{enumerate}[label=\textup{(\roman*)},leftmargin=2.3em]
\item If $\tau=1$, then $\gamma-H>q+1$, or
      $\gamma-H=q+1$ and $\mathcal E_1(h,\delta)$ holds.
\item If $\tau=2$, then $a>H$, or $a=H$ and either
      $\gamma>q+1$ or
      $\gamma=q+1$ with $\mathcal E_2(h,\delta)$.
\end{enumerate}
In every other case $\Xi$ is a uniqueness multiset. The conclusions
hold for every choice of the polygon phases and every finite
modification of the stated net multiplicity.
\end{theorem}

For $\tau=2$, the two critical equalities are $a=H$ and
$\gamma=q+1$. At their intersection, the $\log\log r$ correction
and the width of the angular maxima decide the zero-set property.
For $\tau=1$, the single critical equality is $\gamma-H=q+1$.
No isolation or finite order of the maxima is assumed. For example,
if $\nu_h(0)>0$, the endpoint condition is simply $\delta>1/2$
at either scale.

At the critical equalities just described, the following special cases
summarize the endpoint condition. For the first two rows, assume finitely
many maxima, all of the indicated order; all inequalities are strict.
\begin{center}
\begin{tabular}{@{}lll@{}}
\toprule
Shape of the maximum & $\tau=1$ & $\tau=2$\\
\midrule
Quadratic & $\delta>1/4$ & $\delta>0$\\
Quartic & $\delta>3/8$ & $\delta>1/4$\\
Maximum set of positive measure & $\delta>1/2$ & $\delta>1/2$\\
\bottomrule
\end{tabular}
\end{center}
See Example~\ref{ex:finite-order-angular-maxima} for finite-order maxima.
Example~\ref{ex:smooth-flat-maxima} shows why the full sublevel distribution
is needed for infinitely flat maxima.

These configurations are uniformly separated bounded displacements
of the critical lattice. We record the geometry together with the
potential estimate, so that the criterion is derived directly from
the points.

For the angular estimates, put $\varepsilon_k=(\log k)^\tau/(2k^2)$
for $k\ge K$ and $\varepsilon_k=0$ for $k<K$, and set $r_n=n+1/2$.
All constants may depend on $h,a,b,c,K,\tau$, but not on $n$, the
observation angle, or the polygon phases. In the radial kernel estimate
they are also independent of the positive integer frequency.

\begin{lemma}[Geometry, radial deficit, and the shell product]
\label{lem:profile-geometry-product}
The configuration in \eqref{eq:smooth-profile-configuration} is uniformly
separated, admits a bounded matching to $\mathcal L$, and satisfies
\eqref{eq:variable-population-disk-discrepancy}. Its shell product
\[
 F(z)=\prod_{k\ge1}\prod_{\ell=0}^{2k-1}(1-z/\lambda_{k,\ell})
\]
converges locally uniformly, has zero multiset exactly $\Lambda$, and satisfies
\begin{equation}
 D_\Lambda(r)=-a(\log r)^2-\gamma\log r-\delta\log\log r+O(1).
 \label{eq:smooth-profile-radial-deficit}
\end{equation}
\end{lemma}
\begin{proof}

The angular maps and all their intermediate deformations have derivative
at least $1/2$.
Since $R_k\asymp k$, their angular gaps give a uniform positive
lower bound for chord lengths within a shell. The bound
$|R_k-k|<1/4$ gives a gap greater than $1/2$ between consecutive
tail radii. The first tail radius is separated from the initial
shells by at least $3/4$.
This proves uniform separation. The tangential motions have size
$O((\log k)^\tau/k)$ and hence are bounded. The allocation,
matching, and disk discrepancy follow from
Lemmas~\ref{lem:joint-radial-data} and \ref{lem:bounded-area-allocation}.

For the radial term, expand
\[
 2k\log(R_k/k)=\frac{2a\log k}{k}+\frac{2b}{k}
       +\frac{2c}{k\log k}
       +O\!\left(\frac{(\log k)^2}{k^3}\right).
\]
The error is summable. Sum-integral comparison gives
\[
 \sum_{k=K}^n2k\log(R_k/k)
       =a(\log n)^2+2b\log n+2c\log\log n+O(1).
\]
Indeed, the derivatives of $(\log x)/x$, $1/x$, and $1/(x\log x)$
are absolutely integrable on $[K,\infty)$, so each comparison
has a bounded error. Lemma~\ref{lem:joint-radial-data} now proves
\eqref{eq:smooth-profile-radial-deficit}.

Write $\psi_{k,\ell}=\theta_{k,\ell}
                  -\varepsilon_kh'(\theta_{k,\ell})$.
For $j=1,2$, the reciprocal angular sums on a large shell obey
\[
 \left|\sum_{\ell=0}^{2k-1}e^{-ij\psi_{k,\ell}}\right|
 \le2kj\varepsilon_k\|h'\|_\infty=O((\log k)^\tau/k),
\]
since the corresponding regular polygon sum is zero.
For $|z|\le T$ and $R_k>2T$, expand the logarithms with value zero
at $z=0$. The terms of degrees one and two are summable over $k$,
because they are bounded by $C_T(\log k)^\tau/k^2$ and
$C_T(\log k)^\tau/k^3$. The remaining terms have sum at most
\[
 2k\sum_{j\ge3}\frac{(T/R_k)^j}{j}\le C_T k/R_k^3,
\]
also summable. This proves local uniform convergence of the logarithmic
tails, whose exponentials do not vanish. The finite remaining factors
give exactly the stated zeros.

\end{proof}

\begin{lemma}[Angular variation and sampling error]
\label{lem:profile-angular-sampling}
Let $F$ be the product from Lemma~\ref{lem:profile-geometry-product}, let
$q_k=\min(r_n/R_k,R_k/r_n)$, and put $L_q(x)=\log|1-qe^{ix}|$.
Then, uniformly in $n$, $\theta$, and the phases,
\[
 \log|F(r_ne^{i\theta})|-J_\Lambda(r_n)
 =-\sum_{k\ge K}\frac{2k\varepsilon_k}{2\pi}
      \int_0^{2\pi}h'(x)L'_{q_k}(x-\theta)\,dx+O(1).
\]
The Taylor and discrete sampling errors are bounded shellwise by
\eqref{eq:profile-taylor-error} and \eqref{eq:profile-sampling-error};
their sums are uniformly bounded.
\end{lemma}
The proof is given in Appendix~\ref{app:angular-sampling}.

\begin{lemma}[Uniform radial kernel summation]
\label{lem:profile-radial-kernel}
For $R_k$ and $K$ as above,
\begin{equation}
 T_m(r):=\sum_{k\ge K}\frac{(\log k)^\tau}{k}
                 e^{-m|\log(r/R_k)|}
 =\frac{2(\log r)^\tau}{m}+O(1),
 \label{eq:smooth-profile-radial-convolution}
\end{equation}
uniformly for integers $m\ge1$ and large $r$. For each fixed integer
$N\ge1$, one also has
\begin{equation}
 N\sum_{\substack{k\ge K\\N\mid k}}
       \frac{(\log k)^\tau}{k}e^{-m|\log(r/R_k)|}
       =\frac{2(\log r)^\tau}{m}+O_N(1),
 \label{eq:profile-selected-shells}
\end{equation}
with error independent of $m$ and $r$.
\end{lemma}
The proof is given in Appendix~\ref{app:radial-kernel}.

\begin{lemma}[Realization of angular profiles]
\label{lem:angular-profile-realization}
Let $F$ be the shell product from Lemma~\ref{lem:profile-geometry-product}.
For all sufficiently large $n$, with $r_n=n+1/2$,
\begin{equation}
 \log|F(r_ne^{i\theta})|
       =J_\Lambda(r_n)+h(\theta)(\log r_n)^\tau+O(1).
 \label{eq:smooth-profile-potential}
\end{equation}
The error is uniform in $n$, $\theta$, and the phases $\phi_k$.
\end{lemma}

\begin{proof}
By Lemma~\ref{lem:profile-angular-sampling}, we need only sum the continuous
first variations. Write
$h(\theta)=\sum_{m\ne0}c_me^{im\theta}$ and use
Lemma~\ref{lem:profile-radial-kernel}, with $r=r_n$.
The coefficient of $c_m e^{im\theta}$ in the summed continuous
variations is $|m|T_{|m|}(r)/2=(\log r)^\tau+O(|m|)$.
The sums can be interchanged since
\[
 \sum_{m\ne0}|mc_m|T_{|m|}(r)
 \le2(\log r)^\tau\sum_{m\ne0}|c_m|+C\sum_{m\ne0}|mc_m|<\infty.
\]
By \eqref{eq:smooth-profile-fourier-tail}, the coefficient errors have
a bounded sum. Thus the variations equal $h(\theta)(\log r)^\tau+O(1)$.
Combining all three bounded errors with the regular-polygon potential
proves \eqref{eq:smooth-profile-potential}.
\end{proof}

\begin{proof}[Proof of Theorem~\ref{thm:smooth-angular-profile}]
Let $F$ be the function in Lemma~\ref{lem:angular-profile-realization}
and put $G=Fq_{\nu_+}/q_{\nu_-}$, with its poles canceled. Thus
$\mathcal Z(G)=\Xi$. The finite-change factor has logarithmic modulus
$q\log r+O(1)$ outside a fixed disk, and hence
\begin{equation}
 \begin{split}
 \log|G(r_ne^{i\theta})|={}&\frac{r_n^2}{2}
       -a(\log r_n)^2+h(\theta)(\log r_n)^\tau\\
       &-(\gamma-q)\log r_n-\delta\log\log r_n+O(1).
 \end{split}
 \label{eq:angular-full-expansion}
\end{equation}
All errors in this proof are uniform in the angle. The proof has two
separate tasks: compute the norm of this prescribed generator, and then
show that a divergent norm cannot be repaired by an entire multiplier
introducing additional zeros.

\emph{Step 1: integrability of the prescribed generator.} We determine when $G$ belongs to $\F$.
Lemma~\ref{lem:gaussian-circle-sampling} makes this equivalent to
summability of the integrand in \eqref{eq:smooth-profile-integral}
at the radii $r_n$. All its nonangular factors have bounded ratios
on $[n,n+1]$, since their logarithms change by
$O((1+\log n)/n)$. Also, boundedness of $u_h$ gives
\[
 e^{-2\|u_h\|_\infty|s-t|}
 \le\frac{B_h(s)}{B_h(t)}
 \le e^{2\|u_h\|_\infty|s-t|}.
\]
The change of $(\log r)^\tau$ on a unit interval is bounded.
This proves the equivalence between $G\in\F$ and
\eqref{eq:smooth-profile-integral}.

At a maximum $\theta_0$, Taylor's theorem gives
$0\le u_h(\theta_0+x)\le Cx^2$ for small $x$.
Integration over $|x|\le c t^{-\tau/2}$ therefore yields
\begin{equation}
 c_h t^{-\tau/2}\le B_h(t^\tau)\le1\qquad(t\ge1).
 \label{eq:angular-concentration-bounds}
\end{equation}
For a constant profile, the same inequality follows from $B_h=1$.
Under $t=\log r$, the integral in
\eqref{eq:smooth-profile-integral} becomes
\[
 \int_{\log3}^\infty
 e^{-2at^2+2Ht^\tau-2(\gamma-q-1)t}
                  t^{-2\delta}B_h(t^\tau)\,dt.
\]
For $\tau=1$, recall that $a=0$. The bounds in
\eqref{eq:angular-concentration-bounds} give convergence if
$\gamma-H>q+1$ and divergence if $\gamma-H<q+1$.
For $\tau=2$, they give convergence if $a>H$ and divergence
if $a<H$. When $a=H$, the sign of $\gamma-q-1$ decides in
the same way. At the respective critical equalities, the remaining
condition is
\begin{equation}
 \int_{\log3}^\infty t^{-2\delta}B_h(t^\tau)\,dt<\infty.
 \label{eq:angular-endpoint-integral}
\end{equation}

By Tonelli's theorem, this integral is the angular mean of
$K_{\tau,\delta}(u_h(\theta))$, where
\[
 K_{\tau,\delta}(u)=\int_{\log3}^\infty
                         t^{-2\delta}e^{-2ut^\tau}\,dt.
\]
If $\delta>1/2$, this kernel is bounded even at $u=0$.
If $\delta=1/2$, split the integral at $u^{-1/\tau}$.
The exponential is between $e^{-2}$ and $1$ below that point,
and substitution bounds the tail independently of small $u>0$.
Consequently $K_{\tau,\delta}(u)\asymp1+\log(1/u)$ as
$u\downarrow0$. If $\delta<1/2$, put $p=(1-2\delta)/\tau>0$.
The substitution $v=2ut^\tau$ gives
\[
 K_{\tau,\delta}(u)=\frac{(2u)^{-p}}\tau
       \int_{2u(\log3)^\tau}^\infty v^{p-1}e^{-v}\,dv
       \sim\frac{\Gamma(p)}\tau(2u)^{-p}.
\]
For $\delta\le1/2$, the value at zero is infinite; away from zero
the kernel is bounded. The layer-cake identities are
\begin{align*}
 \frac1{2\pi}\int_0^{2\pi}
            \log^+\frac1{u_h(\theta)}\,d\theta
     &=\int_0^1\frac{\nu_h(u)}u\,du,\\
 \frac1{2\pi}\int_{\{u_h<1\}}
            (u_h(\theta)^{-p}-1)\,d\theta
     &=p\int_0^1\nu_h(u)u^{-p-1}\,du.
\end{align*}
They follow by writing each nonnegative integrand as its integral
over the sublevel parameter and applying Tonelli. They remain valid
with infinite values at $u_h=0$. Strict and non-strict sublevel sets
give the same integrals in $u$. This proves that
\eqref{eq:angular-endpoint-integral} is exactly
$\mathcal E_\tau(h,\delta)$.

\emph{Step 2: exclusion of all other vanishing functions.} Failure of these tests gives uniqueness,
including for functions with additional zeros. When $\tau=1$,
\eqref{eq:angular-full-expansion} is $r_n^2/2+O(\log r_n)$
in every direction. When $\tau=2$ and $a\le H$, it is bounded
below by $r_n^2/2-C(\log r_n)^2$ in every direction and by
$r_n^2/2-M\log r_n$ at a fixed maximum of $h$. Indeed, at that
maximum the quadratic logarithmic term is $(H-a)(\log r_n)^2\ge0$,
and the remaining $\log r_n$ and $\log\log r_n$ terms are bounded
below by $-M\log r_n$ after increasing $M$.
Proposition~\ref{prop:critical-direction} applies in both cases:
every vanishing Fock function is a polynomial multiple of $G$,
and one exists exactly when $G\in\F$.
For $\tau=2$ and $a>H$, the norm test already gives $G\in\F$.
The stated alternatives now follow.
\end{proof}

\subsection{A stable zero set at critical density}

The preceding criteria give a simple example of the class in
Theorem~\ref{thm:finite-modification-stability}.

\begin{example}\label{ex:stable-critical-array}
Let $R_k=k+(\log k)/k$ and put
\[
 \Lambda=\bigcup_{k\ge1}\{R_ke^{\pi i\ell/k}:0\le\ell<2k\}.
\]
This set is uniformly separated, admits a bounded matching to
$\mathcal L$, and satisfies $n_\Lambda(r)=r^2+O(r)$. Every finite
modification is a Fock zero set. The product
\[
 F(z)=\prod_{k\ge1}\left(1-(z/R_k)^{2k}\right)
\]
has zero set $\Lambda$ and $z^qF\in\F$ for every integer $q\ge0$.
\end{example}
\begin{proof}
The inequalities $0\le R_k-k\le1/e<3/8$ separate consecutive radii
by at least $5/8$. Chords on the $k$th circle have minimum length
$2R_k\sin(\pi/(2k))\ge2$. The count and bounded allocation follow
from Lemma~\ref{lem:joint-radial-data}; the matching follows from
Lemma~\ref{lem:bounded-area-allocation}. The product converges normally
on compact sets because $\sum_k(T/R_k)^{2k}<\infty$ for fixed $T$.

Apply Lemma~\ref{lem:angular-profile-realization} with $h=0$, $\tau=2$,
$a=1$, and $b=c=0$. Changing its finitely many initial shells to the
radii $R_k$ is a finite modification of net multiplicity zero. The ratio
of the two generators is a rational function tending to a nonzero
constant at infinity. Hence, uniformly in $\theta$ and for
$r_n=n+1/2$ sufficiently large,
\begin{equation*}
 \log|F(r_ne^{i\theta})|
 =\frac{r_n^2}{2}-(\log r_n)^2-\frac16\log r_n+O(1).
 \end{equation*}
Lemma~\ref{lem:gaussian-circle-sampling} therefore gives
\[
 z^qF\in\F\quad\Longleftrightarrow\quad
 \sum_{n\ge3}n^{2q+2/3}e^{-2(\log n)^2}<\infty.
\]
The series converges for every $q$, since its terms eventually lie
below $n^{-2}$. Proposition~\ref{prop:finite-modification-moments}
proves stability. This also shows that a strict deficit in the leading
radial density is not necessary for finite-modification stability.
\end{proof}

\section{Transfer principles and local moment stability}
\label{sec:transfer}
\label{sec:further-perturbations}
\label{sec:irregular-perturbations}

We give two ways to transfer zero-set criteria. Small motions produce
a radial drift; bounded local replacements with two matching moments
preserve the zero-set property. These results extend the preceding
criteria to configurations that need not retain a radial or circular
form. The local moments here concern replacement clusters, whereas
Proposition~\ref{prop:balanced-five-points} balances the test measure.

\subsection{Inverse-radius displacements}

\begin{theorem}[Inverse-radius displacements]
\label{thm:inverse-radius-transport}
Let $\Lambda=(\lambda_a)_{a\in A}$ be a discrete multiset, and let
$\Gamma=(\mu_a)_{a\in A}$ be a labelled replacement. Suppose there
is a finite set of indices $A_0$ such that the positive measure
\begin{equation}
 \nu=\sum_{a\notin A_0}(1+|\lambda_a|)|\mu_a-\lambda_a|
                                      \delta_{\lambda_a}
 \quad\hbox{satisfies}\quad
 \sup_{w\in\C}\nu(D(w,1))<\infty.
 \label{eq:irt-displacement-mass}
\end{equation}
Enlarge $A_0$ so that $\lambda_a\mu_a\ne0$ and
$|\mu_a-\lambda_a|\le |\lambda_a|/2$ for $a\notin A_0$.
Define
\begin{equation*}
 B(r)=\sum_{\substack{a\notin A_0\\ |\lambda_a|\le r}}
                    \log\left|\frac{\mu_a}{\lambda_a}\right|.
 \end{equation*}
Then $\Gamma$ is discrete and
\[
 \begin{gathered}
 \Gamma\text{ is a zero set for }\F
 \quad\Longleftrightarrow\\
 \exists f\in \mathcal O(\C)\setminus\{0\}:\quad\mathcal Z(f)=\Lambda,
 \quad\int_\C |f(z)|^2
                    e^{-|z|^2-2B(|z|)}\,dA(z)<\infty.
 \end{gathered}
\]
\end{theorem}

\begin{proof}
Consider the meromorphic product
\begin{equation*}
 Q(z)=\prod_{a\in A_0}\frac{z-\mu_a}{z-\lambda_a}
       \prod_{a\notin A_0}
       \frac{1-z/\mu_a}{1-z/\lambda_a}.
 \end{equation*}
We prove that $Q$ converges independently of order and that, for every
entire $f$ vanishing on $\Lambda$,
\begin{equation}
 \|Qf\|_2^2\asymp
 \frac1\pi\int_\C |f(z)|^2
                 e^{-|z|^2-2B(|z|)}\dd(z).
 \label{eq:irt-norm}
\end{equation}
The constants are independent of $f$, and infinite values are allowed.
Multiplication by $Q$ will identify the two vanishing classes through
\begin{equation}
 \mathcal Z(Qf)=\mathcal Z(f)-\Lambda+\Gamma\qquad(f\ne0).
 \label{eq:irt-zero-multiset}
\end{equation}

Put $d_a=(1+|\lambda_a|)|\mu_a-\lambda_a|$ for $a\notin A_0$.
Condition~\eqref{eq:irt-displacement-mass} gives $d_a\le M$, so
$|\mu_a-\lambda_a|\le M/(1+|\lambda_a|)$ and
$L=\sup_a|\mu_a-\lambda_a|<\infty$.
Thus $\Gamma$ is discrete. A covering by unit disks gives
\begin{equation}
 \nu(D(w,t))\le C_0(1+t)^2
 \qquad(w\in\C,\ t\ge0).
 \label{eq:irt-quadratic-mass}
\end{equation}
For $a\notin A_0$, put $b_a=\log|\mu_a/\lambda_a|$. If
$|\lambda_a|$ is sufficiently large, then
\begin{equation}
 |b_a|\le C_1d_a|\lambda_a|^{-2}.
 \label{eq:irt-b-bound}
\end{equation}
For $|z|\le R$ and $|\lambda_a|$ sufficiently large, the segment
from $\lambda_a$ to $\mu_a$ stays outside $D(0,2R)$. On this segment,
\[
 \frac{d}{dw}\log(1-z/w)=\frac{z}{w(w-z)},\qquad
 \left|\log\frac{1-z/\mu_a}{1-z/\lambda_a}\right|
 \le C_R d_a|\lambda_a|^{-3},
\]
where the logarithm vanishes at $z=0$. For $R_1\ge1$,
\eqref{eq:irt-quadratic-mass} gives
\[
 \int_{|w|>R_1}|w|^{-3}\,d\nu(w)
 \le\sum_{j\ge0}(2^jR_1)^{-3}\nu(D(0,2^{j+1}R_1))
 \le\frac C{R_1}\sum_{j\ge0}2^{-j}.
\]
The logarithmic tails converge absolutely and normally and have
nonvanishing exponentials. Hence $Q$ has zero and pole orders given by $\Gamma-\Lambda$,
with cancellation at coincident points. This proves
\eqref{eq:irt-zero-multiset}.

It also follows from \eqref{eq:irt-b-bound} that
$B(r)=O(\log(2+r))$ and that, for each fixed $h>0$,
\begin{equation}
 \sup_{\substack{r,r'\ge0\\ |r-r'|\le h}}
                   |B(r)-B(r')|<\infty.
 \label{eq:irt-local-drift}
\end{equation}
Indeed, outside a fixed disk the contribution of each annulus
$2^jR_0<|w|\le2^{j+1}R_0$ is bounded by
\[
 C(2^jR_0)^{-2}\nu(D(0,2^{j+1}R_0))\le C'.
\]
There are $O(\log(2+r))$ such annuli up to radius $r$.
If $|r-r'|\le h$ and $r>2h+2R_0$, the difference is supported
on $r-h<|w|\le r+h$. This annulus is covered by $O_h(r)$ unit
disks, whence
\[
 |B(r)-B(r')|
 \le Cr^{-2}\nu\{w:r-h<|w|\le r+h\}\le C_h/r.
\]
Bounded radii involve only finitely many terms, proving both claims.

Set $T=2(L+1)$ and $h=4(L+1)$, so that $h>T+L$.
For a nonzero meromorphic function $F$, write
\[
 S_h\log|F|(z)=\frac1{2\pi}\int_0^{2\pi}
                          \log|F(z+he^{it})|\,dt.
\]
The required potential estimate is
\begin{equation}
 \big|S_h\log|Q|(z)+B(|z|)\big|\le C_2.
 \label{eq:irt-mean-potential}
\end{equation}
For $r=|z|$, put
\[
 \ell_a^h(z)=\log\max(h,|z-\mu_a|)
             -\log\max(h,|z-\lambda_a|).
\]
The circular-mean formula for the logarithm gives
\begin{align*}
 S_h\log|Q|(z)+B(r)
 &=\sum_{a\in A_0}\ell_a^h(z)
   +\sum_{\substack{a\notin A_0\\ |\lambda_a|\le r}}\ell_a^h(z)
   \\
 &\quad+\sum_{\substack{a\notin A_0\\ |\lambda_a|>r}}
                              (\ell_a^h(z)-b_a).
 \end{align*}
Termwise integration is valid by uniform convergence of the tail;
the remaining finite logarithmic singularities are integrable.

On compact sets, the finite terms and the locally uniformly convergent
tail are bounded. Suppose $r$ is large. When $|z-\lambda_a|\le T$,
both distances are below $h$, so $\ell_a^h(z)=0$. The residual
terms $-b_a$ are bounded by \eqref{eq:irt-b-bound} and the local
mass bound. The finite set $A_0$ also contributes a bounded amount.

For every other index, replacing each logarithm by its maximum with
$\log h$ does not increase the absolute difference. Since
$|z-\lambda_a|>T>2L$, we obtain
\[
 |\ell_a^h(z)|\le
 \frac{2|\mu_a-\lambda_a|}{|z-\lambda_a|}.
\]
Divide these indices into three regions. The contribution from
$|\lambda_a|<r/2$ is at most
\[
 \frac C r\int_{R_0\le|w|<r/2}\frac{d\nu(w)}{|w|}+O(1)
 \le C',
\]
where $R_0>0$ is fixed. For $r/2\le|\lambda_a|\le2r$, the bound is
\[
 \frac C r\int_{T<|z-w|\le3r}\frac{d\nu(w)}{|z-w|}
 +\frac C{r^2}\nu(D(0,2r))\le C'.
\]
Here the second term bounds the terms $b_a$ with $|\lambda_a|>r$.
For $|\lambda_a|>2r$ and sufficiently large $r$, neither distance
is capped. On the segment from $\lambda_a$ to $\mu_a$,
$|w|\ge3|\lambda_a|/4$ and $|w-z|\ge|\lambda_a|/4$.
Differentiation of $\log(1-z/w)$ along this segment bounds the
contribution from the outer region by
\[
 Cr\int_{|w|>2r}|w|^{-3}\,d\nu(w)\le C'.
\]
The estimates used here follow from the following dyadic sums,
uniformly in $v$ and for $R\ge a\ge1$:
\[
 \begin{split}
 \int_{a<|w-v|\le R}|w-v|^{-1}\,d\nu(w)
 &\le \sum_{\substack{j\ge0\\2^ja<R}}
        (2^ja)^{-1}\nu(D(v,2^{j+1}a))\le CR,\\
 \int_{|w-v|>R}|w-v|^{-3}\,d\nu(w)
 &\le \sum_{j\ge0}(2^jR)^{-3}\nu(D(v,2^{j+1}R))\le C/R.
 \end{split}
\]
In the first sum the last annulus may be truncated at $R$;
the estimate follows from $\sum_{2^ja<R}2^ja\le2R$.
The outer-region bound is the second estimate multiplied by $r$.
These estimates prove \eqref{eq:irt-mean-potential}.

Now let $f$ be entire and vanish on $\Lambda$, and put $g=Qf$.
The product is entire. If $f\ne0$, subharmonicity and
\eqref{eq:irt-mean-potential} give
\begin{align}
 \log|g(z)|&\le S_h\log|f|(z)-B(|z|)+C_2,
       \label{eq:irt-capped}\\
 \log|f(z)|&\le S_h\log|g|(z)+B(|z|)+C_2.
       \label{eq:irt-reverse-capped}
\end{align}
The mean of $|z+he^{it}|^2$ is $|z|^2+h^2$. Including the
Gaussian in \eqref{eq:irt-capped} and applying Jensen gives
\begin{equation}
 |g(z)|^2e^{-|z|^2}
 \le C_3e^{-2B(|z|)}\frac1{2\pi}\int_0^{2\pi}
             |f(z+he^{it})|^2e^{-|z+he^{it}|^2}\,dt.
 \label{eq:irt-gaussian}
\end{equation}
Put $W(z)=e^{-2B(|z|)}$ and $F(z)=|f(z)|^2e^{-|z|^2}$.
The constant in \eqref{eq:irt-gaussian} includes $e^{h^2}$.
By \eqref{eq:irt-local-drift}, $W(z)\le C_hW(z+he^{it})$.
Translation invariance and Tonelli's theorem therefore give
\[
 \begin{split}
 \int_\C |g(z)|^2e^{-|z|^2}\,dA(z)
 &\le \frac{C_3}{2\pi}\int_0^{2\pi}
       \int_\C W(z)F(z+he^{it})\,dA(z)\,dt\\
 &\le C_3C_h\int_\C W(w)F(w)\,dA(w).
 \end{split}
\]
For the reverse bound, apply the same Gaussian calculation to
\eqref{eq:irt-reverse-capped} before multiplying by $W(z)$.
The two drift factors cancel at the center, giving
\[
 |f(z)|^2e^{-|z|^2}W(z)
 \le e^{h^2+2C_2}\frac1{2\pi}\int_0^{2\pi}
       |g(z+he^{it})|^2e^{-|z+he^{it}|^2}\,dt.
\]
A further integration proves the other half of \eqref{eq:irt-norm}.
All integrands are nonnegative, so both applications of Tonelli
remain valid for infinite norms.

Conversely, if $g$ is entire and vanishes on $\Gamma$, the zero and pole orders
of $Q$ make $g/Q$ entire and vanishing on $\Lambda$. Multiplicities
are counted, so collisions between labels cause no change. The reverse
estimate uses the same $Q$ and drift and requires no mass condition
on the target. Thus multiplication by $Q$ is a bijection, and the
zero-set equivalence follows.

\end{proof}

\begin{corollary}[Radial and angular displacements]
\label{cor:irt-bounded-drift}
\label{cor:radial-weight-shift}
\label{cor:irt-angular}
Under the hypotheses of Theorem~\ref{thm:inverse-radius-transport},
if $B(r)=\gamma\log(1+r)+O(1)$, then
\[
 \begin{gathered}
 \Gamma\text{ is a Fock zero set}
 \quad\Longleftrightarrow\\
 \exists f\in \mathcal O(\C)\setminus\{0\}:\quad\mathcal Z(f)=\Lambda,
 \quad\int_\C|f(z)|^2(1+|z|^2)^{-\gamma}e^{-|z|^2}\,dA(z)<\infty.
 \end{gathered}
\]
In particular, bounded drift preserves Fock zero sets. The hypotheses
hold in each of the following cases.
\begin{enumerate}[label=\textup{(\roman*)}]
\item If $\Lambda$ has uniformly bounded counting in unit disks,
$|\mu_a-\lambda_a|\le C/(1+|\lambda_a|)$ outside a finite set,
and $|\mu_a|=|\lambda_a|$ there, then $\gamma=0$.
Here $C$ need not be small, and no rotational symmetry or moment
identity is assumed.
\item If $\Lambda$ has uniformly bounded counting in unit disks,
$n_\Lambda(r)=\rho r^2+O(r)$ with $\rho\ge0$, and, outside a
sufficiently large disk,
\[
 \mu=\lambda\left(1+\frac c{|\lambda|^2}\right),\qquad c\in\mathbb R,
\]
then $\gamma=2c\rho$.
\end{enumerate}
\end{corollary}

\begin{proof}
The equivalence follows from $e^{-2B(r)}\asymp(1+r^2)^{-\gamma}$
and Theorem~\ref{thm:inverse-radius-transport}.
Bounded local counting gives \eqref{eq:irt-displacement-mass} in
both cases. In \textup{(i)}, the drift is zero. In \textup{(ii)},
choose $R$ with $1+c/|\lambda|^2>0$ for $|\lambda|>R$ and use
\[
 \log(1+c/|\lambda|^2)=c|\lambda|^{-2}+O(|\lambda|^{-4}).
\]
Since $n_\Lambda(t)=O(t^2)$, the sum of the error terms is finite.
Stieltjes integration by parts gives
\[
 \begin{split}
 \sum_{R<|\lambda|\le r}|\lambda|^{-2}
 &=r^{-2}n_\Lambda(r)-R^{-2}n_\Lambda(R)
      +2\int_R^r\frac{n_\Lambda(t)}{t^3}\,dt\\
 &=2\rho\log r+O(1),
 \end{split}
\]
because the integral of the $O(t)$ remainder divided by $t^3$
converges. Thus $B(r)=2c\rho\log r+O(1)$, as asserted.
\end{proof}

\begin{remark}
Changing $A_0$ multiplies $Q$ by a nonzero constant and changes $B$
by a bounded step function, leaving the conclusions unchanged.
Condition~\eqref{eq:irt-displacement-mass} allows unbounded local
multiplicities and arbitrarily many stationary points.
\end{remark}

\subsection{Perturbations of the critical lattice}
\label{sec:lattice-criteria}

\begin{corollary}[Inverse-radius perturbations of the lattice]
\label{thm:lattice-drift-criterion}
Suppose $\Gamma=(\mu_\lambda)_{\lambda\in\mathcal L}$ satisfies
$|\mu_\lambda-\lambda|\le C/(1+|\lambda|)$ outside a finite
set $A_0$. Enlarge $A_0$, if necessary, so that the following
quotients are nonzero, and put
\[
 B(r)=\sum_{\substack{\lambda\notin A_0\\|\lambda|\le r}}
                       \log|\mu_\lambda/\lambda|.
\]
For any finite modification $\Xi$ of net multiplicity $q$,
\begin{equation*}
 \Xi\text{ is a zero set for }\F
 \quad\Longleftrightarrow\quad
 \int_1^\infty r^{2q+1}e^{-2B(r)}\,dr<\infty.
 \end{equation*}
Equivalently, the condition is $\int_1^\infty re^{2D_\Xi(r)}dr<\infty$.
If the integral diverges, $\Xi$ is a uniqueness multiset.
\end{corollary}

\begin{proof}
The lattice satisfies \eqref{eq:irt-displacement-mass}. Let $Q$ be
the product of Theorem~\ref{thm:inverse-radius-transport}, and set
$G_0=Q\sigma$. Equations \eqref{eq:irt-mean-potential} and
\eqref{eq:lattice-classical-sigma} give
\[
 S_h\log|G_0|(z)-|z|^2/2=-B(|z|)+O(1).
\]
The uniform lattice mean bound was proved in
Lemma~\ref{lem:bounded-displacement-potential}. For the finite change,
put $G=G_0P_+/P_-$, where $P_\pm$ represent the added and deleted
points. Then $G$ is entire with zero multiset $\Xi$.
Since $\log\max(h,|z-a|)=\tfrac12\log(1+|z|^2)+O(1)$ for
fixed $a$, we obtain
\[
 e^{2S_h\log|G|(z)-|z|^2}
 \asymp (1+|z|^2)^q e^{-2B(|z|)}.
\]
The right side is bounded above and below on compact sets.
Polar coordinates and \eqref{eq:smoothed-fock-norm} therefore give
$G\in\F$ exactly when the stated radial integral converges.

Since $\Gamma$ is a bounded displacement of the lattice,
\eqref{eq:bounded-radial-jensen-drift} applies. Together with
\eqref{eq:finite-change-jensen}, it gives
\[
 D_\Gamma(r)=-B(r)+O(1),\qquad
 D_\Xi(r)=-B(r)+q\log r+O(1).
\]
This proves the equivalent criterion in terms of $D_\Xi$.
By \eqref{eq:jensen-radial-norm-lower}, divergence excludes every
nonzero Fock function vanishing on $\Xi$. This proves necessity and
the uniqueness assertion.
\end{proof}

For $B(r)=\gamma\log r+\delta\log\log r+O(1)$, the criterion
reduces to $\gamma>q+1$, or $\gamma=q+1$ and $\delta>1/2$.
The integral criterion requires no asymptotic expansion of the drift.

\Needspace{6\baselineskip}
\subsection{Local moment conditions}
\label{sec:cluster-deformations}

Two matching moments make the logarithmic ratio of a replacement
begin with a cubic term. Summability of this remainder gives the
following preservation result.

\begin{theorem}[Two local moment identities]
\label{thm:general-cluster-isomorphism}
Let $\Lambda$ be a discrete multiset partitioned, with multiplicity,
into finite nonempty clusters $\Gamma_\alpha$. Let $c_\alpha$
be centers, let $0\le r_\alpha\le H<\infty$, and let
$\Gamma'_\alpha$ be finite multisets such that
\begin{gather*}
 \Gamma_\alpha\cup\Gamma'_\alpha
 \subset\{z:|z-c_\alpha|\le r_\alpha\},\qquad
 m_\alpha:=|\Gamma_\alpha|=|\Gamma'_\alpha|,\\
 \sum_{\lambda\in\Gamma_\alpha}\lambda^p
 =\sum_{\mu\in\Gamma'_\alpha}\mu^p\qquad(p=1,2).
\end{gather*}
Assume the local cubic-mass bound
\begin{equation}
 B_3:=\sup_w\sum_{|c_\alpha-w|<1}m_\alpha r_\alpha^3<\infty.
 \label{eq:local-cubic-mass}
\end{equation}
Set $\Lambda'=\sum_\alpha\Gamma'_\alpha$.  Then
\[
 \Lambda\text{ is a zero set for }\F
 \quad\Longleftrightarrow\quad
 \Lambda'\text{ is a zero set for }\F.
\]
Uniqueness is preserved as well.
\end{theorem}

\begin{proof}
We use the product
\[
 Q(z)=\prod_\alpha
 \frac{\prod_{\mu\in\Gamma'_\alpha}(z-\mu)}
      {\prod_{\lambda\in\Gamma_\alpha}(z-\lambda)}
\]
and prove that it converges meromorphically, independently of order.
For every nonzero entire $f$ vanishing on $\Lambda$, we will obtain
\begin{equation}
 \mathcal Z(Qf)=\mathcal Z(f)-\Lambda+\Lambda',\qquad
 \|Qf\|_2\asymp\|f\|_2,
 \label{eq:cluster-zero-norm-isomorphism}
\end{equation}
with constants independent of $f$ and with infinite norms allowed.
The inverse product will give the reverse correspondence.

Increase $H$ to at least one if necessary. Each nonempty original
cluster contains a point within $H$ of its center. Since the source
multiset is discrete, the centers and the target multiset are locally
finite. For $\sigma=\sum_\alpha m_\alpha r_\alpha^3\delta_{c_\alpha}$,
\[
 \sigma(D(z,t))\le C B_3(t+1)^2\qquad(t\ge1)
\]
by a unit-disk covering. The moment identities are invariant under
translation by $-c_\alpha$: for $p=1,2$, the difference is
\[
 \sum_{j=0}^p\binom pj(-c_\alpha)^{p-j}
 \left(\sum_{\mu\in\Gamma'_\alpha}\mu^j
       -\sum_{\lambda\in\Gamma_\alpha}\lambda^j\right)=0,
\]
where the $j=0$ term vanishes by equal cardinality.
For $|z-c_\alpha|>2H$, the logarithm
of the cluster ratio, with the branch that vanishes at infinity, is
therefore
\[
 -\sum_{p\ge3}
 \frac{\sum_{\mu\in\Gamma'_\alpha}(\mu-c_\alpha)^p
       -\sum_{\lambda\in\Gamma_\alpha}(\lambda-c_\alpha)^p}
      {p(z-c_\alpha)^p}.
\]
Writing $R=|z-c_\alpha|$, its absolute value is at most
\[
 2m_\alpha\sum_{p\ge3}\frac{r_\alpha^p}{pR^p}
 \le\frac{2m_\alpha r_\alpha^3}{3R^3(1-r_\alpha/R)}
 \le\frac{4m_\alpha r_\alpha^3}{3R^3}.
\]
Summing over dyadic annuli gives
\[
 C_*:=\sup_z\sum_{|z-c_\alpha|>2H}
       \frac{4m_\alpha r_\alpha^3}{3|z-c_\alpha|^3}<\infty.
\]
Indeed, $H\ge1$ and the mass estimate give
\[
 \begin{split}
 \sum_{|z-c_\alpha|>2H}
      \frac{m_\alpha r_\alpha^3}{|z-c_\alpha|^3}
 &\le\sum_{j\ge0}(2^{j+1}H)^{-3}
                    \sigma(D(z,2^{j+2}H))\\
 &\le \frac{CB_3}{H}\sum_{j\ge0}2^{-j}<\infty.
 \end{split}
\]
Beyond radius $R$, the same sum is $O(B_3/R)$. On a fixed compact
set, sufficiently distant centers satisfy
$|z-c_\alpha|\ge|c_\alpha|/2$. The logarithmic tails therefore
converge absolutely and normally. Together with the finite remaining
factors they define $Q$ with the required zero and pole orders. Thus $g=Qf$ is
entire and has the zero multiset in
\eqref{eq:cluster-zero-norm-isomorphism}.

To compare the norms, put $h=4H$ and write
\[
 S_h\log|f|(z)=\frac1{2\pi}\int_0^{2\pi}
                          \log|f(z+he^{it})|\,dt.
\]
In Jensen's formula, the contribution of an original root changes
from $\log|z-\lambda|$ to $\log\max(h,|z-\lambda|)$.
When $|z-c_\alpha|\le2H$, the roots of both clusters lie within
distance $3H<h$ of $z$, and hence
\[
 \sum_{\mu\in\Gamma'_\alpha}\log|z-\mu|
 \le m_\alpha\log h
 =\sum_{\lambda\in\Gamma_\alpha}\log\max(h,|z-\lambda|).
\]
For distant clusters, the absolute logarithmic ratio bounds the excess
over the capped original sum. Let $E$ be a finite set of cluster
indices, with product $Q_E$, and put $g_E=Q_Ef$. This is entire.
Jensen's formula, away from the zeros of $f$, gives
\[
 \begin{split}
 \log|g_E(z)|-S_h\log|f|(z)
 &\le\sum_{\alpha\in E}\left[
       \sum_{\mu\in\Gamma'_\alpha}\log|z-\mu|
       -\sum_{\lambda\in\Gamma_\alpha}
                         \log\max(h,|z-\lambda|)\right].
 \end{split}
\]
Uncanceled original zeros and additional zeros of $f$ contribute
nonpositive terms to this upper bound. Near clusters also contribute
nonpositive terms. For a distant cluster, removing the cap from the
denominator increases the bracket, which is then bounded by the
absolute logarithmic ratio. The sum is at most $C_*$, independently
of $E$. Exhausting the clusters gives
\begin{equation}
 \log|g(z)|\le S_h\log|f|(z)+C_*.
 \label{eq:general-cluster-smoothed-comparison}
\end{equation}
The finite-product inequality extends across the original roots by
continuity. Indeed, factor off the finitely many zeros near the
averaging circle and use
$S_h\log|z-a|=\log\max(h,|z-a|)$ to see that the circular mean
is continuous. On a compact neighborhood, include all clusters with
original roots there. Their poles cancel against $f$, and the remaining
holomorphic tail converges normally. Thus $g_E\to g$ locally
uniformly, and the exponentiated inequality passes to the limit,
including at zeros. Interchanging the two configurations preserves
all hypotheses and replaces $Q$ by $Q^{-1}$, giving the same
estimate with $f$ and $g$ reversed.

The mean of $|z+he^{it}|^2$ is $|z|^2+h^2$. Applying Jensen
to \eqref{eq:general-cluster-smoothed-comparison} with the Gaussian
included gives
\begin{equation}
 |g(z)|^2e^{-|z|^2}
 \le e^{h^2+2C_*}\frac1{2\pi}\int_0^{2\pi}
       |f(z+he^{it})|^2e^{-|z+he^{it}|^2}\,dt.
 \label{eq:general-cluster-local-norm}
\end{equation}
Integration of \eqref{eq:general-cluster-local-norm} and Tonelli's
theorem give
\[
 \|g\|_2^2\le e^{h^2+2C_*}\|f\|_2^2.
\]
Interchanging the configurations gives the reverse estimate.
Multiplication by $Q^{-1}$ maps every entire function vanishing on
$\Lambda'$ back to one vanishing on $\Lambda$. This proves the
zero-set equivalence and, for the same reason, preservation of uniqueness.

\end{proof}

\begin{corollary}[Vanishing-space and insertion invariants]
\label{cor:cluster-invariants}
Under the hypotheses of Theorem~\ref{thm:general-cluster-isomorphism},
multiplication by its meromorphic product $Q$, with poles canceled against
the prescribed zeros, defines a bounded linear isomorphism
\[
 M_Q:\mathcal I(\Lambda)\longrightarrow\mathcal I(\Lambda'),
 \qquad f\longmapsto Qf.
\]
In particular, $\dim\mathcal I(\Lambda)=\dim\mathcal I(\Lambda')$.
If these multisets are exact zero multisets, then
$\iota(\Lambda)=\iota(\Lambda')$, including the value $+\infty$.
\end{corollary}

\begin{proof}
The divisor identity and the two norm bounds in
\eqref{eq:cluster-zero-norm-isomorphism} give the isomorphism, with inverse
$M_{Q^{-1}}$, and hence equality of dimensions. Equality of dimensions
alone does not imply equality of insertion indices, as
Theorem~\ref{thm:infinite-dimensional-prescribed-index} shows.
For the latter assertion, add the same arbitrary finite effective multiset $\nu$ to both configurations, treating each
added point, with its multiplicity, as an unchanged cluster of radius zero.
These clusters are understood as labeled parts of the counting measure,
so an added point may coincide with an existing point. All moment
identities remain valid and the cubic mass is unchanged; each added
cluster has ratio identically one. The theorem gives that $\Lambda+\nu$ is an exact zero multiset if and only
if $\Lambda'+\nu$ is. This proves equality of the insertion indices.
\end{proof}

Thus any replacement satisfying the local moment hypotheses, applied to
the examples of Theorem~\ref{thm:infinite-dimensional-prescribed-index},
preserves their infinite-dimensional vanishing spaces and their prescribed
finite insertion indices. The resulting multisets need not remain simple.

The local mass condition is automatic when the source has uniformly
bounded local counting, since
\[
 \sum_{|c_\alpha-w|<1}m_\alpha r_\alpha^3
 \le H^3\Lambda(D(w,1+H))\le C_H.
\]
In particular, the hypothesis holds for the arrays considered above.
It also permits unbounded local multiplicities when the cluster radii
are sufficiently small.

More generally, retain equal cardinalities, the common cluster disks,
and the radius bound in Theorem~\ref{thm:general-cluster-isomorphism}.
For a fixed $m\ge2$, the same conclusions hold if the first $m$
complex moments agree and \eqref{eq:local-cubic-mass} is replaced by
\[
 B_{m+1}:=\sup_w\sum_{|c_\alpha-w|<1}
                         m_\alpha r_\alpha^{m+1}<\infty.
\]
The binomial identity above shows that the first $m$ translated
moments also agree. For $R=|z-c_\alpha|>2H$, the remaining
series satisfies
\[
 2m_\alpha\sum_{p\ge m+1}\frac{r_\alpha^p}{pR^p}
 \le\frac{4m_\alpha r_\alpha^{m+1}}{(m+1)R^{m+1}}.
\]
The measure $\sigma_{m+1}=\sum_\alpha
m_\alpha r_\alpha^{m+1}\delta_{c_\alpha}$ has quadratic disk
growth. Consequently,
\[
 \sum_{|z-c_\alpha|>2H}
   \frac{m_\alpha r_\alpha^{m+1}}{|z-c_\alpha|^{m+1}}
 \le C_mB_{m+1}H^{1-m}\sum_{j\ge0}2^{-j(m-1)}<\infty.
\]
This gives locally uniform convergence of the product tail.
The near-root Jensen comparison uses only equal cardinalities and
the radius bound, so the same norm comparison follows with $B_{m+1}$
in place of $B_3$. For shrinking clusters, this condition can be weaker.
When the cubic-mass condition already holds, the extra moments are
unnecessary.

The radial and angular arrays above have bounded local counting.
Their criteria therefore extend to bounded local replacements preserving
two moments, even if the new points leave the circles. For example,
let $\omega=e^{2\pi i/3}$ and replace
$\{c+u,c+\omega u,c+\omega^2u\}$ by
$\{c+v,c+\omega v,c+\omega^2v\}$. Both first moments equal $3c$,
and both second moments equal $3c^2$. Any disjoint family of these
replacements with uniformly bounded radii satisfies the theorem when
the source has bounded local counting.

\section{Sharpness and limits of finite geometric data}
\label{sec:counterexamples}

The examples in this section identify information that the complete
criterion must retain. We first compare area integrability with a
pointwise potential bound, then examine the displacement and angular
terms in the geometric criteria.

\subsection{The entropy term and finite changes}

Let $\sigma$ be the critical square-lattice sigma function. The proof
of Lemma~\ref{lem:bounded-displacement-potential} gives
\begin{equation}
 S_1\log|\sigma|(z)-|z|^2/2=O(1),\qquad z\in\C.
 \label{dual-eq:sigma-mean}
\end{equation}
The familiar finite-puncture examples \cite{Zhu2011} also distinguish
the entropy criterion from a pointwise potential bound.

\Needspace{12\baselineskip}
\begin{example}\label{ex:punctured-lattice-tests}
Let $D$ consist of $d$ distinct points of $\mathcal L$, where $d\ge0$,
and put $\Lambda_d=\mathcal L-D$. Then
\begin{enumerate}[label=\textup{(\roman*)},leftmargin=2.2em]
 \item $\calT_q(\Lambda_d)<\infty$ if and only if $q\le d$;
 \item $\Lambda_d$ is a zero set for $\F$ if and only if $d\ge2$;
 \item there is an $f$ with $\mathcal Z(f)=\Lambda_d$ and $f,z^qf\in\F$
       if and only if $d\ge q+2$.
\end{enumerate}
In particular, finiteness of $\calT_0$ alone does not characterize Fock
zero sets, and a Fock zero set need not be stable under finite
modifications.
\end{example}

\begin{proof}
Let $A_D(z)=\prod_{\lambda\in D}(z-\lambda)$ and
$g_d=\sigma/A_D$. By \eqref{dual-eq:sigma-mean} and the circle-mean formula,
\begin{equation}\label{dual-eq:punctured-mean}
 S_1\log|g_d|(z)-|z|^2/2
 =-d\log(1+|z|)+O(1).
\end{equation}
Here we used $\ell(z-\lambda)=\log(1+|z|)+O_\lambda(1)$, uniformly in
$z$. Hadamard factorization shows that $g_d$ and $P_{\Lambda_d}$ differ
by a nonzero constant times the exponential of a polynomial of degree
at most two. Thus their smoothed potentials give atomic tests that
differ only by a fixed constant.

If $q\le d$, \eqref{dual-eq:punctured-mean} makes the atomic objective
bounded above. If $q>d$, use the balanced measure assigning mass $1/5$
to each point $r\xi_j$, where the $\xi_j$ are the fifth roots of unity.
The objective is $(q-d)\log(1+r)+O(1)$ and tends to infinity. This
proves (i).

For (iii), every potential Fock realization has the form
$f=cg_de^{az+bz^2}$. Applying Lemma~\ref{lem:circular-fock-smoothing} to $z^qf$ and
using \eqref{dual-eq:punctured-mean}, its smoothed squared norm is comparable
to
\[
 \int_\C(1+|z|)^{2(q-d)}e^{2\Re(az+bz^2)}\dd(z).
\]
The circular mean of the exponential is at least one, by Jensen's
inequality. Thus finiteness requires
$\int_1^\infty r^{2(q-d)+1}\,dr<\infty$, or $d>q+1$.
If this inequality holds, take $f=g_d$; the same estimate proves both
required norms finite. This gives (iii), and (ii) is its case $q=0$.
\end{proof}

This also illustrates why the convention on multiplicities matters.
The locations $\Lambda_2$ are realized by simple zeros. Prescribing order
two at every one of these locations gives a different multiset, with
count $2r^2+O(r)$ and Jensen mean $r^2+O(r)$. Jensen's formula and
\eqref{eq:fock-point-evaluation} would instead bound that mean by $r^2/2+O(1)$
for a Fock realization. Thus the doubled multiset is not a zero set.

\subsection{The quadratic displacement term}

\begin{example}[Radial oscillation]
\label{ex:bounded-oscillatory-endpoint}
Let $b(r)=c+d\sin(\log r)$ outside a fixed disk. The hypotheses
of Corollary~\ref{cor:smooth-radial-displacements} hold, since,
with $t=\log r$,
\[
 b'(r)=\frac{d\cos t}{r},\qquad
 b''(r)=-\frac{d(\sin t+\cos t)}{r^2}.
\]
Moreover,
\begin{align*}
 \int_{R_0}^r b(s)\,ds
   &=cr+\frac{dr}{2}(\sin\log r-\cos\log r)+O(1),\\
 \int_{R_0}^r\frac{b(s)^2}{s}\,ds
   &=(c^2+d^2/2)\log r-2cd\cos\log r
                         -\frac{d^2}{4}\sin(2\log r)+O(1).
\end{align*}
The drift expansion in the proof of
Corollary~\ref{cor:smooth-radial-displacements} therefore gives
\begin{equation*}
 B(r)=r[2c+d(\sin\log r-\cos\log r)]
       -(c^2+d^2/2)\log r+O(1).
 \end{equation*}
If $2c>\sqrt2|d|$, every finite modification is a zero set. If
$2c<\sqrt2|d|$, every finite modification is a uniqueness multiset.
The minimum of $2c+d(\sin t-\cos t)$ is $2c-\sqrt2|d|$.
If it is positive, exponential decay gives convergence for every $q$.
If it is negative, the expression is at most $-\epsilon$ on an
interval of phases. On the corresponding, geometrically recurring
intervals of radii, exponential growth gives divergence for every $q$.
Suppose $2c=\sqrt2|d|$ and $c>0$. Then
\begin{equation*}
 \Xi\text{ is a Fock zero set}
       \quad\Longleftrightarrow\quad q<-\frac34-2c^2.
 \end{equation*}
In particular, equality gives uniqueness. When $c=d=0$, the
criterion reduces to the lattice condition $q<-1$.

For the endpoint, put $F(t)=2c+d(\sin t-\cos t)$ and
$\kappa=c^2+d^2/2=2c^2$. Then $F\ge0$, and its zeros
$t_n=t_0+2\pi n$ are quadratic. Set $r_n=e^{t_n}$.
In a fixed logarithmic neighborhood of $r_n$,
\[
 rF(\log r)\asymp (r-r_n)^2/r_n.
\]
Indeed, choose $\eta>0$ so that
$C_1x^2\le F(t_n+x)\le C_2x^2$ for $|x|\le\eta$,
uniformly in $n$. On $I_n=[e^{-\eta}r_n,e^\eta r_n]$,
we have $r\asymp r_n$ and
$|\log(r/r_n)|\asymp |r-r_n|/r_n$. Hence, with
$\rho=2q+1+2\kappa$,
\[
 \int_{I_n}r^\rho e^{-2rF(\log r)}\,dr
 \asymp r_n^\rho
       \int_{I_n}e^{-C(r-r_n)^2/r_n}\,dr
 \asymp r_n^{\rho+1/2}.
\]
The two bounds use different positive constants $C$; the last
comparison follows by $y=(r-r_n)/\sqrt{r_n}$. The contribution
of $I_n$ is therefore comparable to
\[
 r_n^{2q+1+2\kappa}\sqrt{r_n}
       =r_n^{2q+3/2+4c^2}.
\]
Off these intervals, $F(\log r)$ is bounded below by a positive
constant, so the remaining integral converges. Since $r_n$ grows
geometrically, the contributions above are summable exactly when
$2q+3/2+4c^2<0$.

Take $A_0$ to be the lattice points in a sufficiently large centered
disk, and use the same $A_0$ and cutoff $T$ for $b$ and $-b$. Write
$V_\Gamma=V_1+V_2$ for the first- and second-order terms in
\eqref{eq:bounded-displacement-field}. In
\eqref{eq:bounded-second-order-reduction}, the linear correction
vanishes on opposite pairs: $\delta_{-\lambda}=-\delta_\lambda$
and $\mu_{-\lambda}=-\mu_\lambda$, so
$\sum\delta_\lambda^2/(\lambda^2\mu_\lambda)=0$.
Apply \eqref{eq:smooth-radial-potential} to $b$ and $-b$.
Under this change $V_1$ changes sign and $V_2$ is unchanged.
Adding and subtracting gives
\[
 V_1(z)=-2\int_{R_0}^{|z|}b(t)\,dt+O(1),\qquad
 V_2(z)=\int_{R_0}^{|z|}\frac{b(t)^2}{t}\,dt+O(1).
\]
Take $c=1/2$ and $d=1/\sqrt2$. The integral with $V_1$ alone
converges for $q<-3/4$, whereas the zero-set condition is
$q<-5/4$. Thus deleting one point from the original array gives a uniqueness
multiset, whereas deleting two gives a zero set. The quadratic term
cannot be omitted.
\end{example}
\subsection{Angular width and lower-order radial terms}

\begin{example}[Finite-order angular maxima]
\label{ex:finite-order-angular-maxima}
Suppose $h$ has finitely many maxima and
$H-h(\theta_j+x)\asymp |x|^{2m_j}$ near each one, with
$m_j\in\N$. Compactness bounds the gap away from zero elsewhere.
Putting $m_* =\max_jm_j$, we obtain
\[
 \nu_h(u)\asymp\sum_j u^{1/(2m_j)}\asymp u^{1/(2m_*)}
                   \qquad(u\downarrow0).
\]
At the critical equalities of
Theorem~\ref{thm:smooth-angular-profile}, the zero-set condition is
\begin{equation}
 \delta>\frac12-\frac{\tau}{4m_*}.
 \label{eq:finite-order-angular-endpoint}
\end{equation}
Indeed, for $\delta<1/2$ the exponent of $u$ in
\eqref{eq:smooth-profile-geometric-endpoint} is
$1/(2m_*)+(2\delta-1)/\tau-1$, which must exceed $-1$.
For $\delta\ge1/2$ the condition holds directly.
Equality in \eqref{eq:finite-order-angular-endpoint} is excluded.
Thus quadratic and quartic maxima give, respectively,
$\delta>1/4$ and $\delta>3/8$ at the logarithmic scale, but
$\delta>0$ and $\delta>1/4$ at the squared logarithmic scale.

For example,
\[
 h_m(\theta)=H_m-(1-\cos\theta)^m,
 \qquad H_m=2^{-m}\binom{2m}{m},
\]
has mean zero and a maximum of order $2m$ at zero. The value of
$H_m$ follows by integrating
$2^m\sin^{2m}(\theta/2)$, or by taking its constant Fourier
coefficient. At $\tau=2$, $a=H_m$, and
$b=(q+1-1/6)/2$, the criterion is $c>1/4-1/(4m)$.
\end{example}

\begin{example}[The leading angular profile is insufficient]
\label{ex:squared-log-lower-order}
In \eqref{eq:smooth-profile-configuration}, take $\tau=2$,
$h(\theta)=\cos\theta$, $a=1$, $b=5/12$, and $q=0$.
The generators satisfy
\[
 \log|F(r_ne^{i\theta})|
 =\frac{r_n^2}{2}+(\cos\theta-1)(\log r_n)^2
           -\log r_n-2c\log\log r_n+O(1).
\]
The squared logarithmic profile and the coefficient of $\log r$
are the same for all $c$. Nevertheless the configuration is a
zero set for $c>0$ and a uniqueness set for $c\le0$, by
\eqref{eq:finite-order-angular-endpoint}. In the first case,
Corollary~\ref{cor:angular-vanishing-dimension} gives coherent-state
codimension one; in the second case the system is complete.
Changing $c$ moves the radii by only $O(1/(k\log k))$.
Its cumulative $\log\log r$ contribution decides the endpoint,
so an unspecified $O(\log r)$ remainder would lose essential data.
\end{example}

\subsection{Flat angular maxima}

\begin{example}[Infinitely flat maxima]
\label{ex:smooth-flat-maxima}
Let $p>0$. Choose a smooth nonnegative $2\pi$-periodic function
$b_p$ whose only zero modulo $2\pi$ is $0$, with
\[
 b_p(\theta)=e^{-|\theta|^{-p}}
       \qquad(0<|\theta|<\epsilon).
\]
Fix $H>0$. If $\overline b_p$ denotes the circular mean of $b_p$, set
\[
 h_p(\theta)=H-\frac{H}{\overline b_p}b_p(\theta).
\]
A smooth cutoff joins the specified expression to a positive constant
away from zero. Every derivative near zero is a finite sum of powers
of $|\theta|$ times $e^{-|\theta|^{-p}}$ and tends to zero.
Thus $h_p$ is smooth, has mean zero and maximum $H$, and satisfies
$h_p^{(j)}(0)=0$ for all $j\ge1$.

Fix one sequence of phases and one sequence of radii for all $p$.
For a fixed sufficiently large $K_0$, take $R_k=k$ for $k<K_0$ and
\[
 R_k=k+\frac{H+5/6}{2k}+\frac1{4k\log k}
                  \qquad(k\ge K_0),
\]
where $K_0$ is chosen so that $|R_k-k|<1/4$. For each $p$, choose
$K_p\ge K_0$ so that the angular map in
\eqref{eq:smooth-profile-configuration}, with $\tau=1$ and $h=h_p$,
has derivative at least $1/2$ for $k\ge K_p$. Apply that map on
these shells and use regular polygons on the common radii $R_k$
for $k<K_p$. Thus every configuration is simple and has exactly
$2k$ points on the same circle $|z|=R_k$. Its difference from the
corresponding construction with starting index $K_p$ is a finite modification of net
multiplicity zero. Theorem~\ref{thm:smooth-angular-profile} therefore
applies, with $a=0$, to the resulting set $\Lambda_p$ and gives
\begin{equation*}
 \Lambda_p\text{ is a zero set for }\F
       \quad\Longleftrightarrow\quad p<1;
 \qquad
 p\ge1\ \Longrightarrow\ \Lambda_p\text{ is a uniqueness set}.
 \end{equation*}
Here $\gamma-H=1$ and $\delta=1/2$. Put $C_p=H/\overline b_p$.
For small $u$, the set $\{H-h_p\le u\}$ lies near zero, where
\[
 C_p e^{-|\theta|^{-p}}\le u
 \quad\Longleftrightarrow\quad
 |\theta|\le\bigl(\log(C_p/u)\bigr)^{-1/p}.
\]
Its normalized length is
$\pi^{-1}(\log(C_p/u))^{-1/p}$. In particular, as $u\downarrow0$,
\[
 \nu_{h_p}(u)\asymp(\log(1/u))^{-1/p}.
\]
Theorem~\ref{thm:smooth-angular-profile} therefore gives the condition
\[
 \int_0^{1/2}\frac{du}{u(\log(1/u))^{1/p}}<\infty.
\]
Substitution gives $\int_{\log2}^\infty v^{-1/p}\,dv$, which
converges exactly when $p<1$. The radial counts and all derivatives
at the common maximum agree, but the zero-set property changes at
$p=1$. Also,
\[
 r e^{2D_{\Lambda_p}(r)}\asymp
                r^{-2H-1}(\log r)^{-1}\qquad(r\to\infty).
\]
Since $H>0$, the radial integral in
\eqref{eq:jensen-radial-norm-lower} converges for every $p$.
It is therefore insufficient even for bounded displacements of the
critical lattice.

Nor can these potentials be approximated, with bounded error, by
finitely many angular harmonics. A nonconstant trigonometric polynomial
has no infinitely flat maximum, so $h_p$ has infinitely many nonzero
Fourier coefficients $c_m$. By
\eqref{eq:smooth-profile-potential}, the $m$th coefficient of the
centered potential is
\[
 c_m\log r_n+O(1).
\]
For any fixed $L$, choose $m>L$ with $c_m\ne0$. Subtracting
modes of orders at most $L$ leaves this coefficient unchanged, even
if their coefficients depend on $r_n$. Since a Fourier coefficient
is bounded by the supremum norm, the error cannot be uniformly bounded.
\end{example}
\subsection{The first moment condition}

\begin{proposition}[Preserving the first local moment]
\label{prop:pair-moment-sharpness}
On each circle $|z|=k$, take a regular $2k$-gon with arbitrary
phase $\theta_k$. For $k\ge2$, set
\[
 a_k=\frac{\pi}{2k},\qquad
 \phi_{kj}=\theta_k+\frac{(2j+1/2)\pi}{k},\quad 0\le j<k.
\]
Write its adjacent pairs as $c_{kj}\pm v_{kj}$, where
\[
 c_{kj}=k\cos a_k\,e^{i\phi_{kj}},\qquad
 v_{kj}=ik\sin a_k\,e^{i\phi_{kj}}.
\]
For fixed $t>0$, replace each pair by $c_{kj}\pm tv_{kj}$ and
leave the first shell unchanged. If $\Lambda_t$ is the resulting
multiset and $\Xi$ is a finite modification of net multiplicity $q$,
then $\Xi$ is a zero set for $\F$ if and only if
\begin{equation*}
 \frac{\pi^2}{2}(t^2-1)>2q+\frac53.
 \end{equation*}
Otherwise $\Xi$ is a uniqueness multiset. The replacements are
bounded, satisfy \eqref{eq:local-cubic-mass}, and preserve the first
moment of every pair. For $t\ne1$, they change the second moment.
\end{proposition}

\begin{proof}
The replaced shell is the union of two regular $k$-gons of radius
\[
 R_k(t)=k\sqrt{\cos^2a_k+t^2\sin^2a_k}
       =k+\frac{\alpha_t}{k}+O_t(k^{-3}),\qquad
 \alpha_t=\frac{\pi^2}{8}(t^2-1).
\]
The expansion follows from $\sin^2a_k=\pi^2/(4k^2)+O(k^{-4})$
and $\sqrt{1+x}=1+x/2+O(x^2)$. With
$\psi_k=\arctan(t\tan a_k)$ and
$\beta_{k,\pm}=\theta_k+\pi/(2k)\pm\psi_k$, the two polygon
factors are
\[
 P_{k,\pm}(z)=1-e^{-ik\beta_{k,\pm}}(z/R_k(t))^k.
\]
Let $P_1$ be the factor for the unchanged first shell and set
$F_t=P_1\prod_{k\ge2}P_{k,+}P_{k,-}$. On $|z|\le M$,
\[
 \sum_{k\ge K}\sup_{|z|\le M}|P_{k,\pm}(z)-1|
 \le\sum_{k\ge K}(2M/k)^k<\infty
\]
for sufficiently large $K$. Normal convergence of the logarithmic tail shows that $F_t$ is
entire with the prescribed zeros, counting collisions between the
two polygons with their multiplicities.

Let $r_n=n+1/2$. For large $n$, these radii stay a fixed positive
distance from the occupied circles. Dividing each polygon factor by
its larger radial monomial leaves a modulus between
$1-e^{-k|\log(r_n/R_k)|}$ and $1+e^{-k|\log(r_n/R_k)|}$.
We claim that
\[
 \log|F_t(r_ne^{i\theta})|=J_{\Lambda_t}(r_n)+O_t(1)
\]
uniformly in $n,\theta$. Put $x_{kn}=e^{-k|\log(r_n/R_k)|}$.
For large $n$, $x_{kn}\le\rho<1$ uniformly in $k$, and
\[
 |\log|1-x_{kn}e^{i\vartheta}||
 \le-\log(1-x_{kn})\le\frac{x_{kn}}{1-\rho}.
\]
The sum in $k$ is bounded by
\eqref{eq:variable-population-shell-overlap}; finitely many initial
factors contribute $O(1)$. Each shell contributes
$2k\log^+(r_n/R_k)$ to the radial term, proving the claim. Also,
\[
 2k\log(R_k(t)/k)=\frac{2\alpha_t}{k}+O_t(k^{-3}),\qquad
 T_n=2\alpha_t\log n+O_t(1).
\]
Lemma~\ref{lem:joint-radial-data} now gives
\[
 D_{\Lambda_t}(r)=-\left(\frac16+2\alpha_t\right)\log r+O_t(1).
\]
Let $G_t$ be the finitely modified generator.
Lemma~\ref{lem:gaussian-circle-sampling} gives
\[
 r_n e^{-r_n^2}\frac1{2\pi}\int_0^{2\pi}
                       |G_t(r_ne^{i\theta})|^2\,d\theta
 \asymp r_n^{2q+1}e^{2D_{\Lambda_t}(r_n)}
 \asymp n^{2q+2/3-4\alpha_t}.
\]
Hence $G_t\in\F$ exactly when these powers are summable.
Conversely, \eqref{eq:finite-change-jensen} gives
$D_\Xi(r)=-(\tfrac16+2\alpha_t-q)\log r+O_t(1)$.
By \eqref{eq:jensen-radial-norm-lower}, every nonzero entire function
vanishing on $\Xi$ has Fock integral at least a positive multiple of
$\int_R^\infty r^{2q+2/3-4\alpha_t}\,dr$. Thus the condition is
\[
 2q+\frac23-4\alpha_t<-1
 \quad\Longleftrightarrow\quad
 \frac{\pi^2}{2}(t^2-1)>2q+\frac53.
\]
This proves both assertions, including uniqueness at equality.

The first power sum remains $2c_{kj}$ and the second changes by
$2(t^2-1)v_{kj}^2$. The displacement tends to $|t-1|\pi/2$.
Each center is a bounded distance from a distinct original vertex,
so Lemma~\ref{lem:joint-radial-data} gives bounded local counting
for the centers. For the cluster centered at $c_{kj}$, take
$r_{kj}=\max(1,t)|v_{kj}|\le\max(1,t)\pi/2=:H_t$.
Consequently
\[
 \sup_w\sum_{|c_{kj}-w|<1}2r_{kj}^3
 \le2H_t^3\sup_w\#\{(k,j):|c_{kj}-w|<1\}<\infty,
\]
which proves the cubic-mass condition.
\end{proof}

\subsection{Finitely many angular moments at critical density}
\label{sec:geometric-obstructions}
\label{sec:angular-necessity}

The obstruction persists among uniformly separated bounded
displacements of the critical lattice, even for arbitrarily small
motions which tend to zero at infinity.

\Needspace{15\baselineskip}
\begin{theorem}[Insufficiency of finite angular data]
\label{thm:finite-angular-data-insufficient}
For every integer $L\ge0$, every $\eta>0$, and every $R_*>0$,
there are simple uniformly separated sets $\Lambda_+$ and $\Lambda_-$
and a bijection $T:\Lambda_+\to\Lambda_-$ with the following properties.
\begin{enumerate}[label=\textup{(\roman*)},leftmargin=2.3em]
\item Each set admits a bounded matching to $\mathcal L$, and
      \[
       \sup_{w\in\C}
       \bigl|\#(\Lambda_\pm\cap D(w,r))-r^2\bigr|=O(r)
       \qquad(r\to\infty).
      \]
\item The matching preserves moduli and is the identity in $D_{R_*}$.
      Moreover,
      \begin{equation}
       \sup_{\lambda\in\Lambda_+}|T\lambda-\lambda|<\eta,
       \qquad
       |T\lambda-\lambda|
       =O_L\!\left(\frac{\log^2(2+|\lambda|)}{1+|\lambda|}\right).
       \label{eq:obstruction-small-motion}
      \end{equation}
\item The radial counting functions agree at every radius. For every
      $\rho>0$ and $1\le j\le L$,
      \begin{equation*}
       \sum_{\substack{\lambda\in\Lambda_+\\|\lambda|=\rho}}
                   e^{ij\arg\lambda}
       =\sum_{\substack{\mu\in\Lambda_-\\|\mu|=\rho}}
                   e^{ij\arg\mu}.
       \end{equation*}
      These moments vanish outside a fixed disk.
\item Every finite modification of $\Lambda_+$ is a Fock zero set.
      Every finite modification of $\Lambda_-$ is a uniqueness multiset.
\end{enumerate}
\end{theorem}

\begin{proof}
Choose $N=\max\{L+1,3\}$, put $h(\theta)=4\cos(N\theta)$, and set
\[
 R_k=k+\frac{\log k}{k},\qquad
 \theta_{k,\ell}=\frac{\pi\ell}{k},\qquad 0\le\ell<2k.
\]
For a sufficiently large $K\ge2N$, define
\[
 t_k=\begin{cases}
       N(\log k)^2/(2k^2),&k\ge K,\ N\mid k,\\
       0,&\text{otherwise},
      \end{cases}
\]
and take
\[
 \lambda_{k,\ell}=R_ke^{i\theta_{k,\ell}},\qquad
 \mu_{k,\ell}=R_k
       e^{i(\theta_{k,\ell}-t_kh'(\theta_{k,\ell}))}.
\]
Let $\Lambda_+$ and $\Lambda_-$ consist of these points, matched
by their labels. Only shells with $N\mid k$ are deformed. The
factor $N$ in $t_k$ compensates for the omitted shells.

The derivative of $x+\log x/x$ is bounded below by $1/2$ for
$x\ge1$, so consecutive radii are separated by at least $1/2$.
Choose $K$ so that $t_k\|h''\|_\infty\le1/2$ for $k\ge K$.
Then every intermediate angular map
$\theta\mapsto\theta-s t_kh'(\theta)$, $0\le s\le1$, has
derivative at least $1/2$. The angular gaps on a deformed shell
are at least $\pi/(2k)$. Since $R_k\asymp k$, this proves
uniform separation and bijectivity. Also,
\[
 |\mu_{k,\ell}-\lambda_{k,\ell}|
 \le R_kt_k\|h'\|_\infty
 \le C_N\frac{(\log k)^2}{k}\longrightarrow0.
\]
Increasing $K$ makes the supremum less than $\min\{\eta,1\}$,
and $K>R_*+1$ makes the configurations identical in $D_{R_*}$.
The regular array has a bounded allocation by
Lemma~\ref{lem:joint-radial-data}, and the motion preserves it.
The same lemma and Lemma~\ref{lem:bounded-area-allocation} give
the uniform disk discrepancy and bounded lattice matchings in
\textup{(i)}.

On an active shell, $N\mid k$, so addition of $2\pi/N$ permutes
the regular $2k$-gon. The function $h'$ has this period, so the
same rotation also permutes the deformed shell. Its $j$th moment
is multiplied by $e^{2\pi ij/N}\ne1$ for $1\le j<N$, and hence
vanishes. On every inactive shell the two sets coincide. A regular
$2k$-gon has vanishing moments of orders $1,\ldots,L$ once
$2k>L$. This proves \textup{(ii)} and \textup{(iii)}.

It remains to prove \textup{(iv)}. Write $J$ for the common radial
potential and put $r_n=n+1/2$. By
Lemma~\ref{lem:joint-radial-data} and the summable expansion
\[
 2k\log(R_k/k)=\frac{2\log k}{k}
                     +O\!\left(\frac{(\log k)^2}{k^3}\right),
\]
we have
\begin{equation*}
 J(r_n)=\frac{r_n^2}{2}-(\log r_n)^2-\frac16\log r_n+O(1).
 \end{equation*}
The grouped products
\[
 F_+(z)=\prod_{k\ge1}\prod_{\ell<2k}(1-z/\lambda_{k,\ell}),
 \qquad
 F_-(z)=\prod_{k\ge1}\prod_{\ell<2k}(1-z/\mu_{k,\ell})
\]
converge locally uniformly. Indeed, the first two reciprocal power
sums vanish on every sufficiently large shell: on inactive shells
by regularity, and on active shells by $N$-fold symmetry and $N\ge3$.
For $|z|\le A$ and $R_k>2A$, the logarithm of either shell product,
normalized to vanish at zero, consequently has modulus at most
\[
 2k\sum_{j\ge3}\frac{(A/R_k)^j}{j}\le C_A k/R_k^3.
\]
This bound is summable. The logarithmic tails have nonvanishing
exponentials, and the remaining factors give precisely the stated zeros.

We claim that, uniformly in $\theta$,
\begin{align}
 \log|F_+(r_ne^{i\theta})|&=J(r_n)+O(1),\notag\\
 \log|F_-(r_ne^{i\theta})|&=J(r_n)
                    +4\cos(N\theta)(\log r_n)^2+O(1).
 \label{eq:obstruction-potentials}
\end{align}
Here we can use the estimates already proved in
Lemmas~\ref{lem:profile-angular-sampling} and
\ref{lem:profile-radial-kernel}. For clarity, we give
the first variation and its summation on the selected shells.
Put $q_k=\min(r_n/R_k,R_k/r_n)$, $d_k=-\log q_k$,
and $L_q(x)=\log|1-qe^{ix}|$.
Although the initial radii here differ from those in
\eqref{eq:smooth-profile-configuration}, only finitely many shells are
affected. Their centered logarithmic contributions are bounded for
large $n$. On the tail, $R_k-k=o(1)$, so
\eqref{eq:smooth-profile-overlap} applies unchanged. The
regular-polygon terms therefore have bounded sum, proving the first
identity.
The first variation on an active shell is
\[
 -t_k\sum_{\ell<2k}h'(\theta_{k,\ell})
                         L'_{q_k}(\theta_{k,\ell}-\theta).
\]
The Taylor and sampling estimates
\eqref{eq:profile-taylor-error} and
\eqref{eq:profile-sampling-error}, with $t_k$ at most $N$ times
the parameter there, bound the errors by
\[
 C_N\frac{(\log k)^4}{k^2},\qquad
 C_N\frac{(\log k)^2}{k}e^{-kd_k}
       +C_N\frac{(\log k)^2}{k^{3/2}},
\]
respectively. Their sums over the selected shells are uniformly
bounded. Fourier multiplication, or sine orthogonality, gives
\[
 -\frac1{2\pi}\int_0^{2\pi}h'(x)L'_q(x-\theta)\,dx
                         =2Nq^N\cos(N\theta).
\]
The continuous first variation is therefore
$4Nkt_kq_k^N\cos(N\theta)$.
The selected-shell estimate \eqref{eq:profile-selected-shells},
with $\tau=2$ and $m=N$, gives
\[
 N\sum_{\substack{k\ge K\\N\mid k}}
       \frac{(\log k)^2}{k}q_k^N
                       =\frac{2(\log r_n)^2}{N}+O_N(1).
\]
Multiplication by $2N\cos(N\theta)$ shows that the variations
sum to $4\cos(N\theta)(\log r_n)^2+O(1)$.
This proves \eqref{eq:obstruction-potentials}.

Let $G_+$ or $G_-$ be a generator after any finite modification
of net multiplicity $q$. Its logarithmic modulus acquires the term
$q\log r_n+O(1)$. For $G_+$, circular-mean sampling gives
\[
 \|G_+\|_2^2<\infty
 \quad\Longleftrightarrow\quad
 \sum_{n\ge3} n^{2q+2/3}e^{-2(\log n)^2}<\infty,
\]
and the series converges for every $q$.
For $G_-$, the coefficient of $(\log r_n)^2$ is
$-1+4\cos(N\theta)$. It is bounded below in every direction
and equals $3$ at $\theta=0$. Thus $G_-$ satisfies the hypotheses
of Proposition~\ref{prop:critical-direction}, while
\[
 \log|G_-(r_n)|-r_n^2/2
       =3(\log r_n)^2+(q-1/6)\log r_n+O(1)\longrightarrow\infty.
\]
The Fock pointwise estimate excludes $G_-\in\F$.
The proposition then excludes every nonzero Fock function vanishing
on its zeros, including functions with additional zeros. This proves
\textup{(iv)}.
\end{proof}

For simple sets, the coherent states indexed by $\Lambda_+$ are
incomplete, whereas those indexed by $\Lambda_-$ are complete.
The same is true after any finite addition or deletion of distinct
states, by \textup{(iv)} and the orthogonal-complement identity in
Section~\ref{sec:coherent-states}.

The construction chooses $N$ and then $K$ after $L$, $\eta$, and
$R_*$ have been fixed. Its constants may depend on $L$; no uniformity
as $L\to\infty$ is claimed. The motions in
\eqref{eq:obstruction-small-motion} are not asserted to satisfy an
inverse-radius bound, and no optimality of their rate is claimed.
Circle moments also differ from the moments of bounded local clusters
in Theorem~\ref{thm:general-cluster-isomorphism}.

Theorem~\ref{thm:complete-zero-set-characterization} distinguishes the
two configurations: $\calH(\Lambda_+)<\infty$ and
$\calH(\Lambda_-)=\infty$. Theorem~\ref{thm:finite-modification-stability}
also gives $\calT_q(\Lambda_+)<\infty$ for every $q$. Thus finite
circle moments cannot replace either family of potential tests.
The full angular measures, their radii, and the multiplicity at zero
determine the multiset and are retained by the complete criterion.

\section{Zero-based subspaces and coherent states}
\label{sec:coherent-states}
\label{rem:coherent-state-interpretation}

The zero-set criteria have consequences for coherent states when a
uniqueness alternative is available. The distinction between realizing
all prescribed zeros and merely vanishing on them is essential here.
For $\lambda\in\supp\Lambda$ and $0\le j<m_\lambda$, put
\[
 K_{\lambda,j}(z)=z^j e^{\overline\lambda z}.
\]
The coefficient formula \eqref{eq:fock-coefficient-norm} gives
\[
 \|K_{\lambda,j}\|_2^2
 =\sum_{n\ge0}\frac{(n+j)!}{(n!)^2}|\lambda|^{2n}<\infty,
 \qquad \langle f,K_{\lambda,j}\rangle=f^{(j)}(\lambda),
\]
with the inner product linear in its first variable. The first series
converges by the ratio test; the second identity follows from the
coefficient inner product and Cauchy--Schwarz. Consequently,
\begin{equation}
 \left(\overline{\operatorname{span}}
 \{K_{\lambda,j}:\lambda\in\supp\Lambda,\ 0\le j<m_\lambda\}\right)^\perp
 =\mathcal I(\Lambda).
 \label{eq:kernel-orthogonal-complement}
\end{equation}
Thus the system is complete exactly when $\Lambda$ is a uniqueness
multiset. For simple points, the normalized kernels
\[
 k_\lambda(z)=e^{\overline\lambda z-|\lambda|^2/2},\qquad
 \|k_\lambda\|_2=1,
\]
are the usual coherent states in the Bargmann representation
\cite{BargmannEtAl1971}. They are nonzero scalar multiples of
$K_{\lambda,0}$. Repeating the same kernel does not express a higher
zero order; the derivative kernels do.

\subsection{Dimension at angular endpoints}

\begin{corollary}[Vanishing functions and coherent-state codimension]
\label{cor:angular-vanishing-dimension}
Use the notation of Theorem~\ref{thm:smooth-angular-profile}, and
let $G=Fq_{\nu_+}/q_{\nu_-}$ be its modified generator.
If $\tau=1$, or if $\tau=2$ and $a=H$, set
\[
 D=\begin{cases}
    \gamma-H-q-1,&\tau=1,\\
    \gamma-q-1,&\tau=2.
   \end{cases}
\]
Let $\mathcal D$ consist of the integers $d\ge0$ such that $d<D$,
including $d=D$ when $D$ is a nonnegative integer and
$\mathcal E_\tau(h,\delta)$ holds. In these cases,
\begin{equation}
 \mathcal I(\Xi)=\operatorname{span}\{z^dG:d\in\mathcal D\},
 \qquad \dim\mathcal I(\Xi)=\#\mathcal D.
 \label{eq:angular-vanishing-dimension}
\end{equation}
The span of the empty family is understood to be $\{0\}$.
In the remaining cases, $\tau=2$ and $a\ne H$, the dimension
$\dim\mathcal I(\Xi)$ is infinite when $a>H$ and zero when $a<H$.
For every multiset $\Xi$ in these cases, with
$m_\lambda=\Xi(\{\lambda\})$,
\begin{equation}
 \operatorname{codim}_{\F}
   \overline{\operatorname{span}}\{K_{\lambda,j}:\lambda\in\supp\Xi,\ 0\le j<m_\lambda\}
          =\dim\mathcal I(\Xi).
 \label{eq:angular-coherent-codimension}
\end{equation}
\end{corollary}
\begin{proof}
In the stated critical regimes, Proposition~\ref{prop:critical-direction}
shows that every nonzero member of $\mathcal I(\Xi)$ is $PG$
for a polynomial $P$. If $P$ has degree $d$, then
$|P(z)|\asymp |z|^d$ outside a fixed disk, uniformly in angle.
Its norm test is therefore \eqref{eq:smooth-profile-integral}
with $q$ replaced by $q+d$. Theorem~\ref{thm:smooth-angular-profile}
gives exactly $d\in\mathcal D$. This set is empty or an initial
segment of the nonnegative integers, proving
\eqref{eq:angular-vanishing-dimension}.
For $\tau=2$ and $a>H$, all the functions $z^dG$ belong to
$\F$ and are linearly independent; for $a<H$ the theorem gives
uniqueness. Finally, \eqref{eq:kernel-orthogonal-complement} proves
\eqref{eq:angular-coherent-codimension}.
\end{proof}

\subsection{Stable zero sets}

\begin{corollary}\label{cor:stable-coherent-codimension}
If $\Lambda$ satisfies the equivalent conditions of
Theorem~\ref{thm:finite-modification-stability}, then the closed span of
$\{K_{\lambda,j}:\lambda\in\supp\Lambda,\ 0\le j<m_\lambda\}$ has
infinite codimension in $\F$.
\end{corollary}

\begin{proof}
By \eqref{eq:kernel-orthogonal-complement}, it suffices to prove that
$\mathcal I(\Lambda)$ is infinite dimensional. Choose
$w\notin\supp\Lambda$. For each integer $j\ge0$, stability
supplies $g_j\in\F$ with $\mathcal Z(g_j)=\Lambda+j\delta_w$.
For every $N$, the functions $g_0,\ldots,g_N$ are linearly independent:
in a nontrivial relation, the lowest nonzero order at $w$ cannot cancel.
They all belong to $\mathcal I(\Lambda)$, which is therefore infinite
dimensional.
\end{proof}

The positive configuration in
Theorem~\ref{thm:finite-angular-data-insufficient} therefore has an
associated kernel system of infinite codimension, while the negative
configuration gives a complete system. Both conclusions persist under
finite modifications, with multiplicities interpreted by derivative
kernels.

For the radial configurations in
Corollaries~\ref{cor:smooth-radial-displacements} and
\ref{thm:lattice-drift-criterion}, and the angular configurations in
Theorem~\ref{thm:smooth-angular-profile}, the proved uniqueness
alternatives give
\[
 \Lambda\text{ is a Fock zero set}
 \quad\Longleftrightarrow\quad
 \text{its zero-set integral converges}
 \quad\Longleftrightarrow\quad\mathcal I(\Lambda)\ne\{0\}.
\]
Equation~\eqref{eq:kernel-orthogonal-complement} then identifies
convergence with incompleteness. For general multisets, including the
full bounded-displacement class, failure of the zero-set criterion
alone does not imply completeness. The main theorem prescribes all
zeros; the kernel system tests vanishing with at least those orders.

Corollary~\ref{cor:intermediate-zero-status} supplies explicit examples
with no exact realization but infinite-dimensional vanishing space.

\appendix
\numberwithin{equation}{section}
\section{Auxiliary estimates for circular arrays}
\label{app:circular-estimates}

We give the proofs of Lemmas~\ref{lem:gaussian-circle-sampling},
\ref{lem:sampled-ray-polynomial}, \ref{lem:profile-angular-sampling}, and
\ref{lem:profile-radial-kernel}. The angular estimates use the notation of
Section~\ref{subsec:smooth-angular-profiles}; in particular,
$\varepsilon_k=(\log k)^\tau/(2k^2)$ for $k\ge K$ and
$\varepsilon_k=0$ for $k<K$, with $r_n=n+1/2$. Their error constants
may depend on $h,a,b,c,K,\tau$, but are independent of $n$, the
observation angle, and the polygon phases. The kernel estimates are
also uniform in the positive integer frequency.

\subsection{Sampling circular means}
\label{app:circle-sampling}

\begin{proof}[Proof of Lemma~\ref{lem:gaussian-circle-sampling}]
Assume $f\not\equiv0$ and set
$M_f(r)=(2\pi)^{-1}\int|f(re^{i\theta})|^2d\theta$.
Since $M_f(r)=\sum_{m\ge0}|a_m|^2r^{2m}$, H\"older's
inequality implies that $\log M_f(r)$ is convex as a function of
$\log r$. Thus, if $a=r_n\le r\le b=r_{n+1}$ and
$\log r=(1-t)\log a+t\log b$, then
\[
 M_f(r)e^{-r^2}
 \le e^{(1-t)a^2+tb^2-r^2}
       (M_f(a)e^{-a^2})^{1-t}(M_f(b)e^{-b^2})^t.
\]
The interpolation error for $x\mapsto e^{2x}$ gives
\[
 0\le(1-t)a^2+tb^2-r^2
 \le\tfrac12 b^2(\log b-\log a)^2
 \le\tfrac12(b/a)^2(b-a)^2\le C.
\]
Here $b-a\le3/2$ and $a\ge5/4$, so $b/a$ is bounded.
The arithmetic--geometric mean inequality then yields
\[
 \int_a^b M_f(r)e^{-r^2}r\,dr
 \le C\bigl(aM_f(a)e^{-a^2}+bM_f(b)e^{-b^2}\bigr).
\]
Summing over adjacent intervals counts each endpoint at most twice.
Also, $r_1\in[5/4,7/4]$ and $M_f$ is nondecreasing, whence
\[
 \int_0^{r_1}M_f(r)e^{-r^2}r\,dr
 \le\tfrac12r_1^2M_f(r_1)
 \le C r_1M_f(r_1)e^{-r_1^2}.
\]
Polar integration proves the upper estimate in
\eqref{eq:gaussian-circle-sampling}.

For the opposite estimate, fix $0<\epsilon<1/8$ and apply the
submean inequality to $f(w)e^{-\overline z w}$ on
$D(z,\epsilon)$. Since
$|z|^2-2\Re(\overline z w)=|w-z|^2-|w|^2$, this gives
\[
 |f(z)|^2e^{-|z|^2}
 \le\frac1{\pi\epsilon^2}
       \int_{D(z,\epsilon)}|f(w)|^2e^{-|w|^2}
                                  e^{|w-z|^2}\,dA(w)
 \le C_\epsilon
       \int_{D(z,\epsilon)}|f(w)|^2e^{-|w|^2}\,dA(w).
\]
Integrate over $|z|=r_n$ with respect to arclength and apply Tonelli.
For fixed $w$, the arc $|z-w|<\epsilon$ has length at most
$C\epsilon$ and is empty unless $||w|-r_n|<\epsilon$.
Thus
\[
 r_ne^{-r_n^2}M_f(r_n)
 \le C_\epsilon
       \int_{r_n-\epsilon<|w|<r_n+\epsilon}
                       |f(w)|^2e^{-|w|^2}\,dA(w).
\]
Since
$r_{n+1}-r_n\ge1/2$, these annuli are disjoint. Summing over $n$
completes the proof.
\end{proof}

\subsection{Polynomial bounds on a sampled ray}
\label{app:sampled-ray}

\begin{proof}[Proof of Lemma~\ref{lem:sampled-ray-polynomial}]
The assertion is immediate for $Q=0$. Rotation reduces the general
ray to the positive real axis. We first obtain a polynomial bound
on the entire positive axis from \eqref{eq:quotient-ray-data}.

For large $x>0$, choose $x_0=n+1/2$ with $x_0\le x<x_0+1$.
Set $m=\lceil D\log x\rceil$, where $D>4(C+1)$ is fixed, and
use the nodes $x_j=x_0+j$, $0\le j<m$. For large $x$ we have
$m<x/6$. Lagrange interpolation with its Cauchy remainder gives
\begin{equation}
 Q(x)=\sum_{j=0}^{m-1}Q(x_j)\ell_j(x)
 +\frac1{2\pi i}\int_{|\zeta-x|=x/2}
       \frac{Q(\zeta)}{\zeta-x}
       \prod_{j=0}^{m-1}\frac{x-x_j}{\zeta-x_j}\,d\zeta,
 \label{eq:sampled-ray-interpolation}
\end{equation}
where $\ell_j(x)=\prod_{i\ne j}(x-x_i)/(x_j-x_i)$.
The formula follows by residues at $x$ and at the nodes; at a node,
it holds by continuity. The estimates $|x-x_j|\le m$ and
$\prod_{i\ne j}|x_j-x_i|=j!(m-1-j)!$ imply
\[
 \sum_{j=0}^{m-1}|\ell_j(x)|
 \le\frac{2^{m-1}m^{m-1}}{(m-1)!}\le C_0^m
\]
for an absolute $C_0$. Here the elementary bound
$s!\ge(s/e)^s$ suffices. Since $x_j\asymp x$, the interpolation
sum is bounded by $C x^B C_0^m$, hence by a fixed power of $x$.

On the contour in \eqref{eq:sampled-ray-interpolation},
$|\zeta-x_j|\ge x/3$ and $|\zeta|\le2x$. The remainder is at most
\[
 \exp\bigl(C\log^2(2x)\bigr)\left(\frac{3m}{x}\right)^m.
\]
Its logarithm is at most
$C\log^2(2x)+m(\log(3m)-\log x)$, which tends to $-\infty$
by the choice of $D$. Thus $|Q(x)|\le C_2 x^{B_1}$ for $x\ge2$,
with a finite $B_1$.

We give the Phragm\'en--Lindel\"of step explicitly. Choose an integer
$M>B_1$ and $0<\sigma<1/2$. In the slit exterior domain
\[
 \Omega=\{z:|z|>2,\quad0<\arg z<2\pi\},
\]
take the branch $\arg(-z)\in(-\pi,\pi)$. For $\epsilon>0$ the
function
\[
 Q(z)z^{-M}\exp\{-\epsilon(-z)^\sigma\}
\]
is holomorphic on $\Omega$. Since
$\Re\bigl((-z)^\sigma\bigr)\ge |z|^\sigma\cos(\pi\sigma)>0$, its modulus
on a large outer circle tends to zero by \eqref{eq:quotient-log-squared-growth}.
On both banks of the slit and on the inner circle its modulus is
bounded by a constant independent of $\epsilon$. The maximum
principle on truncated domains, followed by exhaustion and then
$\epsilon\downarrow0$, gives $|Q(z)|\le C_3|z|^M$ for $|z|\ge2$.
Cauchy's coefficient estimates show that $Q$ is a polynomial.
\end{proof}

\subsection{Angular variation and sampling errors}
\label{app:angular-sampling}

\begin{proof}[Proof of Lemma~\ref{lem:profile-angular-sampling}]
Fix $r=r_n$ and set
$q_k=\min(r/R_k,R_k/r)$, $d_k=-\log q_k$.
Since $R_k-k=o(1)$ and $r_n$ stays at distance at least $1/4$ from
the occupied radii for large $n$,
\begin{equation}
 2kd_k\ge c_0>0,\qquad
 \sup_n\sum_{k\ge1}e^{-s k d_k}<\infty\quad(s>0).
 \label{eq:smooth-profile-overlap}
\end{equation}
Indeed, for $r/2\le k\le2r$, $kd_k\ge c(1+|k-n|)$;
outside this range, $d_k$ has a fixed positive lower bound after
omitting finitely many initial indices. The resulting geometric and
exponential sums prove \eqref{eq:smooth-profile-overlap}.

Let $L_q(x)=\log|1-qe^{ix}|$. For a point at angle $\psi$,
\[
 \log|1-re^{i\theta}/(R_ke^{i\psi})|
       -\log^+(r/R_k)=L_{q_k}(\psi-\theta).
\]
The centered sum for a regular polygon is
$\log|1-q_k^{2k}e^{2ki(\phi_k-\theta)}|$.
The sum of these terms is uniformly bounded by \eqref{eq:smooth-profile-overlap}.
The first variation caused by the angular motion is
\begin{equation}
 -\varepsilon_k\sum_{\ell=0}^{2k-1}
       h'(\theta_{k,\ell})L'_{q_k}(\theta_{k,\ell}-\theta).
 \label{eq:smooth-profile-first-variation}
\end{equation}

We first control the Taylor remainder. Differentiation gives
$|L_q''(x)|\le q/|1-qe^{ix}|^2$.
For $M$ angles separated by at least $c/M$, comparison with the
integral of $((1-q)^2+x^2)^{-1}$ gives
\[
 \sum_\ell\frac{q}{|1-qe^{ix_\ell}|^2}
 \le Cq\left(\frac{M}{1-q}+\frac1{(1-q)^2}\right).
\]
For $q\ge1/2$, this follows from
$|1-qe^{ix}|^2=(1-q)^2+4q\sin^2(x/2)$ and grouping by angular
distance from zero; for $q<1/2$, bound each summand by $4q$.
Estimate \eqref{eq:smooth-profile-overlap} implies $k(1-q_k)\ge c'>0$.
All intermediate angular maps have derivative at least $1/2$,
so the same separation estimate applies throughout Taylor's integral
remainder. The total shell remainder is at most
\begin{equation}
 C\varepsilon_k^2\frac{kq_k}{1-q_k}
                       \le C\frac{(\log k)^{2\tau}}{k^2}.
 \label{eq:profile-taylor-error}
\end{equation}
Its sum is bounded independently of $n$, the angle, and the phases.

To estimate the sampling error in \eqref{eq:smooth-profile-first-variation}, write
$h(x)=\sum_{j\in\Z}c_je^{ijx}$, with $c_0=0$.
Parseval and Cauchy--Schwarz give
\begin{equation}
 \sum_j|jc_j|<\infty,\qquad
 \sum_{|j|\ge M/2}|jc_j|\le C M^{-1/2}\|h''\|_{L^2}.
 \label{eq:smooth-profile-fourier-tail}
\end{equation}
For $q<1$,
\[
 L_q'(x)=-\frac i2\sum_{m\ne0}\operatorname{sgn}(m)
                                      q^{|m|}e^{imx}.
\]
Both Fourier series converge absolutely. Sampling at
$\phi+2\pi\ell/M$ retains precisely the frequencies divisible by
$M$. Thus the difference between the sampled and continuous means
of $h'(x)L_q'(x-\theta)$ is at most
\[
 \frac12\sum_j|jc_j|\sum_{p\ne0}q^{|pM-j|}.
\]
Terms with zero frequency in $L_q'$ have only increased this bound.
For $|j|<M/2$, the inner sum is at most
$2q^{M/2}/(1-q^M)$; for arbitrary $j$, its extension to all $p$
is at most $2/(1-q^M)$. Use \eqref{eq:smooth-profile-fourier-tail} and
$q^M\le e^{-c_0}$ with $M=2k$. The error in the unnormalized
sum is at most $CM(q^{M/2}+M^{-1/2})$.
Multiplying by $\varepsilon_k$ bounds the shell error by
\begin{equation}
 C\frac{(\log k)^\tau}{k}e^{-kd_k}
           +C\frac{(\log k)^\tau}{k^{3/2}},
 \label{eq:profile-sampling-error}
\end{equation}
The second term is summable over $k$. For the first, use
$\sup_{k\ge K}(\log k)^\tau/k<\infty$ and
\eqref{eq:smooth-profile-overlap}. Thus the total error is bounded
independently of the observation radius.

The continuous first variation is obtained by Fourier multiplication:
\[
 -\frac1{2\pi}\int_0^{2\pi}h'(x)L_q'(x-\theta)\,dx
 =\frac12\sum_{m\ne0}|m|c_mq^{|m|}e^{im\theta}.
\]
For each fixed $r_n$, the continuous-variation series converges
absolutely: for large $k$, $q_k=O(r_n/k)$ and
$\|L'_{q_k}\|_\infty=O(q_k)$, so its terms are
$O_{r_n}((\log k)^\tau/k^2)$. The regular-polygon bound and the two
summable errors prove the assertion.
\end{proof}

\subsection{Uniform radial kernel summation}
\label{app:radial-kernel}

\begin{proof}[Proof of Lemma~\ref{lem:profile-radial-kernel}]
The remainder must be uniform in the frequency because it will be summed
against the Fourier coefficients of $h$. A bound with an uncontrolled
frequency-dependent constant would not suffice. Since $K-1>e^2$, put
\[
 w(x)=\frac{(\log x)^\tau}{x},\qquad
 g(x)=w(x)e^{-m|\log(r/x)|}.
\]
On $[K-1,\infty)$ the function $g$ increases up to
$r$ and decreases afterwards. On either side it is a constant times
$x^{m-1}(\log x)^\tau$ or $x^{-m-1}(\log x)^\tau$, respectively.
Its variation is bounded by $2\sup_{x\ge K-1}w(x)$ uniformly in
$m,r$. Comparison on unit intervals consequently gives
\[
 \sum_{k\ge K}g(k)=\int_K^\infty g(x)\,dx+O(1).
\]
The intervals joining $k$ to $R_k$ lie in disjoint intervals
$[k-1/4,k+1/4]$, so
$\sum_{k\ge K}|g(R_k)-g(k)|\le\operatorname{Var} g=O(1)$.
Also,
\[
 \sum_{k\ge K}|w(R_k)-w(k)|
       \le C\sum_{k\ge K}\frac{(\log k)^{\tau+1}}{k^3}<\infty.
\]
These estimates replace $g(k)$ by
$w(k)e^{-m|\log(r/R_k)|}$ at bounded total cost.
For $r\ge2K$, the absolute value of the missing integral
over $(0,K)$ is at most
\[
 r^{-m}\int_0^K x^{m-1}|\log x|^\tau\,dx
 \le r^{-m}\frac{\Gamma(\tau+1)}{m^{\tau+1}}
       +\left(\frac Kr\right)^m\frac{(\log K)^\tau}{m},
\]
which is uniformly bounded. Finally, $x=re^v$ gives
\[
 \begin{split}
 \int_0^\infty\frac{(\log x)^\tau}{x}e^{-m|\log(r/x)|}\,dx
 &=\int_{-\infty}^{\infty}(\log r+v)^\tau e^{-m|v|}\,dv\\
 &=\frac{2(\log r)^\tau}{m}
       +\begin{cases}0,&\tau=1,\\4/m^3,&\tau=2.\end{cases}
 \end{split}
\]
This proves \eqref{eq:smooth-profile-radial-convolution}.

For the selected-shell estimate \eqref{eq:profile-selected-shells},
let $K_0$ be the first multiple of $N$ at least $K$. Comparison on the intervals $[k,k+N]$ bounds the error
between $N\sum_{N\mid k}g(k)$ and $\int_{K_0}^\infty g$ by
$N\operatorname{Var}(g)$. The displacement estimates above remain
valid on a subset of the indices, and the omitted integral over
$(0,K_0)$ has the same uniform bound. This proves
\eqref{eq:profile-selected-shells}.

\end{proof}

\Needspace{8\baselineskip}
\section*{Acknowledgments}

\noindent\textbf{Funding.}
Pan Ma was supported by the National Natural Science Foundation of China
(Grant Nos.~12171484 and 12571140), the Natural Science Foundation of Hunan
Province (Grant No.~2023JJ20056), the Science and Technology Innovation
Program of Hunan Province (Grant No.~2023RC3028), the Central South
University Innovation-Driven Research Programme (Grant No.~2023CXQD032),
and the Hunan Basic Science Research Center for Mathematical Analysis
(Grant No.~2024JC2002). Shengzhao Hou was supported by the National Natural
Science Foundation of China (Grant No.~12371133).

\medskip
\noindent\textbf{AI Statement.}
The authors used artificial-intelligence tools for language editing, LaTeX
formatting, and limited assistance with local mathematical reasoning. The
proof strategy and mathematical development are the authors' own; all
mathematical content was independently written and checked by the authors,
who take full responsibility for it.

\end{document}